\documentclass[11pt]{article}

\usepackage[margin=1in,a4paper]{geometry}
\usepackage{amssymb,amsmath}
\usepackage{amsthm} 
\usepackage{graphicx} 
\usepackage{tikz} 
\usepackage{xcolor}
\usepackage{microtype}
\usepackage{esint}
\usepackage[colorlinks,linkcolor=blue, urlcolor  = blue,citecolor=blue,pagebackref,hypertexnames=false, breaklinks]{hyperref}

\numberwithin{equation}{section}
\usepackage[shortlabels]{enumitem}
\setlist[enumerate,1]{wide, labelindent=0pt,label={\upshape(\roman*)}}
\usepackage{comment}

\usepackage{mathtools}
\usepackage{float}
\mathtoolsset{showonlyrefs}

\usepackage[nottoc,notlot,notlof]{tocbibind}

\usepackage{etoolbox}
\patchcmd{\thebibliography}
  {\settowidth}
  {\setlength{\itemsep}{0pt plus 0.1pt}\settowidth}
  {}{}
\apptocmd{\thebibliography}
  {\small}
  {}{}
 
    \def\XXint#1#2#3{{\setbox0=\hbox{$#1{#2#3}{\int}$}
    \vcenter{\hbox{$#2#3$}}\kern-.5\wd0}}

\newcommand{\vast}{\bBigg@{4}}
\newcommand{\Vast}{\bBigg@{5}}
\makeatother

\newcommand{\R}{\mathbb{R}}

\newcommand{\set}[1]{\left\{#1\right\}}

\newcommand{\cN}{\mathcal{N}}

\theoremstyle{plain}
\newtheorem{thm}{Theorem}[section]
\newtheorem{lem}[thm]{Lemma}
\newtheorem{cor}[thm]{Corollary}
\newtheorem{prop}[thm]{Proposition}

\newtheorem{example}[thm]{Example}
\newtheorem{assump}[thm]{Assumption}
\newtheorem{notation}[thm]{Notation}

\theoremstyle{definition}
\newtheorem{defn}[thm]{Definition}

\theoremstyle{remark}
\newtheorem{remark}[thm]{Remark}
\newcommand{\bremark}{\begin{remark} \em}
\newcommand{\eremark}{\end{remark} }
\usepackage{color}

\hbadness=\maxdimen

\begin{document}


\title{Parabolic Anderson model in bounded domains of recurrent metric measure spaces}

\author{F. Baudoin
\footnote{Partly funded by NSF DMS-2247117 and the Simons Foundation} 
, L. Chen\footnote{Partly funded by the Simons Foundation \#853249}, C.H. Huang, C. Ouyang\footnote{Partly funded by the Simons Foundation \#851792},  S. Tindel\footnote{Partly funded by NSF DMS-1952966 and NSF DMS-2153915} 
, J. Wang\footnote{Partly funded by NSF DMS-2246817}}

\maketitle

\begin{abstract}
A metric measure  space equipped with a Dirichlet form is called recurrent if its Hausdorff dimension is less than its walk dimension. In bounded domains of such spaces we study the parabolic Anderson models
\[
\partial_{t} u(t,x) = \Delta u(t,x) + \beta u(t,x) \, \dot{W}_\alpha(t,x)
\]
where the noise $W_\alpha$ is white in time and colored in space when $\alpha >0$ while for $\alpha=0$ it is also white in space. Both Dirichlet and Neumann boundary conditions are considered. Besides proving existence and uniqueness in the It\^o sense we also get precise $L^p$ estimates for the moments and   intermittency properties of the solution as a consequence. Our study reveals new exponents which are intrinsically associated to the geometry of the underlying space and the results for instance apply in metric graphs or fractals like the Sierpi\'nski gasket for which we prove scaling invariance properties of the models. 

\end{abstract}

\tableofcontents

\section{Introduction}

This article originates in two different threads which have been very active in the last decades. Namely on the one hand models from   mathematical physics involving fractals have drawn interest in the recent literature, see for instance \cite{balsam2023density} and references therein. On the other hand many fascinating connections between parabolic Anderson models \cite{Da, Dalang-Quer, Hu15, HuNualart} and directed polymer models \cite{CY,CSZ, La10, MRT, RT} on spaces such as $\mathbb{R}^d$, $\mathbb{Z}^d$, or graphs have been established. Growth and weak/strong disorder properties for polymers have now been obtained either in discrete or continuous setting. Our aim in this paper is to bring those two worlds together, initiating a theory of parabolic Anderson and related models on general metric measure spaces equipped with a Dirichlet form. 

There has been a recent spur in the study of stochastic heat equations and polymer measures on geometric structures. In a discrete setting one should probably go back to \cite{CSZ}, which focuses on disorder properties for polymers defined on infinite graphs. In continuous spaces, properties of the stochastic heat equation in sub-Riemannian spaces such as Heisenberg groups have been examined in \cite{BCHOTW,BOTW}. One should also mention the recent preprint \cite{hairer2023regularity} dealing with regularity structures in non-flat settings. The current contribution fits in this global picture, our aim being to quantify the interaction between geometric and random environments for a rich class of models.

As mentioned above, we are focusing here on general metric measure spaces. In that class of spaces we are particularly interested in post critically finite (p.c.f.) fractals, such as the Sierpi\'nski gasket and other nested fractals, which are limits of graphs. For an overview of the wealth of probabilistic and  analytic techniques on p.c.f. fractals, we refer to the classical works \cite{Barlow, Kigami1993, Lindstrom}. Post critically finite fractals are compelling examples of geometric structures for two main reasons. First they lead to a wide variety of exponents for diffusions, heat equations and polymers. We will see some of these features below. Next they are limits of graphs, which means that corresponding limits of discrete objects can also be investigated for stochastic PDEs and polymers. We plan to pursue this line  of investigation in future works.

Let us specify briefly our geometric setting (see Section \ref{sec: Prelim} for more details). Namely, in the present paper, we study fractional parabolic Anderson models in bounded domains of recurrent metric measure spaces.  A recurrent metric measure space $(X,d,\mu)$ is a metric measure space  that supports a Brownian motion whose heat kernel $p_t(x,y)$ has sub-Gaussian estimates of the form
\begin{align}\label{subgauss intro}
p_{t}(x,y) \simeq  t^{-d_h/d_w}\exp\biggl(-C\Bigl(\frac{d(x,y)^{d_w}}{t}\Bigr)^{\frac{1}{d_w-1}}\biggr) \, ,
\end{align}
with $d_h<d_w$. The parameter $d_h$ is the Hausdorff dimension of the metric space $(X,d)$ and the parameter $d_w \ge 2$ is the so-called walk dimension. Intuitively $d_h$ controls the volume growth rate of balls of small radii (i.e., $\mu (B(x,r)) \simeq r^{d_h}$),  while $d_w$ controls the escape rate of such balls by the Brownian motion (i.e., $\mathbf{E}_x(d(x,B_t))\simeq t^{1/d_w}$). For instance, in the Euclidean space $(\mathbb R^n,d_{\mathrm{euc.}}, dx)$ one has $d_h=n$ and $d_w=2$. In the Sierpi\'nski gasket fractal, one has $d_h=\frac{\ln 3}{\ln 2}$ and $d_w=\frac{\ln 5}{\ln 2}$, see \cite{Barlow}. The condition $d_h <d_w$ that we assume throughout this paper implies that the Brownian motion is a strongly recurrent Markov process which admits (H\"older continuous) local times, see \cite{Barlow} and references therein.

We now turn to a brief description of our stochastic model of interest (see Section \ref{sec: Dirichlet PAM}  for a more rigorous definition). In a  bounded open domain $U$ of a recurrent metric space $(X,d,\mu)$, we study the parabolic Anderson model
\begin{align}\label{PAM intro}
\partial_{t} u(t,x) = \Delta u(t,x) + \beta u(t,x) \, \dot{W}_\alpha(t,x),
\end{align}
where $\Delta$ is the Laplacian on $U$ and where the family of Gaussian noises $W_\alpha$ is white in time and colored in space when $\alpha >0$, while for $\alpha=0$ it is also white in space. The parameter $\beta >0$ controls the intensity of the noise in the equation. Solutions are understood in the It\^o-Skorohod sense. Both Dirichlet and Neumann boundary conditions on the boundary of $U$ for the Laplacian $\Delta$ are considered. In order to describe our main findings, let us delve deeper into the structure of our noise. Namely the spatial covariance of our noise is similar to the one introduced in \cite{BaudoinChen, BaudoinLacaux}, as well as the covariance on Heisenberg groups handled in \cite{BCHOTW,BOTW}. Roughly speaking, the covariance function of $\dot{W}_\alpha$ is given by
$$\mathbf{E}\left[\dot{W}_\alpha(t,x)\dot{W}_\alpha(s,y)\right]=\delta_0(t-s) \, G_{2\alpha}(x,y),$$
where $\delta_0$ is the Dirac delta at $0\in\mathbb{R}$ and $G_{2\alpha}$ is the kernel of the operator $(-\Delta)^{-2\alpha}$ (notice that the definition of the Laplace operator, its powers $(-\Delta)^{-2\alpha}$ as well as estimates on $G_{2\alpha}$ will be spelled out in Section \ref{sec: Prelim}). 
 In particular, one can distinguish three regimes:
\begin{enumerate}
\setlength\itemsep{.05in}

\item 
For $0<\alpha<d_h/2d_w$, the noise is distributional with spatial covariance diagonal singularity of order $ \frac{1}{d(x,y)^{d_h-2\alpha d_w}}$.

\item 
For $\alpha=d_h/2d_w$, we have a log-correlated field which is distributional with spatial covariance diagonal singularity of order $|\ln d(x,y)|$.

\item 
For $\alpha >d_h/2d_w$, the noise $W_\alpha([0,t),x)$ can be defined pointwise and is H\"{o}lder continuous.
\end{enumerate}
Observe that in the sequel we will mostly focus on cases where $W_\alpha$ is distributional, corresponding to (i) and (ii) above.

With this brief description of $W_\alpha$ in hand, our first objective is to find conditions on $\alpha$ so that equation \eqref{PAM intro} admits a unique solution. Our main theorem in this direction can be informally summarized as follows (the reader is referred to Theorem \ref{existence PAM diri} in the Dirichlet case and Theorem~\ref{existence unique neuma}  in the Neumann case for a more rigorous version including the regularity assumptions on the domain under consideration).

\begin{thm}\label{th: intro existence-uniqueness}
On any bounded open domain $U$ of $(X,d,\mu)$, with either Neumann or Dirichlet boundary conditions, equation \eqref{PAM intro} admits a unique solution in the It\^{o} sense for any value $\alpha\in[0,\infty)$. In addition, for $\alpha>d_h/2d_w$,  the solution admits a Feynman-Kac representation.
\end{thm}

\noindent
We point out and stress that $\alpha$ can be $0$ in Theorem \ref{th: intro existence-uniqueness}, which means that the space-time white noise is admissible for our equation. This is a consequence of our assumption $d_h<d_w$, which implies that the Brownian motion on  $X$ is strongly recurrent. Now observe that this condition  can be rewritten in terms of  the so-called spectral dimension $d_s$, which measures a  growth rate  for the eigenvalues of the Laplace operator $\Delta$ on $(X,d,\mu)$.  Specifically the relation $d_h<d_w$ can also be stated as $d_s<2$,  where we recall that $d_s=\frac{2d_h}{d_w}$. As one can see from  Theorem~\ref{th: intro existence-uniqueness}, the space-time white noise is a proper driving noise whenever $d_s<2$. However, for $d_s>2$ we expect that some conditions on $\alpha$ will have to be imposed in order to solve \eqref{PAM intro}. In fact based on our previous considerations in \cite[Section 1]{BOTW}, we expect the condition on $\alpha$ to be
$$\alpha>\frac{d_s}{4}-\frac{1}{2}=\frac{1}{2}\left(\frac{d_h}{d_w}-1\right).$$
This type of hypothesis will be analyzed in a future work.

Once existence and uniqueness are proved, we turn our attention to estimating Lyapounov exponents associated with the moments of the solution. In this context it is worth separating Dirichlet from Neumann boundary conditions, which lead to quite different asymptotic behaviors in~\eqref{PAM intro}.

\noindent
\textbf{(a)} \emph{Dirichlet boundary conditions.}
Whenever equation~\eqref{PAM intro} is specified with Dirichlet boundary conditions and bounded initial conditions, our main upper bound for the large time asymptotics of the $p$-moments of order $p\ge 2$ can be read as
\begin{align}\label{upper bound diri intro}
\limsup_{t \to +\infty} \frac{\ln \mathbf E(u(t,x)^p)}{t} \le \frac{p}{2}   \Theta_\alpha (\beta \sqrt{p-1})-p\lambda_1 , \quad x \in U,
\end{align}
where $\lambda_1>0$ is the first Dirichlet eigenvalue of $U$. 
Furthermore, recalling that the spectral dimension $d_{s}$ is defined by  $d_s=\frac{2d_h}{d_w}$, the quantity $\Theta_\alpha$ in~\eqref{upper bound diri intro} is a continuous increasing function such that:
\begin{enumerate}
\item If $0\le \alpha<\frac{d_{s}}{4}$, then $\Theta_\alpha (\beta) \sim_{\beta \to 0} C\beta^2 $ and  $\Theta_\alpha (\beta) \sim_{\beta \to +\infty} C\beta^{\frac{4}{2(1+2\alpha)-d_s}} $.
\item If $\alpha=\frac{d_{s}}{4}$, then  $\Theta_\alpha (\beta) \sim_{\beta \to 0} C\beta^2 $ and $\Theta_\alpha (\beta) \sim_{\beta \to +\infty} C\beta^2 (\ln \beta)^2 $.

\item If $\alpha>\frac{d_{s}}{4}$, then $\Theta_\alpha (\beta)=C\beta^2$.
\end{enumerate}
While a variety of conclusions can be drawn from~\eqref{upper bound diri intro}, let us focus on weak/strong disorder type interpretations according to the parameter $\beta$ in~\eqref{PAM intro}. Indeed, the upper bound~\eqref{upper bound diri intro} illustrates competing effects. The dissipative effect of the Dirichlet semigroup controlled by $\lambda_1$ induces an exponential decay for $u(t,x)$, while the noise whose intensity is controlled by $\beta$ induces an exponential growth. This competing effect is further illustrated by the following lower bound proved in Theorem~\ref{thm:p-moment-lower-Diri}:
\begin{align}\label{lower bound diri intro}
\liminf_{t \to +\infty} \frac{\ln \mathbf E(u(t,x)^p)}{t} \ge C  \beta^2  p(p-1)  -p\lambda_1 (A) , \quad x \in A,
\end{align}
where $A\subset U$ is an open set that does not intersect the boundary of $U$ and $\lambda_1(A)$ is the first Dirichlet eigenvalue for $-\Delta_{A}$. In particular, for any such fixed domain $A$ there exist  critical parameters $\beta_{c,1}>0$ and $\beta_{c,2}>0$  depending on $A$, $p$ and $\alpha$ such that uniformly on $A$ we have 
\begin{equation*}
\limsup_{t \to +\infty} \frac{\ln \mathbf E(u(t,x)^p)}{t} <0  \text{ if } \beta < \beta_{c,1},
\quad\text{and}\quad
\liminf_{t \to +\infty} \frac{\ln \mathbf E(u(t,x)^p)}{t} 
>0  \text{ if } \beta > \beta_{c,2}.
\end{equation*}
For the parabolic Anderson model on a bounded interval, this interesting competing effect between the driving noise $W_\alpha$ and the Dirichlet semigroup was already pointed out in \cite{FoondunNualart}. We also note that the rates we obtain recover results from \cite{MR3359595} for bounded intervals in the white noise case, see also \cite{MR2480553}. 

Remarkably, the lower bound \eqref{lower bound diri intro} can be substantially improved in the range $0<\alpha < \frac{d_s}{4}$ for the large $\beta$ regime. More precisely, we obtain the following lower bounds whose exponents in (i) below match the upper bound \eqref{upper bound diri intro}:
\begin{enumerate}
\item
For $0<\alpha <\frac{d_{s}}{4}$, we have
$\liminf\limits_{\beta \to +\infty}\liminf\limits_{t \to +\infty} 
\beta^{-\frac{4}{2(1+2\alpha)-d_s}} 
\frac{\ln \mathbf E(u(t,x)^p)}{t} \ge C p(p-1).
$

\item
For $\alpha =\frac{d_{s}}{4}$, we get
$
\liminf\limits_{\beta \to +\infty}\liminf\limits_{t \to +\infty} \,
\, (\beta^{2} \ln \beta)^{-1}
\frac{\ln \mathbf E(u(t,x)^p)}{t} \ge C p(p-1).
$
\end{enumerate}
Those bounds quantify both the strong disorder effect and the intermittency phenomenon for large values of $\beta$.

\noindent
\textbf{(b)} \emph{Neumann boundary conditions.}
For the solution $u(t,x)$ associated with Neumann boundary conditions, the heat semigroup is conservative and there is therefore nothing to counteract the excitability effect induced by the noise. In that case we find that
\[
\limsup_{t \to +\infty} \frac{\ln \mathbf E(u(t,x)^p)}{t} \le \frac{p}{2}   \widehat\Theta_\alpha (\beta \sqrt{p-1}),
\]
where $\widehat\Theta_\alpha$ is a continuous increasing function  which is comparable to $\Theta_\alpha$ when $\beta \to 0$ or $\beta \to +\infty$. The corresponding lower bounds on $p$-th order moments are given by
\[
\liminf_{t \to +\infty} \frac{\ln \mathbf E(u(t,x)^p)}{t} \ge C  \beta^2  p(p-1),
\]
which seems to indicate a strong disorder regime as soon as $\beta>0$ and is due again to the fact that we are considering spectral dimensions $d_{s}<2$. Observe that we have   improvements on those lower bounds as before for the large $\beta$'s regime in the range $0 < \alpha \le \frac{d_s}{4}$.

Concerning the exponents in both the Dirichlet and Neumann boundary case, let us note that in the context of a locally compact Hausdorff Abelian group $G$ a near-dichotomy for the (expected square) energy of the solution $u$ was discovered in \cite{KK2015} when the driving noise is white in both the time and space variable. Roughly speaking, the dichotomy says that, for large $\beta$ (equivalently for low temperatures),  the energy of the solution behaves like $\exp\{C \beta^2\}$ when $G$ is discrete (e.g. $G=\mathbf{Z}^d$ or $(\mathbf{Z}/n\mathbf{Z})^d $) and is bounded from below by $ \exp\{C \beta^4\}$ when $G$ is connected (e.g. $G=\mathbf{R}$ or $[0,1])$. Our results suggest that in our setting in the range $0 \le \alpha < \frac{d_s}{4}$ the energy behaves for large $\beta$  like $\exp\left\{C \beta^{\frac{4}{2(1+2\alpha)-d_s}} \right\}$ which demonstrates a rich variety of exponents for the class of models we are considering and illustrates their dependency in terms of both the noise regularity and the geometric environment.

Chaos expansions, Feynman-Kac representations and martingale type arguments are widely used tools in the study of parabolic Anderson models in Euclidean spaces. We will use all three of them in the current contribution, according to the values of the parameters in~\eqref{PAM intro} and the kind of bound we are dealing with. Namely
our strategy to prove the upper bound \eqref{upper bound diri intro} is to make use of the solution chaos expansion, together with the equivalence of $L^p$ norms on chaos of fixed orders. 
For the lower bound \eqref{lower bound diri intro}, we need to divide our analysis in two cases. If $\alpha >0$ we use a Feynman-Kac representation of the moments, while for $\alpha =0$ in the case $p=2$ we use a rather direct martingale type argument based on the mild formulation of~\eqref{PAM intro}. Those basic ingredients are then combined with thorough analytic estimates involving the heat kernel~\eqref{subgauss intro}. Notice that when $\alpha=0$, the Feynman-Kac representation for moments would involve the construction and study of mutual intersection local times for independent Brownian motions in recurrent metric measure spaces. This is a task which is interesting in its own right, and deferred to a later work.

The paper is organized as follows. In Section \ref{sec: Prelim}, we first present the general framework of a recurrent metric measure space $(X,d,\mu)$, equipped with a Dirichlet form with generator $\Delta$ whose heat kernel satisfies the sub-Gaussian estimates \eqref{subgauss intro}. Given a bounded open set $U \subset X$, one can restrict the Dirichlet form to $U$ using either Dirichlet or Neumann boundary condition and under mild regularity assumptions on the domain  the restricted Dirichlet form has a heat kernel with sub-Gaussian type estimates, see Lierl \cite{Lierl} for the Dirichlet boundary case and Murugan \cite{murugan2023heat} for the Neumann boundary case. This allows us to define the fractional operator $(-\Delta_U)^{-\alpha}$ where $\Delta_U$ is the Laplacian on $U$ with either Dirichlet or Neumann boundary condition. Estimates for the kernel of this operator are then obtained. In that setting, the driving noise to equation \eqref{PAM intro} is then defined as
\[
W_\alpha (t,x)=(-\Delta_U)^{-\alpha}W_0(t,x),
\]
where $W_0(t,x)$ is a space-time white noise on $U$ and the expression $(-\Delta_U)^{-\alpha}W_0(t,x)$ is understood in a weak $L^2$ sense which is made precise in the text. In order to show the scope of our results, we then present in detail three examples of spaces to which our results apply: bounded intervals of the real line, metric graphs and the Sierpi\'nski gasket. In the case of Sierpi\'nski gasket we exhibit in Section \ref{sec: property intermittency} a scale invariance property of the parabolic Anderson models. In Section \ref{sec: Dirichlet PAM}, we study equation \eqref{PAM intro} with Dirichlet boundary condition and prove the above mentioned results. Section \ref{sec: Neumann PAM} deals with the Neumann boundary case.

\section{Preliminaries}\label{sec: Prelim}

This section is devoted to recall basic facts about metric spaces equipped with a Dirichlet form. We also prove some crucial heat kernel estimates and  give some details about powers of the Laplace operator and associated fractional Riesz kernels in this context, with different types of boundary conditions. Then we shall describe three examples of relevant applications of our abstract setting.

\subsection{Recurrent metric measure spaces with sub-Gaussian heat kernel estimates}

Our considerations will be based on a locally compact complete geodesic metric measure space  $(X,d,\mu)$, where $\mu$ is a Radon measure supported on $X$. We will denote $\mathrm{diam}(X) \in (0,+\infty]$ the diameter of $X$, which can possibly be infinite.

The space $(X,d,\mu)$ is equipped with a Dirichlet form. That is  a densely defined closed symmetric form $(\mathcal{E},\mathcal{F}:=\rm{dom}\,\mathcal{E})$  on $L^2(X,\mu)$. A function $v\colon X\to\mathbb{R}$ is called a normal contraction of the function $u$ if for $\mu$-almost every $x,y \in X$,
\[
| v(x)-v(y)| \le |u(x) -u(y)|, \qquad\text{and}\qquad |v(x)| \le |u(x)|.
\]
The form $(\mathcal{E},\mathcal{F})$ is called a Dirichlet form if it is Markovian, that is, it has the property that if $u \in \mathcal{F}$ and $v$ is a normal contraction of $u$ then $v \in \mathcal{F}$ and $\mathcal{E}(v,v) \le \mathcal{E} (u,u)$.
Some basic properties of Dirichlet forms are collected in~\cite[Theorem 1.4.2]{FOT}. In particular, we note that $\mathcal{F} \cap L^\infty(X,\mu)$ is an algebra and $\mathcal F$ is a Hilbert space equipped with the $\mathcal{E}_1$-norm defined as
\begin{equation}\label{e-E1}
\|f\|_{\mathcal{E}_1}:=\left( \| f \|_{L^2(X,\mu)}^2 + \mathcal{E}(f,f) \right)^{1/2}.
\end{equation} 

In the sequel it will be convenient to express Dirichlet forms as integrals. This is usually true in the so-called \emph{regular} setting that we proceed to recall now. Denoting by $C_c(X)$ the space of continuous functions with compact support in $X$, we recall that a core for $(\mathcal{E},\mathcal{F})$ is a subset $\mathcal{C}\subseteq C_c(X) \cap \mathcal{F}$ which is dense in $C_c(X)$ in the supremum norm and dense in $\mathcal{F}$ in the $\mathcal{E}_1$-norm.

\begin{defn} 
A Dirichlet form $(\mathcal{E},\mathcal{F})$ on $L^2(X,\mu)$ is called regular if it admits a core.
\end{defn}

\noindent
In addition, given the regular Dirichlet form $(\mathcal{E},\mathcal{F})$, for every $u,v\in \mathcal F\cap L^{\infty}(X)$ one can define the energy measure $\Gamma (u,v)$ through the formula
 \begin{align}\label{def: Gamma}
\int_X\phi\, d\Gamma(u,v)=\frac{1}{2}[\mathcal{E}(\phi u,v)+\mathcal{E}(\phi v,u)-\mathcal{E}(\phi, uv)], \quad\text{for\ all}\  \phi\in \mathcal F \cap C_c(X).
\end{align}
The latter definition of $\Gamma(u,v)$ can be extended to any $u,v\in \mathcal F$ by truncation. According to Beurling-Deny's theorem~\cite{FOT}, for regular Dirichlet forms the quantity $\Gamma(u,v)$ is a signed Radon measure for any $u,v\in \mathcal F$ and $\mathcal E(u,v)$ can be decomposed as
\[
\mathcal E(u,v)=\int_X d\Gamma(u,v).
\]
For $u \in \mathcal{F}$, $\Gamma(u,u)$ is called  the energy measure of $u$.
Another useful property of Dirichlet form, which will be invoked in this paper, is locality. We recall the definition below for further use.

\begin{defn}\label{def:strong-local}(\cite[p.6]{FOT}) 
A Dirichlet form $(\mathcal{E},\mathcal{F})$ is called  \emph{strongly local} if for any $u,v\in\mathcal{F}$ with compact supports such that $u$ is constant in a 
neighborhood of the support of $v$, it holds that $\mathcal{E}(u,v)=0$. 
\end{defn}

Related to a Dirichlet form $(\mathcal{E},\mathcal{F})$, there is a notion of generator and semigroup. The generator $\Delta$ is defined using an integration by part formula, namely $\Delta$ is defined on a dense domain $\mathrm {dom}(\Delta)$ such that for every $f\in \mathcal F$ and $g \in \mathrm {dom}(\Delta)$,
\[
\mathcal E (f,g)=-\int_X f \Delta g \, d\mu.
\]
We  denote $\{P_{t}\}_{t \ge 0}$ the self-adjoint  semigroup on $L^2(X,\mu)$ associated with the Dirichlet space $(X,\mu,\mathcal{E},\mathcal{F})$. Using spectral theory this semigroup can be defined as $P_t=e^{t\Delta}$;  see~\cite[Section 1.4]{FOT}) for further details. 

With those general notions in hand, let us summarize the assumptions on our metric space and its Dirichlet form that will prevail throughout the paper.  Note that they are not independent (some may be derived from combinations of the others) and the list is chosen for comprehension rather than minimality.
\begin{assump}\label{hyp:dirichlet-form}
Let $(X,d,\mu)$ be the locally compact complete geodesic metric measure space introduced above, together with its Dirichlet form $(\mathcal{E},\mathcal{F})$. In the sequel we assume the following:
\begin{enumerate}
\item The measure $\mu$ is Ahlfors $d_h$-regular, i.e., there exist $c_{1},c_{2},d_h\in(0,\infty)$ such that  for any $x\in X$ and any $r\in\bigl(0,\mathrm{diam}(X) \bigr)$ we have
$
c_{1}r^{d_h}\leq\mu\bigl(B(x,r)\bigr)\leq c_{2}r^{d_h} .
$
\item $\mathcal{E}$ is a regular and strongly local Dirichlet form.

\item The semigroup $\{ P_t \}_{t\ge 0}$ is stochastically complete, i.e., $P_t1=1$ for every $t \ge 0$.

\item The semigroup $\{P_t\}_{t\ge 0}$ has a continuous heat kernel $p_t(x,y)$ satisfying sub-Gaussian estimates, 
that is, for some $c_{3},c_{4}, c_5, c_6 \in(0,\infty)$ and $d_w\in [2,+\infty)$,
 \begin{equation}\label{eq:subGauss-upper}
 c_{5}t^{-d_h/d_w}\exp\biggl(-c_{6}\Bigl(\frac{d(x,y)^{d_w}}{t}\Bigr)^{\frac{1}{d_w-1}}\biggr) 
 \le p_{t}(x,y)\leq c_{3}t^{-d_h/d_w}\exp\biggl(-c_{4}\Bigl(\frac{d(x,y)^{d_w}}{t}\Bigr)^{\frac{1}{d_w-1}}\biggr)
 \end{equation}
 for $\mu\!\times\!\mu$-a.e.\ $(x,y)\in X\times X$ and each $t\in\bigl(0,\mathrm{diam}(X)^{1/d_w})$. 
 \item We have $d_h<d_w$.
 
 \end{enumerate}
 \end{assump}
\noindent
Observe that in Assumption \ref{hyp:dirichlet-form}, the parameter $d_h$ is the Hausdorff dimension and the parameter $d_w$ is called the walk dimension. Also note that when $d_w=2$, which is possible in our setting, we simply have Gaussian estimates.

A fundamental theorem in the theory of Dirichlet forms, see \cite[Theorem 7.2.1 and Theorem 7.2.2]{FOT}, is that in such a setting one can construct a Brownian motion $(B_t)_{t \ge 0}$ on $X$ as the continuous Hunt process with semigroup $P_t$. Otherwise stated, for every bounded and Borel function $f$ and $x \in X$ we have
\[
\mathbf{E}_x( f(B_t) )=P_t f(x).
\]
Notice that the condition $d_h<d_w$ in Assumption \ref{hyp:dirichlet-form} is known as the recurrence condition.  It implies in particular (see Corollary 3.29 in \cite{Barlow}) that the Brownian motion $(B_t)_{t \ge 0}$ is a strongly recurrent Markov process: for every $x,y \in X$,
\[
\mathbf{P}_x( \exists \, t \ge 0 \text{ such that } B_t =y)=1.
\]
The recurrence condition also implies that functions in the domain $\mathcal{F}$ have a continuous version. We note that our assumptions put us in the framework of Barlow's fractional spaces, which were introduced and studied in \cite{Barlow}.

\subsection{Dirichlet boundary conditions}\label{Dirichlet boundary}

This section specializes our setting to Dirichlet boundary conditions on a bounded set $U$. We first derive some bounds on the related semigroup.  Then we will define fractional powers of the Laplacian on $U$.

\subsubsection{Heat semigroup with Dirichlet boundary conditions}\label{sec-semi-g}
One of our goals in this note is to deal with stochastic heat equations on bounded domains of our metric space $X$. Let us thus consider a non-empty bounded open set $U$ whose closure $\overline{U}$ is compact. We will localize the notions of Dirichlet form and related objects to functions defined on $U$.
The domain of our localized Dirichlet form will be called $\mathcal{F}(U)$. It is defined as
\begin{align}\label{def: F(U)}
\mathcal{F}(U)=\left\{ u\in L^2(U,\mu); \text{there\ exists}\ f\in\mathcal{F}\ \text{such\ that} \ f|_{U}=u\ \text{a.e.}\right\}.
\end{align}
Notice that for every $u\in \mathcal{F}(U)$ the localized energy measure in \eqref{def: Gamma} is simply given by $d\Gamma (u,u)=d\Gamma (f,f)$.

In order to introduce Dirichlet boundary conditions, we now restrict our Dirichlet form. That is, define
\[
\mathcal{F}_c (U)
=
\left\{ u \in \mathcal{F}(U);\text{ the essential support of } u \text{ is compact in } U \right\}.
\]
Then the domain $\mathcal{F}^D(U)$ will be the closure of $\mathcal{F}_c (U)$ for the norm $\mathcal{E}_{1}$ defined by~\eqref{e-E1}.  For $u,v \in \mathcal{F}^D(U)$ we set
\begin{align}\label{def: D Dirichlet form}
\mathcal{E}^D_U(u,v):=\mathcal{E}(u,v).
\end{align}
Notice that $(\mathcal{E}^D_U,\mathcal{F}^D(U)) $ is a Dirichlet form and the functions in $\mathcal{F}^D(U)$ satisfy the Dirichlet boundary condition. The generator $\Delta_{D,U}$ of the Dirichlet form $(\mathcal{E}^D_U,\mathcal{F}^D(U)) $ is called the Dirichlet Laplacian on $U$.  We refer to \cite{MR2807275} for more details on this construction, which we stress does not require any \textit{smoothness} assumption on the boundary of $U$. In particular, see~\cite[Theorem 4.4.2]{FOT},  the semigroup $P^{D,U}_t$ with generator $\Delta_{D,U}$ can be represented for $x \in U$ as
\begin{align}\label{def: Dirichlet semigroup}
P^{D,U}_t f(x)=\mathbf{E}_x ( f(B_t) 1_{t<T_U}) \, ,
\end{align}
where $B_t$ is a Brownian motion on $X$ and $T_U$ is the hitting time of the complement of $U$ by $B_t$.

The construction  of $P^{D,U}$ in \eqref{def: Dirichlet semigroup} is quite general. However, proper bounds on the heat kernel $p^{D,U}$ corresponding to $P^{D,U}$ are available when $U$ is an inner uniform domain. We now define this type of domain. To begin with, we consider a connected set $U$. Then $U$ admits an inner metric:
\begin{align}\label{def: inner metric}
d_U(x,y)=\inf \left\{ \mathrm{length}(\gamma); \, 
\gamma: [0,1] \to U \text{ rectifiable, }
 \gamma(0)=x \text{ and } \gamma(1)=y \right\}.
\end{align}

The following straightforward lemma will be useful.

\begin{lem}\label{Lem:Equa-BU-B}
  Let $U \subset X$ be an open set. Assume that $x \in U$ and $r >0$ are such that $B(x,r) \subset U$. Then we have  $B_U(x,r)=B(x,r)$, where $B_U(x,r)$ denotes the ball with center $x$ and radius $r$ for the distance $d_U$.  
\end{lem}
\begin{proof}
Since $d \le d_U$ we have  always have $B_U(x,r) \subset B(x,r)$. Now, assume that $B(x,r) \subset U$ and let $y \in B(x,r)$. Since the space $(X,d)$ is assumed to be geodesic, there exists a geodesic $\gamma$ connecting $x$ to $y$. This geodesic is included in $B(x,r) \subset U$ so that $d_U(x,y) \le \mathrm{length}( \gamma)=d(x,y) \le r$.
\end{proof}
We denote by $\widetilde{U}$ the completion of an open set $U$ for $d_U$. With those preliminary notions in hand, we can now define inner uniform domains.

\begin{defn}\label{Def: inner uni domain}
Consider an open, connected, bounded set $U$ whose closure is compact. Let $c \in (0,1)$, $C>1$. Let $\gamma: [\alpha,\beta] \to U$ be a rectifiable curve. We say that $\gamma$ is a $(c,C)$ inner  uniform curve in $U$ if the following two conditions are satisfied:
\begin{enumerate}
\item For all $t \in [\alpha,\beta]$, $d_U( \gamma(t), \widetilde{U} \setminus U) \ge c \min \{ d_U(\gamma(\alpha),\gamma(t)), d_U(\gamma(t),\gamma(\beta)) \}$,
\item $\mathrm{length}(\gamma) \le C d_U (\gamma(\alpha),\gamma(\beta))$.
\end{enumerate}
Then the domain $U$ is called a $(c,C)$ inner uniform domain if any two points in $U$ can be joined by a $(c,C)$ inner uniform curve in $U$. The domain $U$ is  simply called inner uniform if there exist $c \in (0,1),C>1$ such that $U$ is $(c,C)$ inner uniform. 
\end{defn}

For example, in  Euclidean spaces, any bounded convex domain is inner uniform.  Any bounded domain with piecewise smooth boundary with a finite number of singularities and non-zero interior angle at each of the singularities is also inner uniform, see Example 3.5 in \cite{MR3170207}. The class of inner uniform domains is actually very wide and also includes sets with very rough boundaries like  the Koch snowflake, see Section 6.5 in \cite{MR2807275}. We refer to \cite{MR2807275,Lierl, MR3170207} for examples and further  discussions about inner uniform domains. In our setting of recurrent metric measure spaces, further examples of (inner) uniform domains will be addressed in sections \ref{interval example}, \ref{metric graph section} and \ref{section gasket}.

Furthermore, for inner uniform domains we have the following upper and lower bounds on the heat kernel that follow from \cite[Theorem 3.3]{Lierl}.

\begin{thm}\label{estimate_heat_diri}
Let $U$ be an inner uniform domain as introduced in Definition \ref{Def: inner uni domain}. Then  the Dirichlet form $(\mathcal{E}^D_U,\mathcal{F}^D(U)) $ given by \eqref{def: D Dirichlet form} is a strongly local regular Dirichlet form which admits a continuous symmetric heat kernel $p_t^{D,U}(x,y)$. Moreover, if $\Phi_1$ is the first eigenfunction of $-\Delta_{D,U}$ and $\lambda_1>0$ is the corresponding principal eigenvalue, then the Dirichlet heat kernel $p_t^{D,U}(x,y)$ satisfies the sub-Gaussian upper estimate
\begin{equation}\label{eq:subGauss-upperD}
p_{t}^{D,U}(x,y)\leq\frac{ c_{1}e^{-\lambda_1 t}}{\min (1,t^{\frac{d_h}{d_w}})}\!\exp\biggl(\!-c_{2}\Bigl(\frac{d(x,y)^{d_w}}{t}\Bigr)^{\frac{1}{d_w-1}}\!\biggr),
\end{equation}
as well as the sub-Gaussian lower estimate
\begin{equation}\label{eq:subGauss-lowerD}
p_{t}^{D,U}(x,y)\geq \frac{c_{3}\Phi_1(x) \Phi_1(y)e^{-\lambda_1t}} {\min (1,t^{\frac{d_h}{d_w}})}\!\exp\biggl(\!-c_{4}\Bigl(\frac{d_U(x,y)^{d_w}}{t}\Bigr)^{\frac{1}{d_w-1}}\!\biggr) ,
\end{equation}
for every \ $(x,y)\in U \times U$ and $t\in\bigl(0,+\infty)$, where we recall that $d_U$ is defined by \eqref{def: inner metric}.

\end{thm}

\begin{remark}
The proof of the upper bound \eqref{eq:subGauss-upperD} in small times comes from the upper bound in \eqref{eq:subGauss-upper}, combined with the estimate
 $$p_t^{D,U}(x,y) \le p_t(x,y).$$ 
 Notice that this last relation is a straightforward consequence of \eqref{def: Dirichlet semigroup}. In large times, relation~\eqref{eq:subGauss-upperD} follows from spectral theory, see \eqref{eq:HKexpansion-diri} below. The lower bound  \eqref{eq:subGauss-lowerD} in small times follows from the stronger result
 \begin{align}\label{lower bound diri 2}
 p_{t}^{D,U}(x,y)\geq \frac{c_{3}\Phi_1(x) \Phi_1(y)} {\mu (B_U (x,t^{1/d_w})) }\!\exp\biggl(\!-c_{4}\Bigl(\frac{d_U(x,y)^{d_w}}{t}\Bigr)^{\frac{1}{d_w-1}}\!\biggr)
 \end{align}
which is stated in \cite[Theorem 3.3]{Lierl}. Here $B_U$ denotes the inner  ball, i.e., the ball in $\tilde U$ for the distance $d_U$. Since $d \le d_U $ we have $B_U (x,t^{1/d_w}) \subset B (x,t^{1/d_w})$ so that owing to our Assumption~\ref{hyp:dirichlet-form}, the Ahlfors regularity of $\mu$ implies that $\mu(B_U (x,t^{1/d_w}))\le c t^{d_h/d_w}$.
In large times, the lower bound  \eqref{eq:subGauss-lowerD} follows from spectral theory, see \eqref{eq:HKexpansion-diri} below.
\end{remark}

\begin{remark}\label{rmk: positivity of first ef}
It follows from \cite[Lemma 7.13]{MR3170207} that the first eigenfunction $\Phi_1$ is positive on $U$. Since it is moreover continuous,  if $U$ is an inner uniform domain and $A\subset U$ is such that $\overline A \subset U$, then
\begin{equation}\label{eq: positivity of first ef}
\hat{m}_A:=\inf\left\{\Phi_1(y); y\in A\right\}>0.
\end{equation}
This information will feature prominently in the application of \eqref{eq:subGauss-lowerD}.
\end{remark}

For the remainder of this subsection, we assume that $U$ is inner uniform so that the above heat kernel estimates \eqref{eq:subGauss-upperD} and \eqref{eq:subGauss-lowerD} are valid.  The following corollaries of  Theorem~\ref{estimate_heat_diri} will be useful later.

\begin{cor}\label{lem:hk-lower-spectral}
Let $U\subset X$ be an inner uniform domain given as in Definition \ref{Def: inner uni domain}. Let $x\in U$ and $r>0$ such that $B(x,2r) \subset U$. There exists a constant $c>0$ (depending on $r$ and $x$) such that for every $t\ge 0$ and $y\in B(x,r/2)$, the semigroup $P^{D,U}$ represented by \eqref{def: Dirichlet semigroup} verifies
\begin{align}\label{hk-lower-spectral}
\left(P_t^{D,U}1_{B(x,r)}\right)\!(y)\ge c e^{-\lambda_1 t}.
\end{align}
\end{cor}

\begin{proof}
We first notice that since we assume $B(x,2r) \subset U$, we have from Lemma \ref{Lem:Equa-BU-B}, $B(x,r)=B_U(x,r)$ and $B(x,r/2)=B_U(x,r/2) $; this will be used without further notice in the proof below.
 From the heat kernel lower bound in Theorem \ref{estimate_heat_diri} and the strict positivity and continuity of $\Phi_1$, one has for $y\in B_U(x,r/2)$, 
 \begin{align}\label{Cor 2.7 MS 1}
 \left(P_t^{D,U}1_{B_U(x,r)}\right)\!(y) \ge \int_{B_U(x,r)}\frac{c_{3}  e^{-\lambda_1t}} {\min (1,t^{\frac{d_h}{d_w}})}\!\exp\biggl(\!-c_{4}\Bigl(\frac{d_U(z,y)^{d_w}}{t}\Bigr)^{\frac{1}{d_w-1}}\!\biggr) d\mu(z).
 \end{align} 
 Assume first $t^{1/d_w} \ge r/2 $. Then it is readily checked that for every  $y\in B_U(x,r/2)$ and $z\in B_U(x,r)$ we have
 \begin{equation}\label{Cor 2.7 MS 2}
 \left(\frac{d_U(z,y)^{d_w}}{t}\right)^{\frac{1}{d_w-1}}\leq c_5.
\end{equation}
Plugging this inequality into \eqref{Cor 2.7 MS 1} we obtain that for every $y\in B_U(x,r/2)$,
  \begin{align}\label{Cor 2.7 MS 3}
 \left(P_t^{D,U}1_{B_U(x,r)}\right)(y) \ge c_6  e^{-\lambda_1t} .
 \end{align} 
 Assume now $t^{1/d_w} \le r/2$. We have then from \eqref{lower bound diri 2} for every  $y\in B_U(x,r/2)$
 \begin{equation}\label{Cor 2.7 MS 4}
 \begin{split}
 \left(P_t^{D,U}1_{B_U(x,r)}\right)(y) \ge& 
 \int_{B_U(x,r)}\frac{c_5} {\mu (B_U (y,t^{1/d_w}))}\!\exp\biggl(\!-c_{4}\Bigl(\frac{d_U(z,y)^{d_w}}{t}\Bigr)^{\frac{1}{d_w-1}}\!\biggr) d\mu(z) \\
  \ge& 
  \int_{B_U(y,r/2)}\frac{c_5} {\mu (B_U (y,t^{1/d_w}))}\!\exp\biggl(\!-c_{4}\Bigl(\frac{d_U(z,y)^{d_w}}{t}\Bigr)^{\frac{1}{d_w-1}}\!\biggr) d\mu(z) \\
  \ge& 
  \int_{B_U(y,t^{1/d_w})}\frac{c_5} {\mu (B_U (y,t^{1/d_w}))}\!\exp\biggl(\!-c_{4}\Bigl(\frac{d_U(z,y)^{d_w}}{t}\Bigr)^{\frac{1}{d_w-1}}\!\biggr) d\mu(z).
\end{split}
 \end{equation}
 Now if $z\in B_U(y,t^{1/d_w})$, one can see that the exponential term in the right hand-side of \eqref{Cor 2.7 MS 4} is lower bounded by a constant.
  Reporting these considerations into \eqref{Cor 2.7 MS 4}, we obtain
 \begin{align}\label{Cor 2.7 MS 5}
  \left(P_t^{D,U}1_{B_U(x,r)}\right)(y) \ge c_8.
  \end{align}
  Gathering \eqref{Cor 2.7 MS 3} and \eqref{Cor 2.7 MS 5}, the proof of our claim \eqref{hk-lower-spectral} is achieved.
 \end{proof}

The general bound \eqref{hk-lower-spectral} can be generalized to a wider class of domains satisfying a cone type condition.  We state this property in a separate corollary.

\begin{cor}\label{lem:hk-lower-spectral no cusp}
Under the same conditions as in Corollary \ref{lem:hk-lower-spectral}, let $A \subset U$ be an open set such that $\overline A \subset U$. Assume that there exist $\delta >0$, $r_0>0$ such that for all $y \in A$,
    \begin{align}\label{no cusp}
    \mu (B_U(y,r) \cap A) \ge \delta r^{d_h},\quad \text{for\ all}\ \ 0 \le r \le r_0.
    \end{align}
 Then, there exist  constants $c_1,c_2>0$ (depending on $\delta$ and $r_0$) such that for every $t\ge 0$ and $y\in A$, 
\begin{align}\label{hk-lower-spectral no cusp}
\left(P_t^{D,U}1_{A}\right)(y)\ge c_1 e^{-\lambda_1 t}.
\end{align}
and
\begin{align}\label{jku}
\int_A p_t^{D,U} (y,z)^2 d\mu (z) \ge c_2 \frac{e^{-2\lambda_1 t}}{t^{d_h/d_w}}.
\end{align}
\end{cor}

\begin{proof}
 Similarly to \eqref{Cor 2.7 MS 1}, from the heat kernel lower bound in Theorem \ref{estimate_heat_diri} and Remark~\ref{rmk: positivity of first ef}, one has for $y\in A$, 
 \begin{align*}
 (P_t^{D,U}1_{A})(y) \ge \int_{A}\frac{c_{1}  e^{-\lambda_1t}} {\min (1,t^{\frac{d_h}{d_w}})}\!\exp\biggl(\!-c_{2}\Bigl(\frac{d_U(z,y)^{d_w}}{t}\Bigr)^{\frac{1}{d_w-1}}\!\biggr) d\mu(z).
 \end{align*} 
 Let us assume first that $0< t \le r_0^{d_w}$, where $r_0$ is the constant in relation \eqref{no cusp}. In that case we first get the trivial lower bound 
 \begin{align*}
 (P_t^{D,U}1_{A})(y) & \ge \int_{B_U(y,t^{1/d_w}) \cap A}\frac{c_{1}  e^{-\lambda_1t}} {\min (1,t^{\frac{d_h}{d_w}})}\!\exp\biggl(\!-c_{2}\Bigl(\frac{d_U(z,y)^{d_w}}{t}\Bigr)^{\frac{1}{d_w-1}}\!\biggr) d\mu(z) .
  \end{align*}
  Next on $B_U(y,t^{1/d_w})$ it is readily checked that $d_U(z,y)^{d_w}\le t$. Hence one can bound the sub-Gaussian term above by $e^{-c_2}$. Therefore we obtain 
   \begin{align*}
 (P_t^{D,U}1_{A})(y) \ge c_3 \frac{\mu (B_U(y,t^{1/d_w}) \cap A)}{\min (1,t^{\frac{d_h}{d_w}})} \, e^{-\lambda_1t}
  \ge c_4 e^{-\lambda_1t},
 \end{align*}
 where we have invoked condition \eqref{no cusp} for the last inequality. 
 On the other hand, if $t\ge r_0^{d_w}$ then our general bound \eqref{Cor 2.7 MS 3} already implies $(P_t^{D,U}1_{A})(y) \ge c_5 e^{-\lambda_1t}$  so that the conclusion \eqref{hk-lower-spectral no cusp} easily follows. The proof of the estimate \eqref{jku} is similar and left to the reader.
\end{proof}
\begin{remark}
    The condition \eqref{no cusp} can be checked in many examples and essentially means that the set $A$ has a nice boundary (no \textit{cusps}). For instance, if $X=\mathbb{R}$ and $U=(0,1)$ then any subinterval $A=(a,b) \subsetneq U$ satisfies \eqref{no cusp} with $\delta=1/2$. In the Sierpi\'nski gasket example which is introduced with notations in Section \ref{section gasket} one can see that if $U=K\setminus V_0$, then  any finite union of sets of the form $\mathfrak{f}_{i_1} \circ \cdots \circ \mathfrak{f}_{i_n} (U)$ satisfies \eqref{no cusp}.
\end{remark}

We now state a matching upper bound for \eqref{jku}, valid on inner uniform domains. 

\begin{cor}\label{lem:L2bound}
Consider an inner uniform domain $U\subset X$, as given in Definition \ref{Def: inner uni domain}. Let $p^{D,U}$ be  the kernel introduced in Theorem \ref{estimate_heat_diri}.  Then there exists a constant $C>0$ such that for all $x\in U$,
\begin{align}\label{L2bound}
\int_U p_t^{D,U}(x,y)^2 d\mu(y)\le
\left\{ \begin{aligned}
& C t^{-\frac{d_h}{d_w}}, &0<t<1,\\
& C e^{-2\lambda_1 t}, &t\ge 1.
\end{aligned}\right.
\end{align}
Therefore, if  $0<\gamma < \lambda_1$, there exists a constant $C>0$ such that for every $t>0$, $x \in U$,
\[
\int_U p_t^{D,U}(x,y)^2 d\mu(y) \le C  \frac{e^{-2\gamma t}}{t^\frac{d_h}{d_w}}.
\]
\end{cor}
\begin{proof}
As mentioned in Theorem \ref{estimate_heat_diri}, the kernel $p_t^{D,U}$ is symmetric. Hence the semigroup property yields
\begin{align}\label{eq-chap-kol}
   \int_U p_t^{D,U}(x,y)^2 d\mu(y) =p_{2t}^{D,U}(x,x). 
\end{align}
The result therefore easily follows from the upper bound \eqref{eq:subGauss-upperD}.
\end{proof}

The upper bound in Corollary \ref{estimate_heat_diri} implies an exponential decay for the semigroup $P^{D,U}_t$, which we label for further use.

\begin{cor}\label{bound_pt_diri_ini}
Let $P_t^{D,U}$ be the semigroup defined by \eqref{def: Dirichlet semigroup} for an inner uniform domain $U$. Then there exists a constant $C>0$ such that for every $t \ge 0$, and $u_0 \in L^\infty(U,\mu)$,
\[
\left\| P^{D,U}_t u_0 \right\|_\infty \le C  e^{-\lambda_1 t} \|u_0\|_\infty.
\]
\end{cor}
\begin{proof}
We divide again the proof into a small time and large time estimate. That is for $0\leq t\leq1$, one simply resorts to \eqref{def: Dirichlet semigroup} in order to write
\begin{align}\label{cor2.11 MS1}
\left\| P^{D,U}_t u_0 \right\|_\infty \le \|u_0\|_\infty.
\end{align}
For $t >1$, we first apply the Cauchy-Schwarz inequality to get
\begin{align}\label{cor2.11 MS2}
    \left|P^{D,U}_t u_0(x)\right|^2 \le  \|u_0 \|_{L^2(U,\mu)}^2 \int_U p^{D,U}_t(x,y)^2 d\mu(y).
\end{align}
Next in order to estimate the right hand-side of \eqref{cor2.11 MS2} we recall (see Definition \ref{Def: inner uni domain}) that $U$ is bounded. Hence $\mu(U)<\infty$ and $\|u_0\|_{L^2(U,\mu)}\leq C\|u_0\|_\infty$. Furthermore the integral $\int_U p^{D,U}(x,y)^2\mu(dy)$ is upper bounded thanks to \eqref{L2bound} for $t\geq1$. Gathering this information into \eqref{cor2.11 MS2} we end up with
\begin{align}\label{cor2.11 MS3}
\left|P_t^{D,U} u_0(x)\right|^2\leq C e^{-2\lambda_1 t} \|u_0\|_\infty^2.
\end{align}
We now conclude our proof by putting together the estimate \eqref{cor2.11 MS1} for $0\leq t<1$ and the bound \eqref{cor2.11 MS3} for $t\geq1$.
\end{proof}

\subsubsection{Fractional Laplacian} As mentioned in the introduction, the correlation of our noise $W_\alpha$ is generated by the operator $(-\Delta)^{-2\alpha}$. We now recall how to define this operator through spectral theory. We first introduce a set of notation for the spectrum of $\Delta_{D,U}$.
\begin{notation}\label{notation: spectrum}
Let $U$ be an inner uniform domain and recall that $\Delta_{D,U}$ is the generator of the Dirichlet form defined by \eqref{def: D Dirichlet form}. The operator $-\Delta_{D,U}$ admits a discrete spectrum 
$\{\lambda_j; j\geq1\}, \text{with}\ \lambda_1>0\  \text{and}\ \lambda_j\leq \lambda_{j+1},$
together with an orthonormal basis $\{\Phi_j; j\geq1\}$ of eigenfunctions such that $\Phi_j\in \mathcal{D}(\Delta_{D,U})$ and $\Delta_{D,U} \Phi_j =-\lambda_j \Phi_j.$
\end{notation}
Let us also label a lemma which will ease our estimates on the covariance function.
\begin{lem}
Under the condition of Notation \ref{notation: spectrum}, the Dirichlet heat kernel $p_t^{D,U}$ introduced in Theorem \ref{estimate_heat_diri} admits the following convergent spectral expansion for all $t>0$ and $x,y\in U$:
\begin{equation}\label{eq:HKexpansion-diri}
p_t^{D,U}(x,y)=\sum_{j=1}^{+\infty} e^{-\lambda_j t} \Phi_j(x) \Phi_j(y).
\end{equation}
Moreover, for every $\alpha>d_h/d_w$ we have
\begin{align}\label{eq:series_eigen-Diri}
\sum_{j=1}^{+\infty} \frac{1}{\lambda_j^{\alpha}} <+\infty.
\end{align}

\end{lem}

\begin{proof}
The decomposition \eqref{eq:HKexpansion-diri} is classical for heat kernels on bounded spaces, see Theorem 3.44 in \cite{Barlow}. The upper bound~\eqref{eq:series_eigen-Diri} stems from diagonal estimates of $p_t^{D,U}$ together with Tauberian type theorems (see \cite{MR0793651}).
\end{proof}


%
%

The preliminaries above allow for a proper definition of negative power of the Laplace operator.
\begin{defn}[Fractional Laplacians]\label{FLaplacian}
We still work under the setting of Notation~\ref{notation: spectrum}, and let $\alpha \ge 0$. For $f \in L^2(U,\mu)$, the fractional Laplacian $(-\Delta_{D,U})^{-\alpha}$ acting on $f$ is defined as
\begin{align}\label{frac Laplacian spectrum}
(-\Delta_{D,U})^{-\alpha} f(x)  =\sum_{j=1}^{+\infty} \frac{1}{\lambda_j^\alpha} \left( \int_U  \Phi_j(y) f(y) d\mu(y)\right)\Phi_j(x).
\end{align}
\end{defn}

\begin{remark}\label{Rmk:frac Lapl D}From the definition \eqref{frac Laplacian spectrum}, it is clear that $(-\Delta_{D,U})^{-\alpha}: L^2(U,\mu) \to L^2(U,\mu)$ is a bounded operator. More precisely, for a generic $f\in L^2(U,\mu)$ one has 
\[
\|(-\Delta_{D,U})^{-\alpha}f\|_{ L^2(U,\mu)}^2=\sum_{j=1}^\infty\lambda_j^{-2\alpha}\langle f,\Phi_j\rangle_{L^2(U,\mu)}^2\le 
\lambda_1^{-2\alpha}\sum_{j=1}^\infty\langle f,\Phi_j\rangle_{L^2(U,\mu)}^2 =\lambda_1^{-2\alpha}\|f\|_{L^2(U,\mu)}^2 .
\]
Therefore we get the upper bound  $\|(-\Delta_{D,U})^{-\alpha}\|_{ L^2(U,\mu) \to L^2(U,\mu)}\le \lambda_1^{-\alpha}$. 
\end{remark}

The existence of a kernel for $(-\Delta_{D,U})^{-\alpha}$ is a consequence of functional calculus considerations. We now make this notion more precise following~\cite{BaudoinChen}. To begin with, we introduce the space of Schwartz functions on $U$.
\begin{defn}\label{def:a1}
Let $\{\Phi_n; n\geq 1\}$ be the set of eigenfunctions introduced in Notation~\ref{notation: spectrum}. Then we define the Schwartz space of test functions on $U$ as
\[
\mathcal{S}_{D}(U)= \left\{ f \in L^2(U,\mu);  
\text{ for all } k \in \mathbb N,  \lim_{n \to +\infty} n^k \left| \int_U  \Phi_n(y) f(y) d\mu(y) \right| =0 \right\}.
\]
\end{defn}

The following lemma, borrowed from \cite[Lemma 2.9]{BaudoinChen} and \cite[Corollary 2.11]{BaudoinChen}, identifies a kernel for the operator $(-\Delta_{D,U})^{-\alpha}$.

\begin{lem}\label{Lem: kernel frac Lapl D} Under the setting of Notation \ref{notation: spectrum}, let $(-\Delta_{D,U})^{-\alpha}$ be the operator given in Definition \ref{FLaplacian}. Then $(-\Delta_{D,U})^{-\alpha}$ admits a kernel $G^{D,U}_\alpha$ which can be expressed as
\begin{align}\label{kernel for frac Laplacian}
G_\alpha^{D,U}(x,y)=\frac{1}{\Gamma(\alpha)}\int_0^{+\infty} t^{\alpha-1}p_t^{D,U}(x,y) dt, \quad x,y \in U, x\neq y.
\end{align}
where $p_t^{D,U}$ is the heat kernel in Theorem \ref{estimate_heat_diri}.  More specifically, for $f\in \mathcal{S}_{D}(U)$ and $y\in U$ we have
\begin{align}\label{frac Laplacian in kernel}
  (-\Delta_{D,U})^{-\alpha} f (y)= \int_U f(x)  G^{D,U}_{\alpha}(x,y) d\mu(x) .
\end{align}
 For $f,g \in \mathcal{S}_{D}(U)$ we can also write
\begin{align}\label{frac Hilbert norm in kernel}
 \int_{U} (-\Delta_{D,U})^{-\alpha} f (-\Delta_{D,U})^{-\alpha} g d\mu= \int_U \int_U f(x) g(y) G^{D,U}_{2\alpha}(x,y) d\mu(x) d\mu(y).
\end{align}
\end{lem}

\begin{remark}
Notice that owing to the sub-Gaussian upper estimate \eqref{eq:subGauss-upperD}, the integral~\eqref{kernel for frac Laplacian} defining $G_\alpha^{D,U}$ is always convergent for $\alpha>0$ and $x\not=y$.
\end{remark}
  
In order to get optimal conditions for the existence-uniqueness of the stochastic PDE~\eqref{PAM intro}, it is important to have precise bounds on the kernel $G_\alpha^{D,U}$. The proposition below takes care of this step.

\begin{prop}\label{estimate G}
For an inner uniform domain $U$, let $G^{D,U}_\alpha$ be the kernel defined by \eqref{kernel for frac Laplacian}. Then the following bounds hold true:
	\begin{enumerate}
		\item If  $0<\alpha <d_h/d_w$, there exist  constants $c,C >0$ such that for every  $x,y \in U$, with $x \neq y$,
		\begin{align}\label{two sided bounds G 1}
		c  \frac{ \Phi_1(x)\Phi_1(y)}{d_U(x,y)^{d_h-\alpha d_w}} \le  G_\alpha^{D,U}(x,y)  \le  \frac{C}{d(x,y)^{d_h-\alpha d_w}}.
		\end{align}
		\item If $\alpha = d_h/d_w$, there exist  constants $c,C >0$ such that for every  $x,y \in U$, with $x \neq y$
		\begin{align}\label{two sided bounds G 2}
		c \,\Phi_1(x) \Phi_1(y) \max(1, | \ln d_U(x,y) |) \le G_\alpha^{D,U}(x,y)  \le C \max(1,  | \ln d(x,y) |).
		\end{align}
		\item If  $\alpha > d_h/d_w$, there exist  constants $c,C >0$ such that for every  $x,y \in U$, 
		\begin{align}\label{two sided bounds G 3}
		c \, \Phi_1(x) \Phi_1(y) \le  G_\alpha^{D,U}(x,y) \le  C  .
		\end{align}
	\end{enumerate}

\end{prop}

\begin{proof}
The proof of the upper estimates is similar to the proof of \cite[Proposition 2.6]{BaudoinLacaux} for the Neumann fractional Riesz kernel and we omit the details here. We will thus handle the lower bounds in \eqref{two sided bounds G 1}, \eqref{two sided bounds G 2} and \eqref{two sided bounds G 3}, focusing on the three regimes for $\alpha$.

\noindent
{\textit{Case $0<\alpha<d_h/d_w$}}.  We plug relation \eqref{eq:subGauss-lowerD} into the definition \eqref{kernel for frac Laplacian} of $G^{D,U}_\alpha$. Lower bounding integrals over $(0,\infty)$ by integrals over $(0,1)$ and writing $\min(1, t^{d_h/d_w})=t^{d_h/d_w}$ for $t\in(0,1)$, we get
\begin{equation}\label{eq:RK-lower-less1}
G_\alpha^{D,U}(x,y)
\ge c_3\,\Phi_1(x)\Phi_1(y)\int_0^1 t^{\alpha-1-\frac{d_h}{d_w}} \!\exp\biggl(\!-c_{4}\Bigl(\frac{d_U(x,y)^{d_w}}{t}\Bigr)^{\frac{1}{d_w-1}}\!\biggr) dt.
\end{equation}
We then perform the elementary change of variable $r=t/d_U(x,y)^{d_w}$ in the integral above. We end up with
\begin{equation}\label{eq:RK-lower-less}
G_\alpha^{D,U}(x,y)\ge c  \frac{ \Phi_1(x)\Phi_1(y)}{d_U(x,y)^{d_h-\alpha d_w}},
\end{equation}
which is the lower bound in \eqref{two sided bounds G 1}.

\noindent
{\textit{Case $\alpha> d_h/d_w$}}. We proceed as in the previous case. However, we now bound integrals on $(0,\infty)$ by integrals on $(1,\infty)$ and write $\min(1,t^{d_h/d_w})=1$ for $t\in(1,\infty)$. We easily get
\begin{equation}\label{eq:RK-lower-ge}
G_\alpha^{D,U}(x,y)
\ge c_3\,\Phi_1(x)\Phi_1(y)\int_1^{\infty} t^{\alpha-1} e^{-\lambda_1t} dt
\ge c\, \Phi_1(x)\Phi_1(y).
\end{equation}
That is the lower bound in  \eqref{two sided bounds G 3} is achieved.

\noindent
{\textit{Case $\alpha= d_h/d_w$}}. As in the previous steps, we plug the lower bound \eqref{eq:subGauss-lowerD} into \eqref{kernel for frac Laplacian}. Restricting the integral to $(0,1)$ and using the relation $\alpha=d_h/d_w$ we discover that
\begin{align*}
    G_\alpha^{D,U}(x,y)\ge c_3 \,\Phi_1(x)\Phi_1(y)\int_0^1 t^{-1} \!\exp\biggl(\!-c_{4}\Bigl(\frac{d_U(x,y)^{d_w}}{t}\Bigr)^{\frac{1}{d_w-1}}\!\biggr) dt.
\end{align*}
Next the change of variable $u=t/d_U(x,y)^{d_w}$ yields
\begin{align}\label{Prop2.19 MS1}
    G_\alpha^{D,U}(x,y)\geq c_3 \,\Phi_1(x)\Phi_1(y)\int_0^{1/d_U(x,y)^{d_w}}  \!\exp\biggl(\!-c_{4}\Bigl(\frac{1}{u}\Bigr)^{\frac{1}{d_w-1}}\!\biggr) \frac{du}{u}.
 \end{align}
 In order to avoid the singularity at $0$ in the integral above, let us introduce the diameter $\widetilde R$ of $U$ with respect to the distance $d_U$. Then from \eqref{Prop2.19 MS1} one sees that
\begin{align}
    G_\alpha^{D,U}(x,y)
     &\ge c_3\,\Phi_1(x)\Phi_1(y) \int_{1/(2\widetilde R)^{d_w}}^{1/d_U(x,y)^{d_w}}\!\exp\biggl(\!-c_{4}\Bigl(\frac{1}{u}\Bigr)^{\frac{1}{d_w-1}}\!\biggr) \frac{du}{u}
    \\ &\ge c\,\Phi_1(x)\Phi_1(y)
    \int_{1/(2\widetilde R)^{d_w}}^{1/d_U(x,y)^{d_w}} \frac{du}{u}
    \\ &= c\,d_w\,\Phi_1(x)\Phi_1(y)(\ln (2\widetilde R)-\ln d_U(x,y)).\label{eq:RK-lower-ge2}
\end{align}
When $x,  y$ are close to eah other, the dominant term in the right hand-side of \eqref{eq:RK-lower-ge2} is $c\, d_w\,\Phi_1(x)\Phi_1(y)|\ln d_U(x,y)|$, while for $x, y$ further apart a trivial lower bound is of the form $c\,\Phi_1(x)\Phi_1(y)$. We thus conclude the proof of \eqref{two sided bounds G 2}.
\end{proof}

The following estimate is the main technical lemma we will be using to estimate solutions of the stochastic heat equation.

\begin{lem}\label{estimate frac pt}
  Let $\alpha \ge 0 $, and consider an inner uniform domain $U$. Recall that the operator $(-\Delta_{D,U})^{-\alpha}$ in Definition \ref{FLaplacian} admits the kernel $G_\alpha^{D,U}$ given by \eqref{kernel for frac Laplacian}. Then there exists a constant $C$ such that for every $t>0$
  \begin{align}\label{est frac pt}
 \sup_{x \in U} \|(-\Delta_{D,U})^{-\alpha}p_t^{D,U}(x,\cdot)\|_{L^2(U,\mu)}  \le
 \begin{cases}
C e^{-\lambda_1 t} (1+t^{\alpha-\frac{d_h}{2d_w}}), \quad 0 \le \alpha < \frac{d_h}{2d_w}, \\
C e^{- \lambda_1 t} |\ln \left(\min (1/2,t)\right)|, \quad \alpha=\frac{d_h}{2d_w}, \\
C e^{-\lambda_1  t}, \quad \alpha >\frac{d_h}{2d_w}.
 \end{cases}
  \end{align}
\end{lem}

\begin{proof}
  The case $\alpha=0$ follows from Corollary \ref{lem:L2bound}, so we assume $\alpha >0$. From relation~\eqref{kernel for frac Laplacian} and the semigroup property for $p_t$, one has for every $t>0$ 
\begin{align}
(-\Delta_{D,U})^{-\alpha}p_t^{D,U}(x,y)&=\int_U G_\alpha^{D,U}(y,z)p_t^{D,U}(x,z) d\mu(z) \\
 &=\frac{1}{\Gamma(\alpha)}\int_0^{+\infty}\int_U s^{\alpha -1}p_s^{D,U}(y,z)p_t^{D,U}(x,z) d\mu(z)ds \\
 &=\frac{1}{\Gamma(\alpha)}\int_0^{+\infty} s^{\alpha-1}p_{s+t}^{D,U}(x,y) ds. \label{Lem2.20 MS 1}
\end{align}
Applying Minkowski's inequality to the right hand-side of \eqref{Lem2.20 MS 1}, we get
\begin{align}\label{eq-est-del-pt}
   \|(-\Delta_{D,U})^{-\alpha}p_t^{D,U}(x,\cdot)\|_{L^2(U,\mu)}  \leq \frac{1}{\Gamma(\alpha)}\int_0^{+\infty} s^{\alpha-1}\| p_{s+t}^{D,U}(x,\cdot)\|_{L^2(U,\mu)} ds. 
\end{align}
We now invoke our identity \eqref{eq-chap-kol} and the bound \eqref{eq:subGauss-upperD} in order to estimate the right hand-side above. This easily yields
\begin{equation}\label{eq:DeltaHK}
\begin{split}
\|(-\Delta_{D,U})^{-\alpha}p_t^{D,U}(x,\cdot)\|_{L^2(U,\mu)}  
 & \le \frac{1}{\Gamma(\alpha)}\int_0^{+\infty} s^{\alpha-1} p_{2(s+t)}^{D,U}(x,x)^{1/2} ds \\
 &\le C e^{-\lambda_1 t} \int_0^{+\infty} s^{\alpha-1} e^{-\lambda_1 s}  \max \left(1,  \frac{1 }{(t+s)^{d_h/2d_w}} \right) ds.
\end{split}
\end{equation}
Inequality \eqref{est frac pt} is then achieved thanks to some elementary calculus considerations.
\end{proof}
To conclude the section we give the definition of the Sobolev space with reproducing kernel $G_{2\alpha}^{D,U}$.

\begin{defn}[Sobolev Space $\mathcal{W}_D^{-\alpha}(U)$] \label{sobo_def_D}
For $\alpha>0$, we define the Sobolev space $\mathcal{W}_D^{-\alpha}(U)$ as the completion of $ \mathcal{S}_{D}(U)$ with respect to the inner product
\begin{align}\label{Sobo_norm_D}
\left\langle  f ,g \right\rangle_{\mathcal{W}_D^{-\alpha}(U)}=\int_U \int_U f(x) g(y) G^{D,U}_{2\alpha}(x,y) d\mu(x) d\mu(y)= \int_{U} \left[(-\Delta_{D,U})^{-\alpha} f\right]\left[ (-\Delta_{D,U})^{-\alpha} g\right] d\mu,
\end{align}
where the kernel $G^{D,U}_{2\alpha}$ is defined by \eqref{kernel for frac Laplacian}.
\end{defn}
 
\subsection{Neumann boundary conditions}\label{Neumann boundary}
In this section we turn to the Neumann boundary case. For further details about the Neumann boundary conditions in the general setting of Dirichlet spaces, we refer to \cite{Chen-Dirichlet} and \cite{{MR2807275}}. We will mimic the structure of Section \ref{Dirichlet boundary}, by first studying a notion of Neumann semigroup and then defining fractional powers of the Laplacian.

\subsubsection{Heat semigroup with Neumann boundary conditions}
As in Section \ref{Dirichlet boundary}, we consider a non-empty bounded open set $U\subset X$, whose closure $\overline{U}$ is compact. The domain $\mathcal{F}(U)$ is still defined by \eqref{def: F(U)}. However, instead of restricting the domain to $\mathcal{F}_c(U)$ like in the Dirichlet case, we directly denote
\begin{align}\label{D form Neumann}
\mathcal{E}^N_U(u,u)=\int_U d\Gamma(u,u),\quad\text{for}\ \  u\in\mathcal{F}(U),
\end{align}
where $\Gamma$ is the energy measure defined by \eqref{def: Gamma}.  In the Euclidean case, if $U$ is a domain with a smooth boundary, $\mathcal{E}_U^N$ yields the usual Neumann Dirichlet form, see for instance \cite{Chen-Dirichlet}. In our general setting, as opposed to the Dirichlet case,  $(\mathcal{E}^N_U,\mathcal{F}(U)) $ is not necessarily a strongly local regular Dirichlet form. In order to ensure this property we shall restrict our class of domains as below.


\begin{defn}\label{def: uniform domain} As in Definition \ref{Def: inner uni domain}, consider an open, connected bounded set $U$ whose closure is compact. Let $c, C$ be two constants with $c\in(0,1)$ and $C>1$. We say that a rectifiable curve 
 $\gamma: [\alpha,\beta] \to U$ is a $(c,C)$ uniform curve in $U$ if the following two conditions are satisfied:
\begin{enumerate}
\item For all $t \in [\alpha,\beta]$, $d( \gamma(t), \partial U) \ge c \min \{ d(\gamma(\alpha),\gamma(t)), d(\gamma(t),\gamma(\beta)) \}$,
\item $\mathrm{length}(\gamma) \le C d (\gamma(\alpha),\gamma(\beta))$.
\end{enumerate}
Then the domain $U$ is called a $(c,C)$  uniform domain if any two points in $U$ can be joined by a $(c,C)$ uniform curve in $U$.
\end{defn}

Since $d \le d_U$, it is easy to check that a uniform domain is always inner uniform, see for instance Section 3.1.2 in \cite{MR2807275}. The converse is not true in general. For instance, a disc with one slit is an example of a inner uniform domain which is not uniform. In Euclidean space, any bounded convex domain is uniform.

The main advantage of uniform domains is that the Dirichlet form $\mathcal{E}^N_U$ becomes strongly local regular and admits a continuous  heat kernel with sub-Gaussian estimates. We spell out this theorem below.

\begin{thm}[\protect{\cite[Theorem 2.8]{murugan2023heat}}]\label{estimate_heat_neumann}
  Assume that $U$ is a  uniform domain as introduced in Definition \ref{def: uniform domain}. Then the form  $(\mathcal{E}^N_U,\mathcal{F}(U)) $ defined by \eqref{D form Neumann} is a strongly local regular Dirichlet form which admits a continuous heat kernel $p_t^{N,U}(x,y)$. Moreover, the Neumann heat kernel $p_t^{N,U}(x,y)$ satisfies the upper estimate
\begin{equation}\label{eq:subGauss-upperN}
p_{t}^{N,U}(x,y)\leq\frac{ c_{1}}{\min(1,t^{\frac{d_h}{d_w}})}\!\exp\biggl(\!-c_{2}\Bigl(\frac{d(x,y)^{d_w}}{t}\Bigr)^{\frac{1}{d_w-1}}\!\biggr) \, ,
\end{equation}
and the lower estimate
\begin{equation}\label{eq:loweGauss-upperN}
p_{t}^{N,U}(x,y)\geq \frac{c_{3}} {\min(1,t^{\frac{d_h}{d_w}})}\!\exp\biggl(\!-c_{4}\Bigl(\frac{d(x,y)^{d_w}}{t}\Bigr)^{\frac{1}{d_w-1}}\!\biggr) \, ,
\end{equation}
for every \ $(x,y)\in U \times U$ and $t\in\bigl(0,+\infty)$.

\end{thm}

Notice that under the conditions of Theorem \ref{estimate_heat_neumann}, the form  $(\mathcal{E}^N_U,\mathcal{F}(U)) $ is referred to as Neumann Dirichlet form on $U$. The related diffusion process is the reflected Brownian motion. 
We now  close this section by stating a lemma about the $L^2$ norm of the heat kernel, which is similar to Corollary \ref{lem:L2bound}.


\begin{lem}\label{lem:L2bound-N} Under the same conditions as for Theorem \ref{estimate_heat_neumann}, there exists a constant $C>0$ such that for all $x\in K$,
\[
\int_U p_t^{N,U}(x,y)^2 d\mu(y)\le
\left\{ \begin{aligned}
& C t^{-\frac{d_h}{d_w}}, &0<t<1,\\
& C , &t\ge 1.
\end{aligned}\right.
\]
\end{lem}

\subsubsection{Fractional Laplacian}
In the case of Neumann boundary conditions, the fractional Laplacian is constructed very similarly to the Dirichlet case. We first summarize the spectral properties of the Laplacian in the proposition below.

\begin{prop}\label{prop:Sepctral decomp Neumann}Consider a uniform domain $U$ like in Definition \ref{def: uniform domain} and the Dirichlet form $\mathcal{E}_U^N$ defined by~\eqref{D form Neumann}. The generator of $\mathcal{E}_U^N$ is denoted by $\Delta_{N,U}$ and called the Neumann Laplacian. Then:
\begin{enumerate}
\item 
The operator $-\Delta_{N,U}$ admits a discrete spectrum $\{\nu_j; j\geq0\}$ with $\nu_0=0$ and $0<\nu_1\le \nu_2\le  \cdots \le \nu_j \le \cdots$. The corresponding eigenfunctions are such that $\Psi_0=\mu(U)^{-1}$ and $\{\Psi_j; j\geq1\}$ form an orthonormal basis of 
$
L^2_0(U,\mu)=\left\{ f\in L^2(U,\mu); \   \int_U fd\mu=0\right\}.
$

\item 
The Neumann heat kernel $p_t^{N,U}$ of Theorem \ref{estimate_heat_neumann} admits a convergent series expansion for all $t>0$ and $x,y\in U$: 
\begin{equation}\label{eq:HKexpansion-Neum}
p_t^{N,U}(x,y)=\frac{1}{\mu(U)}+\sum_{j=1}^{+\infty} e^{-\nu_j t} \Psi_j(x) \Psi_j(y).
\end{equation}

\item
For every $\alpha> d_h/d_w$, we have  $\sum_{j=1}^{+\infty} \frac{1}{\nu_j^{\alpha}} <+\infty.
$
\end{enumerate}

\end{prop}


%

%
%
We now define the fractional powers of the Laplacian along the same lines as Definition~\ref{FLaplacian}. Notice however that this definition avoids the eigenvalue $\nu_0=0$, which would create some singularities.

\begin{defn}[Neumann Fractional Laplacians]\label{FLaplacian-N} Consider the same setting as for Proposition~\ref{prop:Sepctral decomp Neumann}, and let $\alpha$ be a strictly positive parameter. For $f \in L^2(U,\mu)$  the fractional Laplacian $(-\Delta_{N,U})^{-\alpha}$ on $f$ is defined as
\[
(-\Delta_{N,U})^{-\alpha} f  =\sum_{j=1}^{+\infty} \frac{1}{\nu_j^\alpha} \Psi_j \int_U  \Psi_j(y) f(y) d\mu(y).
\]
\end{defn}

\begin{remark} Exactly as for Remark \ref{Rmk:frac Lapl D}, the fractional Laplacian $(-\Delta_{N,U})^{-\alpha}: L^2(U,\mu) \to L_0^2(U,\mu)$ is a bounded operator. More precisely, one has 
$\|(-\Delta_{N,U})^{-\alpha}\|_{ L^2(U,\mu) \to L^2(U,\mu)}\le \nu_1^{-\alpha}$ as  in Remark \ref{Rmk:frac Lapl D}.
\end{remark}

In order to discuss the kernels for the operators $(-\Delta_{N,U})^{-\alpha}$, let us introduce a space of test functions adapted to the Neumann condition context.  
  
\begin{defn}\label{def:Neumann kernel} Consider a uniform domain $U$. Let $\{\Psi_n; n\geq1\}$ be the family of eigenfunctions introduced in Proposition \ref{prop:Sepctral decomp Neumann}. We define the Neumann-Schwartz space of test functions on $U$ as
\[
\mathcal{S}_{N}(U)= \left\{ f \in L^2(U,\mu),  \text{for\ all}\  k \in \mathbb N,  \lim_{n \to +\infty} n^k \left| \int_U  \Psi_n(y) f(y) d\mu(y) \right| =0 \right\}.
\]
\end{defn}
\noindent
We can now identify a kernel for $(-\Delta_{N,U})^{-\alpha}$, similarly to Lemma \ref{Lem: kernel frac Lapl D}. The proof is similar to~\cite{BaudoinChen} and omitted. 

\begin{lem}
Let $U$ be an inner domain and $\alpha>0$. Recall that the operator $(-\Delta_{N,U})^{-\alpha}$ is introduced in Definition \ref{FLaplacian-N}. Then $(-\Delta_{N,U})^{-\alpha}$ admits a kernel $G_\alpha^{N,U}$ given by 
\begin{align}\label{kernel frac Lapl N}
G_\alpha^{N,U}(x,y)=\frac{1}{\Gamma(s)}\int_0^{+\infty} t^{\alpha-1}\left( p_t^{N,U}(x,y)-\frac{1}{\mu(U)}\right) dt, \quad x,y \in U, x\neq y,
\end{align}
where $p_t^{N,U}$ is the heat kernel in Theorem \ref{estimate_heat_neumann}. As in Lemma \ref{Lem: kernel frac Lapl D}, this means that for $f\in\mathcal{S}_N(U)$ and $y\in U$ we have
\[
  (-\Delta_{N,U})^{-\alpha} f (y)= \int_U f(x)  G^{N,U}_{\alpha}(x,y) d\mu(x) .
\]
For $f$ and $g$ in $\mathcal{S}_N(U)$, it also holds that
\[
 \int_{U} (-\Delta_{N,U})^{-\alpha} f (-\Delta_{N,U})^{-\alpha} g d\mu= \int_U \int_U f(x) g(y) G^{N,U}_{2\alpha}(x,y) d\mu(x) d\mu(y).
\]
\end{lem}
We now state some useful bounds for the fractional Laplacian kernel. We sketch the proof of the lower bound for lack of a proper reference.


\begin{prop}\label{Control-G-N}Consider a uniform domain $U$. For $\alpha>0$, let $G_\alpha^{N,U}$ be the kernel given by~\eqref{kernel frac Lapl N}. Then the following bounds hold true:
\begin{enumerate}
\item 
If $0<\alpha<d_h/d_w$,  there exist three finite positive constants $C_1,C_2,C_3$ such that for any $x,y\in U$, $x\ne y$,
\begin{align}\label{est: Neumann G 1}
C_1 d(x,y)^{\alpha d_w-d_h} -C_2\le G^{N,U}_\alpha (x,y) \le C_3 d(x,y)^{\alpha d_w-d_h} .
\end{align}
\item For $\alpha=d_h/d_w$, there exist three finite positive constants $C_1,C_2,C_3$ such that for any $x,y\in U$ such that $x\ne y$, 
\begin{align}\label{est: Neumann G 2}
-C_1 \ln d(x,y) -C_2\le G^{N,U}_\alpha (x,y) \le C_3 \max(\left| \ln d(x,y)\right|,1).
\end{align}
\item
If $\alpha > d_h/d_w$, there exist  finite positive constants $C_1,C_2$ such that for any $x,y\in U$, 
\begin{align}\label{est: Neumann G 3}
-C_1 \le G^{N,U}_\alpha (x,y) \le C_2 .
\end{align}

\end{enumerate} 

\end{prop}

\begin{proof}
The upper estimates are proved in \cite[Proposition 2.6]{BaudoinLacaux} and we omit the details here.
Similarly as in the proof of Proposition \ref{Control-G-N}, we show the lower estimates by using the Neumann heat kernel lower bound in Theorem \ref{estimate_heat_neumann}. To this aim, we decompose the integral defining $G_\alpha^{N,U}$ as
\begin{align}\label{prop 2.30 mid 1}
G_\alpha^{N,U}(x,y)
&=\int_0^1 t^{\alpha-1}\left(p_t^{N,U}(x,y)-\frac1{\mu(U)}\right)dt+\int_1^\infty t^{\alpha-1}\left(p_t^{N,U}(x,y)-\frac1{\mu(U)}\right)dt
\\ &:=\widetilde{G}_\alpha^{N,U}(x,y)+\hat{G}_\alpha^{N,U}(x,y).
\end{align}
In addition, it is proved in  \cite[Lemma 2.3]{BaudoinLacaux} that there exists a constant $C>0$ such that 
\begin{equation}\label{eq:GN-1-infty}
\left|\hat{G}_\alpha^{N,U}(x,y)\right|<C, \quad \text{for\ all}\  x,y\in U.   
\end{equation}
Our lower bounds in \eqref{est: Neumann G 1}-\eqref{est: Neumann G 2}-\eqref{est: Neumann G 3} will thus be obtained through the analysis of $\widetilde{G}_\alpha^{N,U}$. We now separate our considerations according to the value of $\alpha$.

\noindent
\emph{Case $0<\alpha<d_h/d_w$.} \ We resort to relation \eqref{eq:loweGauss-upperN} in order to write
\begin{align}\label{case 1 mid 1}
\widetilde G_\alpha^{N,U}(x,y)
\ge c_3\int_0^1 t^{\alpha-1-\frac{d_h}{d_w}} \!\exp\biggl(\!-c_{4}\Bigl(\frac{d(x,y)^{d_w}}{t}\Bigr)^{\frac{1}{d_w-1}}\!\biggr) dt -\int_0^1 t^{\alpha-1}\frac{1}{\mu(U)}dt.
\end{align}
Then an elementary change of variable (along the same lines as \eqref{eq:RK-lower-less1}-\eqref{eq:RK-lower-less}), plus the fact that $\int_0^1 t^{\alpha-1}dt$ is a convergent integral, show that 
\begin{align}\label{case 1 mid 2}
\widetilde G_\alpha^{N,U}(x,y) \ge C_1d(x,y)^{\alpha d_w-d_h}-C_2.
\end{align}

\noindent
\emph{Case $\alpha= d_h/d_w$.}  Just like in the previous case, we separate the exponential term from the constant term in the expression of $\widetilde{G}^\alpha_{N,U}$ we get
\begin{align*}
    \widetilde G_\alpha^{N,U}(x,y)
    \ge c_3\int_0^1 t^{-1} \!\exp\biggl(\!-c_{4}\Bigl(\frac{d(x,y)^{d_w}}{t}\Bigr)^{\frac{1}{d_w-1}}\!\biggr) dt -C_2.
\end{align*}
Next perform the change of variable $u=t/d_U(x,y)^{d_w}$ and introduce the diameter $R=\mathrm{diam}(U)$, as we did in \eqref{Prop2.19 MS1}-\eqref{eq:RK-lower-ge2}. We end up with 
\begin{align*}
    \widetilde G_\alpha^{N,U}(x,y)
\ge c\,   \int_{1/R^{d_w}}^{1/d(x,y)^{d_w}}\!\exp\biggl(\!-c_{4}\Bigl(\frac{1}{u}\Bigr)^{\frac{1}{d_w-1}}\!\biggr) \frac{du}{u}  -C_2.
\end{align*}
We now trivially bound the exponential term by $1$ in order to get
\begin{align}
    \widetilde G_\alpha^{N,U}(x,y)
   \ge -C_1\ln d(x,y)-C_2.\label{case 2 mid1}
\end{align}

\noindent
\emph{Case $\alpha> d_h/d_w$.} \ If $\alpha>d_h/d_w$, we simply lower bound the exponential term in \eqref{case 1 mid 1} by $0$. This allows to write
\begin{equation}\label{case 3 mid1}
\widetilde G_\alpha^{N,U}(x,y)
\ge -\frac{1}{\mu(U)} \int_0^1 t^{\alpha-1} dt=-C_2.
\end{equation}

\noindent
\emph{Conclusion.} 
Combining  \eqref{prop 2.30 mid 1}, \eqref{eq:GN-1-infty}, \eqref{case 1 mid 2}, \eqref{case 2 mid1} and \eqref{case 3 mid1}, we have achieved our claims.
\end{proof}






We now state a useful bound for the Neumann heat kernel which mirrors Lemma \ref{estimate frac pt}.

\begin{lem}\label{estimate frac pt neumann}
  Let $\alpha \ge 0 $ and $0<\gamma<\nu_1 $. Recall that the operator $(-\Delta_{N,U})^{-\alpha}$ is given in Definition \ref{FLaplacian-N}. Then there exists a constant $C$ such that for every $t>0$
  \[
 \sup_{x \in U} \|(-\Delta_{N,U})^{-\alpha}p_t^{N,U}(x,\cdot)\|_{L^2(U,\mu)}  \le
 \begin{cases}
C  \max (1,t^{-\frac{d_h}{2d_w}}),\quad \alpha=0, \\
C e^{-\gamma t} t^{\alpha-\frac{d_h}{2d_w}}, \quad 0<\alpha<\frac{d_h}{2d_w},
\\
C e^{-\gamma t} (|\ln t|+1), \quad \alpha=\frac{d_h}{2d_w}, \\
C e^{-\gamma t}, \quad \alpha >\frac{d_h}{2d_w}.
 \end{cases}
  \]
  \end{lem}
  
\begin{remark}
    It will be apparent from our proof that the parameter $\gamma<\nu_1$ is used in order to produce some convergent integrals $\int_1^\infty e^{-(\nu_1-\gamma)s}ds$. Hence our decay rate is not as sharp as the $e^{-\lambda_1 t}$ we obtained in \eqref{eq:RK-lower-ge2} for the Dirichlet case. However, Lemma \ref{estimate frac pt neumann} is only applied in Theorem \ref{existence unique neuma}, for which optimal decay rates are not necessary. 
\end{remark}

\begin{proof}
This proof is an adaption of the proof of Lemma \ref{estimate frac pt} to the Neumann context. First the case $\alpha=0$ follows from Lemma \ref{lem:L2bound-N}, so we assume $\alpha >0$. Next owing to \eqref{kernel frac Lapl N} and applying the semigroup property similarly to \eqref{Lem2.20 MS 1} we deduce for every $t>0$ 
\begin{align*}
(-\Delta_{N,U})^{-\alpha}p_t^{N,U}(x,y)&=\int_U G_\alpha^{N,U}(y,z)p_t^{N,U}(x,z) d\mu(z) \\
 &=\frac{1}{\Gamma(\alpha)}\int_0^{+\infty}\int_U s^{\alpha -1}\left(p_s^{N,U}(y,z)-\frac{1}{\mu(U)} \right)p_t^{N,U}(x,z) d\mu(z)ds \\
 &=\frac{1}{\Gamma(\alpha)}\int_0^{+\infty} s^{\alpha-1}\left(p_{s+t}^{N,U}(x,y) -\frac{1}{\mu(U)} \right)ds.
\end{align*}
We now proceed to apply Minkowski's inequality along the same lines as \eqref{eq-est-del-pt}. This yields
\begin{align}\label{eq:DeltaHK-N}
\|(-\Delta_D)^{-\alpha}p_t^{N,U}(x,\cdot)\|_{L^2(U,\mu)}  
 & \le \frac{1}{\Gamma(\alpha)}\int_0^{+\infty} s^{\alpha-1}\left\| p_{s+t}^{N,U}(x,\cdot) -\frac{1}{\mu(U)}\right\|_{L^2(U,\mu)} ds \notag
 \\
 & = \frac{1}{\Gamma(\alpha)}\int_0^{+\infty} s^{\alpha-1} \left(p_{2(s+t)}^{N,U}(x,x)-\frac{1}{\mu(U)} \right)^{1/2} ds,
\end{align}
where the equality \eqref{eq:DeltaHK-N} follows from the semigroup property and the fact that the Neumann heat semigroup is conservative.

The proof is finished thanks to elementary arguments (like for Lemma \ref{estimate frac pt}). Let us just mention that the integral in \eqref{eq:DeltaHK-N} has to be split into two pieces, separately corresponding to integrals over $(0,1]$ and $(1,\infty)$. For the integral over $(0,1]$ one resorts to the heat kernel upper bound \eqref{eq:subGauss-upperN}. For the integral over $(1,\infty)$ we use identity \eqref{eq:HKexpansion-Neum}, which is valid for $p^{N,U}_{2(s+t)}-\mu(U)^{-1}$ and gives exponential decay. Notice however that the parameter $\gamma<\nu_1$ is introduced for this step, in order to be left with a convergent integral $\int_1^\infty e^{-(\nu_1-\gamma)s}ds$.
\end{proof}

We close this section by giving a definition of Sobolev spaces adapting Definition~\ref{sobo_def_D} to Neumann boundary conditions.
\begin{defn}[Sobolev Space $\mathcal{W}_N^{-\alpha}(U)$] \label{sobo_def_N}
Let $C_U>0$ be an arbitrary constant. For $\alpha>0$, we define the Sobolev space $\mathcal{W}_N^{-\alpha}(U)$ as the completion of $ \mathcal{S}_{N}(U)$ with respect to the inner product
\begin{align}
  \left\langle  f ,g \right\rangle_{\mathcal{W}_N^{-\alpha}(U)}
  =\int_U \int_U f(x) g(y)( G^{N,U}_{2\alpha}(x,y) +C_U ) d \mu(x) d\mu(y) \label{Sobo_norm_N}.
\end{align}
\end{defn}

\begin{remark}\label{rmk regarding C_U}
As mentioned in Definition \ref{sobo_def_N}, the value of the  constant $C_U$ is irrelevant. We note however from Proposition \ref{Control-G-N} that it can be chosen in such a way that $G^{N,U}_{2\alpha}(x,y)+C_U$ is bounded from below by a positive constant.
\end{remark}

 \begin{remark}
 It is readily checked that an alternative way to express the Sobolev norm in \eqref{Sobo_norm_N} is 
 \begin{align*}
  \left\langle  f ,g \right\rangle_{\mathcal{W}_N^{-\alpha}(U)}
=&\int_U \int_U f(x) g(y) G^{N,U}_{2\alpha}(x,y) d\mu(x) d\mu(y) +C_U \int_U f d\mu \int_U g d\mu\\
=& \int_{U} (-\Delta_{N,U})^{-\alpha} f (-\Delta_{N,U})^{-\alpha} g d\mu +C_U \int_U f d\mu \int_U g d\mu. 
\end{align*}
\end{remark}

\subsection{Examples of application}
This section is devoted to an overview of the potential applications of our general metric space setting. We will first look at intervals, then at metric graphs. We shall also mention a very typical application to fractal sets.

\subsubsection{Example 1: Compact intervals}\label{interval example}

The simplest example which fits our framework is the unit interval $U=(0,1) \subset \mathbb{R}$ (or any interval $(a,b)$). Here one starts from the metric space $X=\R$ equipped with the Euclidean distance $d(x,y)=|x-y|$. The Brownian Dirichlet form on $\mathbb R$ is
\[
\mathcal{E}(f,g)=\int_{\mathbb R} f'(x)g'(x) dx, \quad f,g \in W^{1,2} (\mathbb R).
\]
The corresponding heat kernel is known to be
\[
p_t(x,y)=\frac{1}{\sqrt{ 4\pi t}} e^{-\frac{(x-y)^2}{4t} }.
\]
Referring to our formula \eqref{eq:subGauss-upper}, this means that $d_h=1$ and $d_w=2$. Then the two boundary conditions we consider on $(0,1)$ are spelled out as follows.

\noindent
 \textit{(a) Dirichlet boundary conditions.} 
The pair $(\mathcal{E}^D, \mathcal{F}^D((0,1)))$ in \eqref{def: D Dirichlet form} is then given by
\[
\mathcal{E}^D(f,g)=\int_0^1 f'(x)g'(x) dx, \quad f,g \in W^{1,2}_0 ([0,1]),
\]
where $W^{1,2}_0 ([0,1])$ designates the usual Sobolev space of functions vanishing at both $0$ and $1$. 
Functions $f$ in the domain of the Dirichlet Laplacian $\Delta_D=\frac{d^2}{dx^2}$ satisfy the boundary condition $f(0)=f(1)=0$.
An explicit formula is available for the heat kernel $p^{D,(0,1)}$ of Theorem \ref{estimate_heat_diri},
\[
p_t^D(x,y)=\frac{1}{\sqrt{ \pi t}} \sum_{k \in \mathbb{Z}}  \left( e^{-\frac{(x-y+k)^2}{4t} }-e^{-\frac{(x+y+k)^2}{4t} } \right).
\]

\noindent
\textit{(b) Neumann boundary conditions.} 
The Neumann Dirichlet form \eqref{D form Neumann} is given by 
\[
\mathcal{E}^N(f,g)=\int_0^1 f'(x)g'(x) dx, \quad f,g \in  W^{1,2} ([0,1]),
\]
where $W^{1,2} ([0,1])$ is the usual Sobolev space on $[0,1]$.
Smooth functions $f$ in the domain of the Neumann Laplacian $\Delta_N=\frac{d^2}{dx^2}$ satisfy the boundary condition $f'(0)=f'(1)=0$.
 The Neumann heat kernel is given by
\[
p_t^N(x,y)=\frac{1}{\sqrt{ \pi t}}  \sum_{k \in \mathbb{Z}} \left( e^{-\frac{(x-y+k)^2}{4t} }+e^{-\frac{(x+y+k)^2}{4t} } \right).
\]
We note that the parabolic Anderson models with Dirichlet and Neumann boundary condition have been studied in that case by \cite{FoondunNualart}, see also \cite{MR3359595}.

\subsubsection{Example 2: Metric graphs}\label{metric graph section}

For a reference on the general theory of metric graphs, also referred to as cable systems in the literature, we refer to \cite{Pos12}, see also \cite{BK} and references therein. Intuitively, a metric graph is a finite collection of intervals (or edges), for which some of them are connected  at their extremities. More precisely, we use $\mathbf{G}$ to denote a connected metric graph, which is composed of a finite set of vertices $V$, a finite set of (internal) edges $\mathbf{E}$ and a finite set of rays $\mathbf{R}$. For each edge $e\in \mathbf{E} $, there are two endpoints $e^-$ and $e^+$ in $V$ as well as a length $r(e)>0$. Each ray has one associated endpoint $e^-$ in $V$ and the length is infinite.  For $v\in V$, define the set of adjacent edges $\mathbf{E}_v = \set{e\in \mathbf{E} \cup \mathbf{R}~|~v = e^- \text{ or } v = e^{+}}$. 

 For $e\in\mathbf{E}$ let $I_e = [0,r(e)]$ and if $e\in\mathbf{R}$ then $I_e = [0,\infty)$. In this case the metric graph $\mathbf{G}$  is the set $\cup_{e\in\mathbf{E}\cup\mathbf{R}} I_e$ modulo the equivalence relation which identifies  endpoints of $I_{e_1}$ and $I_{e_2}$ if associated endpoints of $e_1$ and $e_2$ are the same vertex. 
We also define $\Phi_e: I_e\to\mathbf{G}$ to be the projection onto the equivalence classes. For example $\Phi_{e_1}(0) = \Phi_{e_2}(r(e_2))$ if $e^-_1 = e^+_2$. We may think of $I_e$ as subsets of $\mathbf{G}$ and refer to $0 \in I_e$ as $e^-$ and $r(e) \in I_e$ as $e^+$.

Now, we shall introduce some notation concerning the function spaces on $\mathbf{G}$. Define the reference measure $\mu$ on $\mathbf{G}$ to be the Lebesgue measure when restricted to each $I_e$. 
The space $L^2(\mathbf{G}) = L^2(\mathbf{G},\mu)$ is defined as 
\[
L^2(\mathbf{G}) = \lbrace f = (f_e)_{e\in \mathbf{E}\cup\mathbf{R}}\,; f_e \in L^2(I_e) \rbrace,
\]
that is, we identify $L^2(\mathbf{G})$ with $\oplus_e L^2(I_e)$.
Other function spaces have similar vector decompositions, perhaps with boundary conditions. For example, we shall think of  continuous functions $C(\mathbf{G})$ to be the vectors with entries in $ C(I_e)$ where, if $v$ is an element of both $I_{e_1}$ and $I_{e_2}$ then $f_{e_1}(v) = f_{e_2}(v)$. 
We also define the Sobolev space $H^1_0(\mathbf{G})$ as
\[
H^1_0(\mathbf{G})=\lbrace f = (f_e)_{e\in \mathbf{E}\cup\mathbf{R}}\,; \, f_e \in H^1(I_e) , \mbox{ with $f$ continuous at vertices}  \rbrace.
\]

 \begin{figure}[htb]
  	\centering
  	\includegraphics[height=0.2\textwidth] {{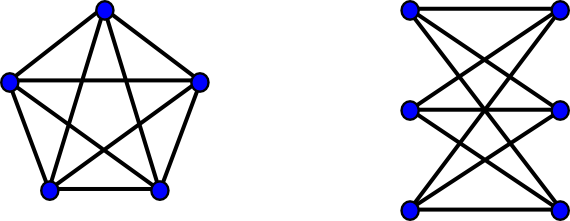}}
  	\caption{Examples of metric graphs (picture from Wikipedia)}
  \end{figure}

\noindent
When it is well defined, we consider $f(v)$ to be the vector $(f_e(v))_{e\in \mathbf{E}_v}$ of values of $f$ (or traces of $f$) at the associated endpoint of $e$. 
Let us now label a definition which will be useful for the definition of Laplacian on $\mathbf{G}$.
\begin{defn}\label{def-Uv}
    For all $v\in V$, we define $U_v:E\to \{-1,1\}$ by setting $U_v(e) = 1$ if $v = e^-$ and $U_v(e) = -1$ if $v = e^+$. In this way, the inward facing normal derivative of $f$ at $v\in V$ along an edge $e$ is $U_v(e)f_e'(v)$. 
\end{defn}

\noindent
Here and later, $f_e'(0)$ or $f'_e(r(e))$ is understood to be the trace of $f'_e$ onto the boundary of the interval $I_e$. 

With those preliminary notions in hand, we are now ready to define Dirichlet forms on $G$. One defines first  a  derivative operator $d: H^1_0(\mathbf{G}) \to L^2(\mathbf{G})$ by $(df)_e(x) = f'_e(x)$, which we will concisely denote by $df = f'$. Note that, up to a sign, $d$ depends on the orientation of the graph. Then the standard Dirichlet form on $G$ corresponds to the so-called Kirchhoff boundary condition, see \cite{BK}. It is defined by
\[
\mathcal{E}(f,g) = \int f'g' \ d\mu :=\sum_{e \in \mathbf{E}} \int_0^{r(e)}  f'_e(x)g'_e(x) dx+  \sum_{e \in \mathbf{R}} \int_0^{+\infty}  f'_e(x)g'_e(x) dx,
\]
the domain of $\mathcal{E}$ being $H^1_0(\mathbf{G})$. The corresponding operator of $\mathcal{E}$ is $\Delta f=-f''$, with domain
\[
\mathrm{Dom}(\Delta )= \set{f\in H^1_0(\mathbf{G})~|~ \forall ~e, f_e' \in H^1(I_e),\text{ and } \forall ~v\in V,~\sum_{v\in \mathbf{E}_v}U_v(e)f_e'(v)=0},
\]
where we recall that the operator $U_v$ is introduced in Definition \ref{def-Uv}. Notice that $\mathcal{E}$ and $\Delta$ do not
 depend on the orientation of the graph. Moreover, it is known, see \cite{haeseler2011heat}, that the heat kernel $p_t(x,y)$ associated with the Dirichlet form $\mathcal{E}$ satisfies the Gaussian estimates
 \begin{equation}\label{eq:subGauss-upperGr}
 c_{1}t^{-1/2}\exp\biggl(-c_{2}\Bigl(\frac{d(x,y)^{2}}{t}\Bigr)\biggr) 
 \le p_{t}(x,y)\leq c_{3}t^{-1/2}\exp\biggl(-c_{4}\Bigl(\frac{d(x,y)^{2}}{t}\Bigr)\biggr) \, ,
 \end{equation}
 for  $(x,y)\in \mathbf G\times \mathbf G $  $t\in\bigl(0,\mathrm{diam}(\mathbf G)^{1/2}]$, where the distance $d$ is the geodesic distance. Therefore as in Section \ref{interval example} and referring to formula \eqref{eq:subGauss-upper}, we have $d_h=1$ and $d_w=2$.


In case of a metric graph $\mathbf{G}$, our main examples of  subdomains are connected bounded open subsets. Those are easily seen to be uniform domains as in Definition~\ref{def: uniform domain} (and hence they are also inner uniform domains as in Definition \ref{Def: inner uni domain}). Therefore our estimates from Sections \ref{Dirichlet boundary} and~\ref{Neumann boundary} apply to this setting.

\subsubsection{Example 3: Fractals,  Sierpi\'nski gasket}\label{section gasket}
Besides metric graphs, a large class of examples that also fit in our setting is p.c.f. fractals as presented in \cite{Barlow}, see also \cite{Kigami}. For the sake of presentation we illustrate in detail the case of the Sierpi\'nski gasket, which is one of the most popular examples of a p.c.f. fractal. 

One of the classical ways to define the Sierpi\'nski gasket is as follows:
let $V_0=\{p_1, p_2, p_3\}$ be a set of vertices of an equilateral triangle of side 1 in $\mathbb C$. Define 
\begin{align}\label{eq-SK-f}
    \mathfrak f_i(z)=\frac{z-p_i}{2}+p_i, 
\quad\text{for}\quad
i=1,2,3 .
\end{align}
The Sierpi\'nski gasket $K$ (see Figure \ref{fig-SG}) is the unique non-empty compact subset in $\mathbb C$ such that 
\begin{align}\label{eq-SK-K}
K=\bigcup_{i=1}^3 \mathfrak f_i(K).
\end{align}
\begin{figure}[htb]
\centering
    	\includegraphics[trim={60 10 180 60},height=0.20\textwidth]{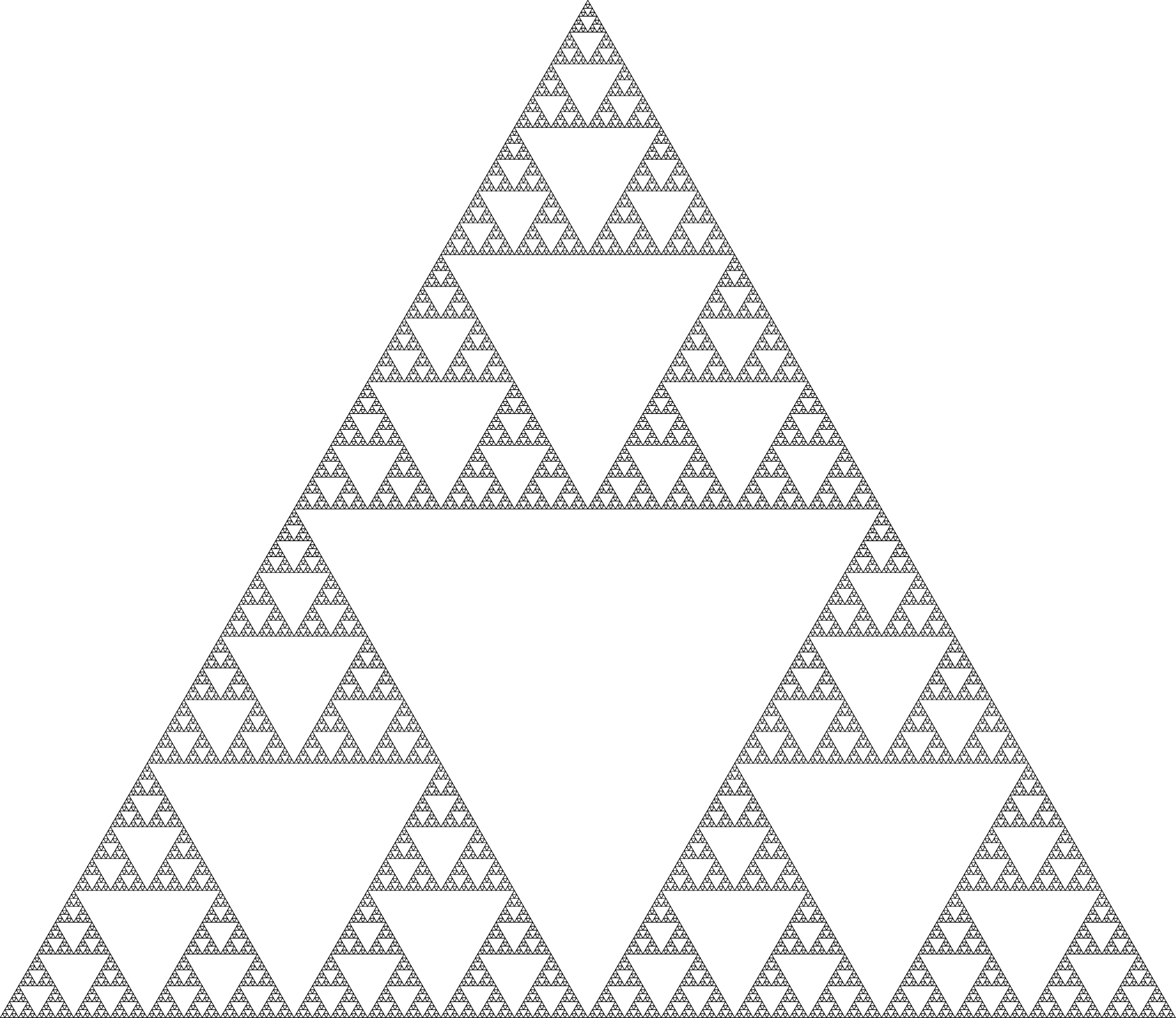}
	\caption{Sierpi\'nski gasket.} \label{fig-SG}
\end{figure}

The set $V_0$ is called the boundary of $K$, we will also denote it by $\partial K$. The Hausdorff dimension of $K$ with respect to the Euclidean metric (denoted $d(x,y)=| x - y |$) is given by $d_h=\frac{\ln 3}{\ln 2}$. A (normalized) Hausdorff measure on $K$ is given by the Borel measure $\mu$ on $K$ such that for every $i_1, \dots, i_n \in \{ 1,2,3 \} $,
\[
\mu \left(  \mathfrak f_{i_1} \circ \cdots \circ \mathfrak f_{i_n}  (K)\right)=3^{-n}.
\]
This measure $\mu$ is $d_h$-Ahlfors regular, i.e., there exist constants $c,C>0$ such that for every $x \in K$ and $r \in [0, \mathrm{diam} (K) ]$,
\begin{equation}\label{eq:Ahlfors}
c r^{d_h} \le \mu (B(x,r)) \le C r^{d_h}.
\end{equation}

%
%
It will be useful to approximate the gasket $K$ by a sequence of discrete objects. Namely, starting from the set $V_0=\{p_1, p_2, p_3\}$, we define a sequence of sets $\{V_m\}_{m\ge 0}$ inductively by
\begin{align}\label{eq-SG}
V_{m+1}=\bigcup_{i=1}^3 \mathfrak f_i(V_m).
\end{align}
Then we have a natural sequence of Sierpi\'nski gasket graphs (or pre-gaskets) $\{G_m\}_{m\ge 0}$ whose edges have length $2^{-m}$ and whose set of vertices is $V_m$, see Figure \ref{SGgraphs}. Notice that $\#V_m=\frac{3(3^m+1)}2$. We will use the notations $V_*=\cup_{m\ge 0}V_m$ and $V_*^0=\cup_{m\ge 0}V_m\setminus V_0$.

 \begin{figure}[htb]
  	\noindent
  \makebox[\textwidth]{\includegraphics[height=0.2\textwidth] {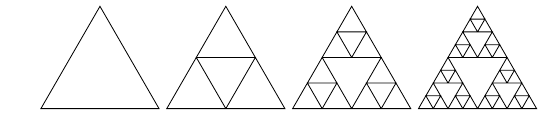}}
  	\caption{Sierpi\'nski gasket graphs $G_0$, $G_1$, $G_2$ and $G_3$}
	   \label{SGgraphs}
  \end{figure}

We are now ready to define Dirichlet forms on the metric space $X=K$ by approximation. 

\noindent
\textit{(a) Neumann boundary conditions.} 
Let $m\ge 1$. For any $f\in \mathbb R^{V_m}$, we consider the quadratic form 
\[
\mathcal E_m(f,f)= \left(\frac53\right)^m\sum_{p,q\in V_m, p\sim q} (f(p)-f(q))^2.
\]
We can then define a resistance form $(\mathcal E, \mathcal F_*)$ on $V_*$ by setting 
\[
\mathcal F_*=\{f\in \mathbb R^{V_*}, \lim_{m\to \infty} \mathcal E_m(f,f)<\infty\}
\]
and for $f\in \mathcal F_*$
\[
\mathcal E(f,f)=\lim_{m\to \infty} \mathcal E_m(f,f).
\]

Each function $f\in \mathcal F_*$ can be uniquely extended into a continuous function defined on $K$. We denote by $\mathcal F$ the set of functions with such extensions. It follows from the book of Kigami~\cite{Kigami} that $(\mathcal E, \mathcal F)$ is a local regular Dirichlet form on $L^2(K,\mu)$. The generator of the Dirichlet form $(\mathcal E, \mathcal F)$, denoted by $\Delta_N$, corresponds to the Laplacian with Neumann boundary condition with the domain  given by
\begin{equation}\label{eq-dom-N}
\mathrm{Dom}(\Delta_N)=\Big\{u\in \mathcal F;  \,
\exists\, f\in L^2(K,\mu) 
\text{ such that } \mathcal E(u,v)=\int_K fvd\mu,\, \forall  v\in \mathcal F
\Big\}.
\end{equation}
In case $u\in \mathrm{Dom}(\Delta_N)$, one sets $\Delta_N u=-f$, where $f$ is the function featuring \eqref{eq-dom-N}.
For more details, we refer to \cite[Theorem 3.4.6]{Kigami} and \cite{Barlow}.

\noindent 
\textit{(b) Dirichlet boundary conditions.} Towards a definition of the Laplace operator with Dirichlet boundary conditions on $K$, define 
\[
\mathcal F_0=\set{f\in \mathcal F, f|_{V_0}=0 }.
\]
Then by \cite[Corollary 3.4.7]{Kigami}, $(\mathcal E, \mathcal F_0)$ is a local regular Dirichlet form on $L^2(K,\mu)$. The generator of $(\mathcal E, \mathcal F_0)$, denoted by $\Delta_D$, is associated with the Dirichlet boundary condition with the domain 
\begin{equation}\label{eq-dom-D}
\mathrm{Dom}(\Delta_D)=\Big\{u\in \mathcal F_0; \,\exists\, f\in L^2(K,\mu)
 \text{ such that } \mathcal E(u,v)=\int_K fvd\mu, \,\forall v\in \mathcal F_0\Big\}.
\end{equation}
Like in the Neumann case, whenever $u\in \mathrm{Dom}(\Delta_D)$ we set $\Delta_Du=-f$.

\noindent 
\textit{(c) Laplacian from approximations.} 
The Dirichlet and Neumann Laplacians can also be realized from the discrete approximation of the gasket, for which we need some notation. First recall that $V_0=\{p_1,p_2,p_3\}$, and that $V_m$ is defined by \eqref{eq-SG}.
For any $p\in V_m$, denote by $V_{m,p}$ the collection of neighbors of $p$ in $V_m$. Then 
\[
\#V_{m,p}=\begin{cases}
4, &\text{ if }p\notin V_0,\\
2, &\text{ if }p\in V_0.
\end{cases}
\]
For any $f\in \mathbb R^{V_m}$, consider the discrete Laplacian on $V_m$ defined by
\begin{align}\label{discrete laplacian}
\Delta_m f(p)=\frac{3}{2} \, 5^m \sum_{q\in V_{m,p}} (f(q)-f(p)), \quad p \in V_m \setminus V_0.
\end{align}
The Laplace operator is now obtained by approximation in the following way.
Let $C(K)$ be the set of continuous functions on $K$. We define
\begin{equation}\label{eq-domain}
\mathcal D=\set{f\in C(K), \text{ there exists }g\in C(K) \text{ such that }\lim_{m\to \infty}\max_{p\in V_m\setminus V_0} |\Delta_m f(p)-g(p)|=0}.
\end{equation}
Then the Kigami Laplacian $\Delta$ on the Sierpi\'nski gasket $K$ is defined by 
\begin{equation}\label{eq:Laplacian}
\Delta f(p)=\lim_{m\to \infty}\Delta_m f(p), \quad f\in \mathcal D.
\end{equation}
Notice that by definition in \eqref{eq-domain}, for $f\in\mathcal{D}$ the  function $\Delta f$ is continuous  on $K$. For any $f\in \mathcal D$ and $p\in V_0$, let us also define the Neumann derivative by 
\[
(df)_p=-\lim_{m\to \infty}\left(\frac{5}{3}\right)^m\sum_{q\in V_{m,p}}(f(q)-f(p)) \, ,
\]
given that the limit converges. 
The relation between Kigami and $L^2$ Laplacians is the content of the following theorem.
\begin{thm}[\protect{\cite[Theorem 3.7.9]{Kigami}}]\label{Laplacian gasket}
Recall that $\mathrm{Dom}(\Delta_N)$, $\mathrm{Dom}(\Delta_D)$ and $\mathcal{D}$ are respectively defined by \eqref{eq-dom-N}, \eqref{eq-dom-D} and \eqref{eq-domain}.
\begin{enumerate}
    \item If $\mathcal D_N=\{f\in \mathcal D,  (df)_p=0 \text{ for all }p\in V_0 \}$, then $\mathcal D_N=\mathrm{Dom}(\Delta_N) \cap\mathcal D$, $\Delta_N|_{\mathcal D_N}=\Delta|_{\mathcal D_N}$ and $\Delta_N$ is the Friedrichs extension of $\Delta$ on $\mathcal D_N$.
    \item If $\mathcal D_D=\{f\in \mathcal D, f|_{V_0}=0\}$, then $\mathcal D_D=\mathrm{Dom}(\Delta_D) \cap\mathcal D$, $\Delta_D|_{\mathcal D_D}=\Delta|_{\mathcal D_D}$ and $\Delta_D$ is the Friedrichs extension of $\Delta$ on $\mathcal D_D$.
\end{enumerate}
\end{thm}
We finish this section by specifying the applicability of Section \ref{Dirichlet boundary} and \ref{Neumann boundary} to the  Sierpi\'nski gasket context. A brief summary is as below:
\begin{enumerate}
\item The Dirichlet Laplacian $\Delta_D$ in Theorem \ref{Laplacian gasket} generates a heat semigroup $P^D_t$ as given in~\eqref{def: Dirichlet semigroup}. This semigroup admits a kernel $p_t^D$ given by Theorem \ref{estimate_heat_diri}. The upper and lower bounds~\eqref{eq:subGauss-upperD} and~\eqref{eq:subGauss-lowerD} hold true with $d_h=\frac{\ln 3}{\ln 2}\simeq 1.58$ and $d_w=\frac{\ln 5}{\ln 2}\simeq 2.32$. Notice in particular that $d_w > d_h$.
\item The Neumann Laplacian $\Delta_N$ in Theorem \ref{Laplacian gasket} generates a heat semigroup with a kernel $p_t^N$ as in Theorem \ref{estimate_heat_neumann}. The constants $d_h$ and $d_w$ are the same as for the Dirichlet case.

\item The domain $U=K\setminus V_0$ is uniform and besides it, we note that further several explicit examples  of inner uniform and uniform domains in the Sierpi\'nski gasket are discussed  in \cite{Lierl}.
\end{enumerate}

\section{Dirichlet parabolic Anderson model in inner uniform domains}\label{sec: Dirichlet PAM}

In this section we focus on SPDE models with Dirichlet boundary conditions. We will first describe the type of noise we are considering. Then we shall get existence and uniqueness results for the equation and give estimates for the moments of the solution. We label the following hypothesis once for all in the section.
\begin{assump}\label{asump-mms-domain}
    Let $(X, d, \mu)$ be a metric measure space satisfying Assumption \ref{hyp:dirichlet-form}. We consider an inner uniform domain $U$ as in Definition \ref{Def: inner uni domain}.
\end{assump}
\noindent
Notice that since the domain is fixed for the whole section, we will drop  the $U$ from our notation on semigroups, Dirichlet forms, etc. For instance we shall write $\Delta_D$ instead of $\Delta_{D,U}$.


\subsection{Fractional noise}
In this section we briefly describe the structure of the Gaussian noise that we consider as the driving noise for our stochastic heat equation. As the definition below shows, it will be a white noise in time with colored covariance in space.

\begin{defn}\label{def-frac-Gaus-field}
Consider a regularity parameter $\alpha>0$ and the following Hilbert space of space-time functions:
\begin{align}\label{eq-hilb}
\mathcal{H}_\alpha=L^2(\mathbb{R}_+, \mathcal{W}_D^{-\alpha}(U)),  
\end{align}
where $\mathcal{W}_D^{-\alpha}(U)$ is the Sobolev space introduced in Definition \ref{sobo_def_D}. On a complete probability space $(\Omega, \mathcal{G},\mathbf{P})$ we define a centered Gaussian family $\{W_\alpha^D(\phi); \phi\in \mathcal{H}_\alpha\}$, whose covariance is given by 
\begin{align}\label{eq-Gau-cov}
\mathbf E\left[ W_{\alpha}^D(\varphi) W_{\alpha}^D(\psi)\right]
=&
\int_{\R_+}\ \left\langle \varphi (t,\cdot) , \psi (t,\cdot)\right\rangle_{\mathcal{W}_D^{-\alpha}(U)}  dt \notag\\
=&
\int_{\R_+} \int_{U^2} \varphi (t,x)  \psi (t,y)G_{2\alpha}^{D,U}(x,y)d\mu(x)d\mu(y) dt
\, ,
\end{align}
for $\varphi$, $\psi$ in $\mathcal{H}_\alpha$. This family is called Dirichlet fractional Gaussian fields.
\end{defn} 

We are working in situations where the stochastic heat equation accommodates a space-time white noise (corresponding to $\alpha=0$ above). Let us label this definition separately. 

\begin{defn}\label{def:st white}
    On the same probability space $(\Omega, \mathcal{G},\mathbf{P})$ as in Definition \ref{def-frac-Gaus-field}, the space-time white noise $\{W_0^D(\phi); \phi\in L^2(\R_+, L^2(U))\}$ is a centered Gaussian family with covariance function
\begin{equation}
    \mathbf E\left( W_{0}^D(\varphi) W_{0}^D(\psi)\right)
=\int_{\R_+}\ \left\langle \varphi (t,\cdot) , \psi (t,\cdot)\right\rangle_{L^2(U)}  dt\ .
\end{equation}  
\end{defn}

\begin{notation}
Since the parameter $\alpha$ will be fixed for the entire section, we will abbreviate the notation $\mathcal{H}_\alpha$ in \eqref{eq-hilb} as $\mathcal{H}$.
\end{notation}

\begin{remark}\label{remk:regularity of the filed}
    Noises whose covariance functions are based on powers of the Laplacian are common in the literature. Let us mention the reference \cite{lodhia2016fractional} for noises defined in $\R^d$, as well as our studies \cite{BCHOTW, BOTW} in Heisenberg groups. Fractional noises on fractals have also been considered in \cite{BaudoinChen}. Therein the following two regimes are exhibited:
\begin{enumerate}
\item For $0\le \alpha \le \frac{d_h}{2d_w}$, the field $W^D_\alpha$ given in Definition \ref{def-frac-Gaus-field} or Definition \ref{def:st white} is a distribution in both time and space.
\item  For $\alpha > \frac{d_h}{2d_w}$, there exists a centered  continuous Gaussian field $(X_{\alpha}^D  (t,x))_{t \ge 0, x \in U}$ with covariance
\[
\mathbf{E} \left( X_{\alpha}^D  (s,y) X_{\alpha}^D  (t,x)  \right)=\min (s,t) G^D_{2\alpha} (x,y),
\]
such that for every $t >0$ and $f \in \mathcal{S}_D(U)$ we have
\begin{align}\label{density field}
W_{\alpha}^D ( 1_{[0,t]} \otimes f)= \int_U X_{\alpha}^D  (t,x) f(x) d\mu (x) \, .
\end{align}

\end{enumerate}
\end{remark}


 \subsection{Parabolic Anderson model and chaos expansion method}
 
 We are now ready to turn to the main object of study in this paper, that is, the stochastic heat equation \eqref{PAM intro}. Specifically, consider two parameters $\alpha, \beta\geq0$ and the noise $W^D_\alpha$ introduced in Definition \ref{def-frac-Gaus-field}. We wish to solve the following parabolic Anderson model with Dirichlet Laplacian:
\begin{equation}\label{eq:pam-D}
\partial_{t} u(t,x) = \Delta_D u(t,x) + \beta u(t,x) \, \dot{W}_\alpha^D(t,x).
\end{equation}
In this section we focus on a proper definition of equation \eqref{eq:pam-D}. Since most of our considerations here are standard, we might skip some details and refer to classical papers on the topic.

Systems like \eqref{eq:pam-D} are usually solved in a mild form. That is, one rewrites our equation of interest  as
\begin{equation}\label{eq:PAM-mild form-Diri}
    u(t,x)=J_0(t,x)+\beta I(t,x),
\end{equation}
where $J_0(t,x)=P_t^Du_0(x)=\int_U p_t^D(x,y)u(0,y)d\mu(y)$ is the solution to the homogeneous heat equation and $I(t,x)$ is the stochastic integral given by
\begin{equation}
    I(t,x)=\int_{[0,t]\times U}p_{t-s}^D(x,y)u(s,y)W_{\alpha}^D(ds,dy).
\end{equation}
Notice that the stochastic integral $I(t,x)$ above should be understood in the It\^{o} sense.  We refer to \cite{Hu15,Walsh} for a proper definition of this integral, as well as \cite{BOTW} in a more geometric context. A solution to \eqref{eq:PAM-mild form-Diri} is then defined as below.

\begin{defn}\label{def of Ito sol}
A process $u=\{u(t,x); (t,x)\in\mathbb{R}_{+}\times U \}$ is called a random field solution of~\eqref{eq:pam-D} in the It\^o-Skorohod sense if the following conditions are met:
\begin{enumerate}
\item $u$ is adapted;

\item $u$ is jointly measurable with respect to $\mathcal{B}(\mathbb{R}_{+}\times U )\otimes \mathcal{G}$; 
\item $\mathbf{E}(I(t,x)^2)<\infty$ for all $(t,x)\in\mathbb{R}_{+}\times U$;
\item The function $(t,x)\to I(t,x)$ is continuous in $L^2(\Omega)$;
\item $u$ satisfies \eqref{eq:PAM-mild form-Diri} almost surely for all $(t,x)\in\mathbb{R}_{+}\times U$.
\end{enumerate}
\end{defn}

Observe that equation \eqref{eq:PAM-mild form-Diri} can classically be solved by Picard iterations schemes combined with standard It\^{o} type estimates for moments of stochastic integrals. However we shall focus here on a chaos expansion methodology, which leads to sharper moment estimates. 
We briefly recall the setting for chaos expansion now. The reader is sent to \cite{BCHOTW,Hu15} for more detailed accounts. Let us first denote by $\mathcal H_k$ the space 
\begin{align}\label{def:H_k}
\mathcal{H}_k=\mathrm{Span}\left(\left\{H_k(W^D_\alpha(\phi));\phi \in \mathcal H, \|\phi\|_{\mathcal H}=1\right\}\right),
\end{align} 
where Span means closure of the linear span in $L^2(\Omega)$, $H_k$ designates the $k$-th Hermite polynomial, $W^D_\alpha$ is the noise in Definition \ref{def-frac-Gaus-field} and $\mathcal{H}$ is the Hilbert space from \eqref{eq-hilb}. The space $\mathcal{H}_k$ is called the $k$-th Wiener chaos. There exists a linear isometry $I_k$ between $\mathcal H^{\otimes k}$ (with modified norm $\sqrt{k!}\|\cdot\|_{\mathcal H^{\otimes k}}$) and $\mathcal H_k$ given by
\[
I_k\big(\phi^{\otimes k}\big)
=
k! \, H_k(W^D_\alpha(\phi)), \quad \text{for any } \phi\in \mathcal H \text{ with }\|\phi\|_{\mathcal H}=1.
\]
Let $u=\{u(t,x); t\ge 0, x\in U\}$ be a random field such that $\mathbf E[u(t,x)^2]<\infty$ for all $t\ge 0$ and $x\in U$.
It is established in \cite{Nualart} that $u(t,x)$ has a Wiener chaos expansion of the form
\begin{equation}\label{eq:chaos expansion}
    u(t,x)=\mathbf E[u(t,x)]+\sum_{k=1}^{\infty} I_k(f_k(\cdot, t,x)),
\end{equation}
where $f_k$'s are symmetric elements of $\mathcal H^{\otimes k}$ uniquely determined by $u$ and where the series converges in $L^2(\Omega)$. When $u$ is the solution to equation \eqref{eq:pam-D} according to Definition \ref{def of Ito sol}, by a standard iteration procedure (borrowed from \cite{HuNualart,Hu15}) $u$ admits a chaos expansion as in \eqref{eq:chaos expansion} with $f_k$ given by 
\begin{equation}\label{eq:fk}
f_k(s_1,y_1,\ldots,s_k,y_k, t,x)=\frac{\beta^k}{k!} p^D_{t-s_{\sigma(k)}}(x, y_{\sigma(k)})\cdots  p^D_{s_{\sigma(2)}-s_{\sigma(1)}}(y_{\sigma(2)}, y_{\sigma(1)})  P^D_{s_{\sigma(1)}}u_0(y_{\sigma(1)}),
\end{equation}
where $\sigma$ denotes the permutation of $\{1,2,\cdots, k\}$ such that $0<s_{\sigma(1)}<\cdots <s_{\sigma(k)}<t$. In this setting we then have the following result.
\begin{prop}\label{uni-SHE}
A process $\mathcal{U}=\{u(t,x); (t,x)\in \mathbb{R}_+\times U\}$ solves equation \eqref{eq:PAM-mild form-Diri} in the sense of Definition \ref{def of Ito sol} if and only if for every $(t,x)$ the random variable $u(t,x)$ admits a chaos decomposition \eqref{eq:chaos expansion}-\eqref{eq:fk}, where the family $\{f_k(\cdot, t,x); k\geq1\}$ satisfies 
\begin{equation}\label{finite chaos}
    \sum_{k=1}^{\infty} k! \|f_k(\cdot, t,x)\|_{\mathcal H^{\otimes k}}^2<\infty.
\end{equation}
\end{prop}

\noindent
The next section is devoted to proving that \eqref{finite chaos} holds true for all $\alpha\geq0$, which will lead to existence and uniqueness for \eqref{eq:PAM-mild form-Diri}. 

\begin{remark}
In the range $\alpha >{d_h}/{2d_w}$, one can also resort to Feynman-Kac representations for the solutions. More specifically, if $\alpha>d_h/2d_w$ we have seen in Remark \ref{remk:regularity of the filed}  that the field $W^D_\alpha$ is  H\"older continuous in the space variable (see \eqref{density field}).
Hence under mild conditions on the initial condition $u_0$, the unique solution of \eqref{eq:pam-D} is given by the following Feynman-Kac formula with Wick exponential:
\begin{equation}\label{FK regular}
u(t,x)
=\mathbf{E}^B \left( u_0(B^x_t) \exp \left( \beta \int_0^t \dot{W}_\alpha^D(s,B^x_{t-s})ds -\frac{\beta^2}{2} \int_0^t G^D_{2\alpha} (B^x_s,B^x_s) ds \right) {1}_{(t < T)}\right) \, ,
\end{equation}
where $B$ is Brownian motion on $U$ independent from $W^D_\alpha$ started from $x$ and $T$ is its hitting time of the complement of $U$ in $X$. In \eqref{FK regular}, also note that $\mathbf{E}^B$ denotes the expectation computed with respect to $B$ only (later $\mathbf{E}^W$ will designate the expectation with respect to $W$ only). 
Moreover, thanks to \eqref{density field} the integral $\int_0^t \dot{W}_\alpha^D(s,B^x_{t-s})ds$ is understood as a Wiener integral  which is well defined for a fixed realization of $B^x$. The proof of formula \eqref{FK regular} follows from an approximation procedure explained in \cite[Section 3.2]{Hu15}, so we omit details for conciseness (see also the proof of Theorem \ref{th: Feynman-Kac moments}).

\end{remark}


\subsection{Existence, uniqueness of solutions and  \texorpdfstring{$L^2$}{L2} estimates}

As we have seen in Proposition \ref{uni-SHE},  existence and uniqueness for equation \eqref{eq:PAM-mild form-Diri} boils down to checking relation \eqref{finite chaos}. We start this section by getting some upper bounds in that direction.


\begin{lem}\label{lem:small t}
Consider $\alpha\ge 0$ and the noise $W^D_\alpha$ in Definition \ref{def-frac-Gaus-field} and \ref{def:st white}. The space $\mathcal{H}$ is given by \eqref{eq-hilb}. For any $\beta>0$ and $k\geq1$ the function $f_k$  is defined by \eqref{eq:fk}. Consider an additional parameter $\rho>0$, whose exact value will be calibrated later. Then the following holds true: 
 
\begin{enumerate}[label={\upshape(\roman*)}]
\item If $0\le \alpha<\frac{d_h}{2d_w}$, then for $t \ge 0$, $x\in U$ we have,
\begin{align}\label{eq-fk-est1}
    \|f_k(\cdot, t,x)\|_{\mathcal H^{\otimes k}}^2\le C^k e^{-2\lambda_1 t} \frac{\beta^{2k}}{k!} \|u_0\|_{\infty}^2 \left(\frac{1}{\rho}+ \frac{1}{\rho^{1+2\alpha -\frac{d_h}{2d_w}}} \right)^k e^{\rho t}.
\end{align}

\item  If $\alpha=\frac{d_h}{2d_w}$, then for $t \ge 0$, $x\in U$  it holds that
\begin{align}\label{eq-fk-est2}
\|f_k(\cdot, t,x)\|_{\mathcal H^{\otimes k}}^2\le  C^k e^{-2\lambda_1 t} \frac{\beta^{2k}}{k!} \|u_0\|_{\infty}^2 \frac{\ln \left(\max (3/2,\rho) \right)^{2k}} {\rho^k} e^{ \rho t}.
\end{align}

\item  If $\alpha>\frac{d_h}{2d_w}$, then for $t \ge 0$, $x\in U$ and $\rho >0$,
\begin{align}\label{eq-fk-est3}
\|f_k(\cdot, t,x)\|_{\mathcal H^{\otimes k}}^2\le  C^k e^{-2\lambda_1 t} \frac{\beta^{2k}}{k!} \|u_0\|_{\infty}^2 \frac{e^{\rho t}}{\rho^k}.
\end{align}

\end{enumerate}
The constants $C$ in each estimate above depend on $\alpha$,  $d_h$ and $d_w$.
\end{lem}

\begin{proof}We will perform most of our computations for the case $\alpha>0$. The case $\alpha=0$ results from simpler considerations and is left to the reader for the sake of conciseness. We now divide the proof in several steps.

\noindent{\it Step 1: Reduction to an integral recursion.}  By symmetry, in order to get \eqref{eq-fk-est1}-\eqref{eq-fk-est3}  we only need to evaluate the $L^2$-norm of $f_k(\cdot,t,x)$ on a particular time simplex $[0,t]^k_<:=\{(s_1,\dots, s_k): 0<s_1<\cdots<s_k<t\}$. For $(s_1,\dots, s_k)\in [0,t]^k_<$ we introduce the following notation:
\begin{align}\label{def:g_k}
g^D_k(s,y,t,x)=p_{t-s_k}^D(x,y_k)\cdots p_{s_2-s_1}^D(y_2, y_1),
\end{align}
where $y=(y_1, \dots, y_k)$ and $s=(s_1, \dots, s_k)$. 
By comparing \eqref{def:g_k} and \eqref{eq:fk} it is clear that on the simplex $[0,t]^k_<$ we have
\begin{align}\label{f_k H-norm bounds}
f_k(s,y,t,x)=\frac{\beta^k}{k!}g_k^D(s,y,t,x)P_{s_1}^Du_0(y_1).
\end{align}
Hence owing to Corollary \ref{bound_pt_diri_ini} and taking into account all possible orderings of $s_1,\dots,s_k$  we get
\begin{equation}\label{f0}
\|f_k(\cdot, t,x)\|_{\mathcal H^{\otimes k}}^2 \le C \frac{\beta^{2k}}{k!} M_k(t,x)\|u_0\|_{\infty}^2, 
\end{equation}
where $M_k(t,x)$ is given by 
\begin{equation}\label{lemma 3.8 ms 2}
M_k(t, x)=
\int_{[0,t]^k_<}\int_{U^{2k}}g_k^D(s,y,t,x)\prod_{i=1}^k G_{2\alpha}^D(y_i,y_i')g_k^D(s,y',t,x) 
e^{-2\lambda_1 s_1} d\mu(y)d\mu(y')ds.
\end{equation}
We now upper bound $M_k$ defined by \eqref{lemma 3.8 ms 2} thanks to a recursive procedure. To this aim, taking the expression~\eqref{def:g_k} of $g^D_k$ into account, observe that one can write \eqref{lemma 3.8 ms 2} as  
\begin{multline}\label{eq-Mk+1}
e^{2\lambda_1 t}M_{k+1}(t, x)=\int_0^t e^{2\lambda_1 (t-s_{k+1})} 
\int_{U^{2}}  p_{t-s_{k+1}}(x,y_{k+1})\\
\cdot G_{2\alpha}^D(y_{k+1},y'_{k+1}) \, p_{t-s_{k+1}}(x,y'_{k+1})\, e^{2\lambda_1 s_{k+1}}M_k (s_{k+1},y_{k+1}) d\mu(y_{k+1})d\mu(y'_{k+1})\, ds_{k+1},
\end{multline}
where we start the recursion with
$M_0(t,x):=e^{-2\lambda_1 t}$.
In the identity above, we shall bound the term $e^{2\lambda_1 s_{k+1}}M_k (s_{k+1},y_{k+1})$ by $N_{k}(s_{k+1})$, where we define
\begin{equation}\label{eq-Nk}
N_k(t)=e^{2\lambda_1 t} \sup_{x\in U}M_k(t, x).
\end{equation}
Hence resorting to Lemma \ref{Lem: kernel frac Lapl D}, we end up with the following relation:
\begin{equation}\label{recursive1}
N_{k+1}(t)\le \int_0^t \Psi(t-s) N_{k}(s)ds,
\quad\text{with}\quad
\Psi(t)=e^{2\lambda_1 t} \sup_{x\in U}\|(-\Delta_D)^{-\alpha}p_t^D(x,\cdot)\|_{L^2(U,\mu)}^2.
\end{equation}
Notice that one can initiate the recursion \eqref{recursive1} by observing that $N_0(t)=1$.

\noindent{\it Step 2: A renewal procedure.} To study the renewal inequality \eqref{recursive1}  and get the estimates we want to prove, we now appeal to an idea we learned from the proof of Theorem 1.1 in \cite{MR3359595} (see also \cite{MR2480553} for a related circle of ideas). Namely for our additional parameter $\rho >0$ let us denote
\begin{equation}\label{f1}
\cN_k(\rho)=\sup_{t \ge 0} e^{-\rho t}N_k(t).
\end{equation}
From \eqref{recursive1} we infer that
\[
e^{-\rho t} N_{k+1}(t)\le \int_0^t e^{-\rho (t-s)}\Psi(t-s) e^{-\rho s} N_{k}(s)ds 
\le 
\left( \int_0^t e^{-\rho (t-s)}\Psi(t-s) ds \right) \cN_k(\rho).
\]
Therefore we have
\begin{equation}\label{f2}
\cN_{k+1}(\rho) \le \hat{\Psi} (\rho) \, \cN_k(\rho) ,
\quad\text{with}\quad
\hat{\Psi} (\rho)=\int_0^{+\infty} e^{-\rho s} \Psi (s) ds .
\end{equation}
Since $\cN_{0}(\rho)=1$ we thus conclude by induction that $\cN_k(\rho) \le \hat{\Psi} (\rho)^k$.
Hence going back to the definition~\eqref{f1}, for every $t \ge 0$ we obtain
\begin{align}\label{eq-Nk-est}
 N_k(t)\le \hat{\Psi} (\rho)^k  e^{\rho t} .   
\end{align}

Having relation \eqref{eq-Nk} in mind, our bound on $M_k(t)$ is now reduced to an estimate on $\hat{\Psi}$ defined by~\eqref{f2}. We now upper bound the function $\hat{\Psi}$ according to the values of $\alpha$. It follows from Lemma \ref{estimate frac pt}
that the function $\Psi$ defined by \eqref{recursive1} satisfies
\begin{align}\label{eq-Psi}
\Psi(t) \le 
\begin{cases}
C  (1+t^{2\alpha-\frac{d_h}{d_w}}), \quad & 0 \le \alpha < \frac{d_h}{2d_w}, \\
C  |\ln \left(\min (1/2,t)\right)|^2, \quad  &\alpha=\frac{d_h}{2d_w}, \\
C , \quad &\alpha >\frac{d_h}{2d_w},
 \end{cases}
 \end{align}
 so that using  calculus
 \begin{align}\label{eq-Psi-hat}
     \hat{\Psi} (\rho) \le 
\begin{cases}
C  \left(\frac{1}{\rho}+ \frac{\Gamma\left(1+2\alpha -\frac{d_h}{d_w}\right)}{\rho^{1+2\alpha -\frac{d_h}{d_w}}} \right), \quad & 0 \le \alpha < \frac{d_h}{2d_w}, \\
C  \frac{\ln \left(\max (3/2,\rho) \right)^2} {\rho}, \quad &\alpha=\frac{d_h}{2d_w}, \\
\frac{C}{\rho} , \quad & \alpha >\frac{d_h}{2d_w}.
 \end{cases}
 \end{align}
Indeed, the cases $0 \le \alpha < \frac{d_h}{2d_w}$ and $\alpha>\frac{d_h}{2d_w}$ in \eqref{eq-Psi-hat} are easily justified by integrating~\eqref{eq-Psi} with an exponential weight and using basic identities for Gamma functions. Some details  about the case $\alpha=\frac{d_h}{2d_w}$ are provided at the end of the proof.

\noindent{\it Step 3: Conclusion.}  With \eqref{eq-Psi-hat} in hand, we plug this inequality into \eqref{eq-Nk-est} and then \eqref{eq-Nk} and \eqref{f0}. This proves \eqref{eq-fk-est1} for the case $\alpha<\frac{d_h}{2d_w}$. Inequalities \eqref{eq-fk-est2} and \eqref{eq-fk-est3} are proved along the same lines.

\noindent{\it Step 4: Some details for \eqref{eq-Psi-hat}.}
We now give some hints about deriving \eqref{eq-Psi-hat} in the case $\alpha=\frac{d_h}{2d_w}$. Namely if one integrates \eqref{eq-Psi} with the exponential weight $e^{-\rho t}$ we get
\begin{align}\label{eq-Psi-h-est}
\hat{\Psi}(\rho)&\le \int_0^{+\infty} e^{-\rho t} |\ln \left(\min (1/2,t)\right)|^2 dt= \int_0^{1/2} e^{-\rho t} (\ln t)^2 dt +C\int_{1/2}^{+\infty} e^{-\rho t}  dt \notag \\
 &\le \int_0^{1} e^{-\rho t} (\ln t)^2 dt +C\int_{0}^{+\infty} e^{-\rho t}  dt 
  \le  J_{01}+\frac{C}{\rho} \, ,
\end{align} 
where we have set $J_{01}= \int_0^{1} e^{-\rho t} (\ln t)^2 dt$. We now take care of the term $J_{01}$ by separating the small $\rho$ from the large $\rho$ regime. Namely consider an arbitrary $\rho_0>1$. Then for $\rho\le \rho_0$ we simply bound the term $e^{-\rho t}$ by $1$ in the integral defining $J_{01}$. We derive
\begin{align}\label{eq-int-erho-est}
\int_0^{1} e^{-\rho t} (\ln t)^2 dt \le C \int_0^{1} (\ln t)^2 dt \le C.
\end{align}
Next for $\rho>\rho_0$ we perform the elementary change of variable $u=\rho t$. We obtain
 \[
 \int_0^{1} e^{-\rho t} (\ln t)^2 dt=\frac{(\ln \rho)^2}{\rho} \int_0^\rho e^{-u} du-2 \frac{\ln \rho}{\rho} \int_0^\rho e^{-u} \ln u\,du+\frac{1}{\rho} \int_0^\rho e^{-u} (\ln u)^2du.
 \]
 Bounding all the integrals over $(0,\rho)$ above by (convergent) integrals over $(0,\infty)$, we end up with 
\begin{align}\label{eq-int-erho-bd}
  \int_0^{1} e^{-\rho t} (\ln t)^2 dt\le C\frac{(\ln \rho)^2}{\rho}.
\end{align}
Summarizing our computations, we gather \eqref{eq-int-erho-est} and \eqref{eq-int-erho-bd} in order to obtain 
\[
J_{01}\le C\frac{(\ln \rho)^2}{\rho}.
\]
Reporting this inequality into \eqref{eq-Psi-h-est}, this yields \eqref{eq-Psi-hat}. 
\end{proof}

We now state the existence and uniqueness result for the solution together with an estimate of the second moment.

\begin{thm}\label{existence PAM diri}
Let $\alpha \ge 0$ be a positive parameter and $W_\alpha^D$ be the Gaussian noise in Definitions~\ref{def-frac-Gaus-field} and \ref{def:st white}. For $\beta \ge 0$ we consider equation \eqref{eq:pam-D}, that is 
\begin{equation}\label{SHEtheot}
\partial_{t} u(t,x) = \Delta_D u(t,x) + \beta u(t,x) \, \dot{W}_\alpha^D(t,x),
\end{equation}
interpreted in the mild form \eqref{eq:PAM-mild form-Diri}. We assume that the initial condition $u(0,\cdot)=u_0$ lies in $L^\infty(U,\mu)$. Then \eqref{SHEtheot} admits a unique solution given by Definition \ref{def of Ito sol}. This solution moreover satisfies  the following $L^2$-upper bound estimates: for every $x \in U$
\begin{align}\label{eq-lim-u2-ub}
   \limsup_{t \to +\infty} \frac{\ln \mathbf E(u(t,x)^2)}{t} \le    \Theta_\alpha (\beta)-2\lambda_1,
 \end{align}
where $\lambda_1=\lambda_1(U)$ is the first eigenvalue of $-\Delta_D$ and $\Theta_\alpha$ is a continuous increasing function such that
\begin{enumerate}
\item If $0\le \alpha<\frac{d_h}{2d_w}$, then $\Theta_\alpha (\beta) \sim_{\beta \to 0} C\beta^2 $ and  $\Theta_\alpha (\beta) \sim_{\beta \to +\infty} C\beta^{\frac{2d_w}{(1+2\alpha)d_w-d_h}} $.
\item If $\alpha=\frac{d_h}{2d_w}$, then  $\Theta_\alpha (\beta) \sim_{\beta \to 0} C\beta^2 $ and $\Theta_\alpha (\beta) \sim_{\beta \to +\infty} C\beta^2 (\ln \beta)^2 $.

\item If $\alpha>\frac{d_h}{2d_w}$, then $\Theta_\alpha (\beta)=C\beta^2$.
\end{enumerate}
The constant $C>0$ in each estimate above is different and depends on $\alpha$, $d_h$ and $d_w$.
\end{thm}

\begin{proof}
The proof is an easy consequence of Lemma \ref{lem:small t}. We will only give details about the case $0\le \alpha<\frac{d_h}{2d_w}$, the other one being treated similarly. Also recall that if $\{f_k;k\ge1\}$ is the family of functions defined by \eqref{eq:fk}, then Proposition \ref{uni-SHE} states that existence-uniqueness is reduced to relation \eqref{finite chaos}:
\begin{equation}\label{f3}
\sum_{k=1}^{\infty} k! \|f_k(\cdot, t,x)\|_{\mathcal H^{\otimes k}}^2<\infty.
\end{equation}
Now observe that for $0\le \alpha<\frac{d_h}{2d_w}$, the function $f_k$ satisfies \eqref{eq-fk-est1} for an arbitrary parameter $\rho>0$. Next we consider the unique real number $\rho_c>0$ such that
\begin{align}\label{eq-rho-c}
   C \beta^2 F(\rho_c)=1,\quad \text{with} \quad F(\rho):=\frac{1}{\rho}+ \frac{1}{\rho^{1+2\alpha -\frac{d_h}{d_w}}},
\end{align}
where $C$ is the same constant as in \eqref{eq-fk-est1}. In the sequel we pick a parameter $\rho$ such that $\rho > \rho_c$. Therefore such a $\rho$ satisfies 
\[
\delta_\rho:= C\beta^2F(\rho)<1.
\]
Hence relation \eqref{eq-fk-est1} can be recast as 
 \begin{align}\label{eq-fk-rc}
\|f_k(\cdot, t,x)\|_{\mathcal H^{\otimes k}}^2 
\le \frac{e^{(\rho-2\lambda_1 )t}}{k!} \|u_0\|_\infty^2 \delta^k_\rho,
 \end{align} 
 from which it is easily seen that relation \eqref{f3} holds true. This finishes the proof of existence and uniqueness.  We now proceed to prove the $L^2$ bounds in our claim \eqref{eq-lim-u2-ub}. To this aim, notice that for the unique solution $u$ to \eqref{SHEtheot} we have
 \[
 \mathbf{E}(u(t,x)^2)=(P_t u_0)^2(x)+\sum_{k=1}^{+\infty} k!\|f_k(\cdot, t,x)\|_{\mathcal H^{\otimes k}}^2.
 \]
Plugging relation \eqref{eq-fk-rc} above and using Corollary \ref{bound_pt_diri_ini} we obtain
 \begin{equation}\label{eq-u-l2-est}
 \mathbf{E}(u(t,x)^2) 
 \le 
 c \|u_0\|_\infty^2 e^{(\rho-2\lambda_1)t} \sum_{k=0}^\infty \delta_\rho^k 
 \le 
 C_\rho \|u_0\|_\infty^2 e^{(\rho-2\lambda_1)t}.
 \end{equation}
This yields
 \begin{equation}\label{eq:limsup-l2}
 \limsup_{t \to +\infty} \frac{\ln \mathbf E(u(t,x)^2)}{t} \le \rho_c-2\lambda_1.
 \end{equation}
Now recall from \eqref{eq-rho-c} that $\rho_c$ satisfies $C\beta^2F(\rho_c)=1$, so that $\rho_c=F^{-1}((C\beta^2)^{-1})$. Furthermore we have 
\[
\lim_{\rho\to0}\rho F(\rho)=1,\quad \text{and}\quad \lim_{\rho\to\infty}\rho^{1+2\alpha-d_h/d_w}F(\rho)=1.
\]
Inverting those relations and plugging into \eqref{eq:limsup-l2}, the asymptotic behavior 
of $\Theta_\alpha$ in (i) follows. 

Finally we mention that for the case $\alpha=\frac{d_h}{2d_w}$, similar argument as above gives \eqref{eq:limsup-l2}, with the parameter $\rho_c$ satisfying 
\[
C\beta^2\frac{\ln \left(\max (3/2,\rho_c) \right)^2}{\rho_c}=1,
\]
where $C$ is the same constant as in \eqref{eq-fk-est2}. Consider $F(\rho)=\frac{\ln \left(\max (3/2,\rho) \right)^2}{\rho}$. We have   $F(\rho)\sim \frac{(\ln 3/2)^2}{\rho}$ when $\rho \to 0$ and $F(\rho)\sim \frac{(\ln \rho)^2}{\rho}$ when $\rho \to +\infty$. Hence we conclude the proof of (ii).
 \end{proof}

\subsection{\texorpdfstring{$L^p$}{Lp} upper bounds on the moments}

An analysis of the asymptotic behavior of $L^p$ moments usually brings useful information on the parabolic Anderson model. We derive an upper bound on those moments below.

\begin{thm}\label{thm:p-moment-upper-Diri}
Under the same conditions as in Theorem \ref{existence PAM diri}, let $u$ be the unique solution to~\eqref{SHEtheot}. Consider an integer $p>2$ and let $\Theta_\alpha$ be the function in \eqref{eq-lim-u2-ub}. Then for all $x\in U$ the following $L^p$-upper bound holds true:
\begin{align}\label{eq-Lp-ub}
  \limsup_{t \to +\infty} \frac{\ln \mathbf E(u(t,x)^p)}{t} \le \frac{p}{2}   \Theta_\alpha \left(\beta \sqrt{p-1}\right)-p\lambda_1.  
\end{align}
\end{thm}

\begin{proof}
As in the proof of Theorem \ref{existence PAM diri}, we will focus on the case $0\le \alpha<\frac{d_h}{2d_w}$. The other regime for $\alpha$ is treated in a similar way and left to the reader for sake of conciseness. 
Next we use the fact that in a fixed Wiener chaos the $L^p$-norm and the $L^2$-norm are equivalent, see the last line of \cite[page 62]{Nualart}. Namely,
\begin{align}\label{eq-C}
    \mathbf E \left( I_k(f_k(\cdot, t,x))^p \right)^{1/p} \le (p-1)^{k/2}\mathbf E \left( I_k(f_k(\cdot, t,x))^2 \right)^{1/2}.
\end{align}
Now in order to estimate the $L^p$-norm of $u(t,x)$ we resort to the decomposition \eqref{eq:chaos expansion} and Corollary \ref{bound_pt_diri_ini}. These allow to write
\[
\|u(t,x)\|_{L^p(\Omega)}  \le C \|u_0\|_\infty e^{-\lambda_1 t}+\sum_{k=1}^{\infty} \| I_k(f_k(\cdot, t,x))\|_{L^p(\Omega)}.
\]
Thus owing to \eqref{eq-C} we have 
\begin{align}\label{eq-u-lp}
\|u(t,x)\|_{L^p(\Omega)} 
   \le C \|u_0\|_\infty e^{-\lambda_1 t}+\sum_{k=1}^{\infty} (p-1)^{k/2}\| I_k(f_k(\cdot, t,x))\|_{L^2(\Omega)} .
  \end{align}
Plugging our estimate \eqref{eq-fk-est1} into \eqref{eq-u-lp} this yields
 \begin{align}\label{eq-u-lp-2}
\|u(t,x)\|_{L^p(\Omega)} 
   \le C \|u_0\|_\infty e^{-\lambda_1 t}+\sum_{k=1}^{\infty}   C^{k/2} (p-1)^{k/2} e^{\frac{\rho t}2-\lambda_1 t} \beta^k \|u_0\|_{\infty} \left(\frac{1}{\rho}+ \frac{1}{\rho^{1+2\alpha -\frac{d_h}{d_w}}} \right)^{k/2}.
 \end{align}
Hence for $p>2$ we define a constant $\rho_{c,p}$ similarly to \eqref{eq-rho-c}, as the unique positive real number satisfying
 \begin{align}\label{eq-rho-c-p}
     C (p-1) \beta^2 \left(\frac{1}{\rho_{c,p}}+ \frac{1}{\rho_{c,p}^{1+2\alpha -\frac{d_h}{d_w}}} \right)=1,
 \end{align}
 where $C$ is the constant in \eqref{eq-fk-est1}. We now pick a parameter $\rho>\rho_{c,p}$. With this value of $\rho$, the proof follows very closely the lines of Theorem \ref{existence PAM diri}. Namely for $\rho>\rho_{c,p}$ and similarly to \eqref{eq-u-l2-est} we obtain 
  \[
\|u(t,x)\|_{L^p(\Omega)}  \le C_\rho \| u_0\|_\infty e^{\frac{\rho t}{2}-\lambda_1 t }.
 \]
 Then we proceed as for the end of the proof for Theorem \ref{existence PAM diri}.
\end{proof}

\subsection{Feynman-Kac representation for the moments}
Inspired by \cite{BOTW, Hu15}, the lower bounds on moments for the solution $u$ to the stochastic heat equation will be obtained through Feynman-Kac formulae for moments. The current section is devoted to prove this type of formula, through a $L^2(\Omega)$-limit procedure relying on the following elementary lemma.


\begin{lem}\label{lem:L2 convergence}
Let $\{X_n\}_{n\ge1}$ be a sequence of $L^2$-random variables defined on a probability space $(\Omega,\mathcal{F},\mathbf{P})$. We assume that the following double limit exists:
\[
\lim_{m,n\to\infty}\mathbf{E}(X_nX_m)=a.
\]
Then $X_n$ converges in $L^2(\Omega)$.
\end{lem}

\noindent
We can now state our Feynman-Kac formula for moments, provided our noise $W_{\alpha}$ is smoother than white noise.

\begin{thm}\label{th: Feynman-Kac moments}
Consider a parameter $\alpha >0$ and let the conditions of Theorem \ref{existence PAM diri} prevail. Then the moments of the unique solution $u$ to \eqref{SHEtheot}  admit the following representation for $t\ge0$, $x\in U$ and $p\ge2$:
\begin{equation}\label{eq-p-mom-fk}
    \mathbf E (u(t,x)^p)=
    \mathbf{E}\left( \prod_{j=1}^p u_0(B^{x,j}_t) \exp\left\{ \beta^2 \sum_{1\le i,k \le p, \, i\not=k}\int_0^tG_{2\alpha}^D(B^{x,i}_s,B^{x,k}_s)ds\right\} 1_{(t<T^j)}\right).
\end{equation}
In formula \eqref{eq-p-mom-fk}, we have used a family $\{B^{x,j}, 1\le j\le p\}$ of independent Brownian motions which are initiated from $x\in U$ and are independent of the noise $W_\alpha$. For $1\le j\le p$ we also denote by $T^j$ the hitting time of $U^c$ in $X$.
\end{thm}

\begin{proof}
The proof of our claim \eqref{eq-p-mom-fk} follows from a standard procedure borrowed from~\cite{Hu15,HuNualart}. We just highlight the main steps below. In the remainder of the proof we write $W$ instead of $W_\alpha$ for the sake of notation. 

First consider a mollifier  $p_\varepsilon(x)$ in space. An example might be provided by $p_\varepsilon(x)=c\,p_\varepsilon^D(x)$, where $p_\varepsilon^D(x)$ is the Dirichlet heat kernel on $U$ and $c$ a normalizing constant such that $c\int_Up_\varepsilon^D(x)dx=1$. Also consider a mollifier $\tilde{p}_\delta(t)$ in time. Then a mollified noise can be defined for $\varepsilon,\delta>0$ as 
\begin{align}\label{eq-W-ep-del}
  W^{\delta,\varepsilon}(t,x)=(\tilde{p}_\delta p_\varepsilon)*{W}=W(\tilde{p}_\delta(t-\cdot)p_\varepsilon(x,\cdot)).  
\end{align}

Now consider equation \eqref{SHEtheot} driven by the smooth noise $W^{\delta,\varepsilon}$, interpreted in the Skorokhod sense. Similarly to \cite[Proposition 5.2]{HuNualart}, one can prove that this equation admits a unique solution $u^{\delta,\varepsilon}$. Moreover this solution can be written through a Feynman-Kac representation as 
\begin{align}\label{eq-F-K smoothed}
u^{\delta,\varepsilon}(t,x)
=\mathbf{E}^B\left(u_0(B^x_t)
\exp\left(\beta Z_{\delta,\varepsilon}^B(t,x)-\frac{\beta^{2}}{2}\mathbf{E}^W[Z^B_{\delta,\varepsilon}(t,x)]^2\right)1_{(t<T)}
\right).
\end{align}
In equation \eqref{eq-F-K smoothed}, $B$ is a Brownian motion independent of $W$ and $\mathbf{E}^B$ (respectively $\mathbf{E}^W$) denotes the expected value with respect to $B$ (respectively $W$) only. In addition, the random variable $Z_{\delta,\varepsilon}^B(t,x)$ is defined by
\begin{align}\label{eq-Z}
    Z_{\delta,\varepsilon}^B(t,x):= \int_0^tW^{\delta,\varepsilon}(t-s,B_s^x)ds=W\left(\int_0^t\tilde{p}_\delta((t-s)-\cdot)p_\varepsilon(B^x_s,\cdot)ds\right).
\end{align}
We shall now take limits in \eqref{eq-F-K smoothed} in order to get our representation of moments. We will focus on the case $p=2$ for sake of conciseness. The case $p>2$ is similar. 

We shall first take limits as $\delta\to0$ in the random variable $Z_{\delta,\varepsilon}^B(t,x)$. Namely for $\delta_1, \delta_2>0$ and $\varepsilon_1, \varepsilon_2>0$ and two independent Brownian motions $B^1$, $B^2$ initiated from $x\in U$, let us compute $\mathbf{E}^W\big(Z_{\delta_1,\varepsilon_1}^{B^1}(t,x)Z_{\delta_2,\varepsilon_2}^{B^2}(t,x)\big)$ (where we recall that  $\mathbf{E}^W$ designates the expected value with respect to $W$ only). We get
\begin{equation}\label{eq-Z-cov}
\begin{split}
    \mathbf E^W\left(Z_{\delta_1,\varepsilon_1}^{B^1}(t,x)Z_{\delta_2,\varepsilon_2}^{B^2}(t,x) \right)=&\,\mathbf E^W\left(\int_0^tW^{\delta_1,\varepsilon_1}(t-s,B_s^{1})ds\int_0^tW^{\delta_2,\varepsilon_2}(t-s,{B}_s^{2})ds\right)
    \\
    =&\int_{[0,t]^2} \int_{[0,\infty)\times U^2}\tilde{p}_{\delta_1}((t-s_1)-v) \tilde{p}_{\delta_2}((t-s_2)-v) \\
    &\times p_{\varepsilon_1}({B}^{1}_{s_1},z_1)p_{\varepsilon_2}({B}^{2}_{s_2},z_2)G_{2\alpha}^D(z_1,z_2)dvdz_1dz_2 ds_1ds_2,
\end{split}
\end{equation}
where we have resorted to the definition \eqref{eq-W-ep-del} of $ W^{\delta,\varepsilon}$ and formula \eqref{eq-Gau-cov} for the covariance of $W$. Now by smoothness of $p_\varepsilon$ one can easily take limits in \eqref{eq-Z-cov} in order to obtain
\begin{equation}\label{eq-Z-cov-limit}
\begin{split}
   & \lim_{\delta_1,\delta_2\to0}\mathbf E^W\left(Z_{\delta_1,\varepsilon_1}^{B^1}(t,x)Z_{\delta_2,\varepsilon_2}^{B^2}(t,x) \right) \\
   & =\int_{[0,t]\times U^2}p_{\varepsilon_1}({B}^{1}_{v},z_1)p_{\varepsilon_2}({B}^{2}_{v},z_2)G_{2\alpha}^D(z_1,z_2)dvdz_1dz_2.
    \end{split}
\end{equation}
Let us report the information about $Z_{\delta,\varepsilon}^B(t,x)$ into the Feynman-Kac representation \eqref{eq-Z}. To this aim we use the fact that $Z_{\delta,\varepsilon}^B(t,x)$ is a centered Gaussian random variable conditionally on $B$. Moreover, for a  centered Gaussian vector $(Y_1,Y_2)$ in $\R^2$ we have $\mathbf{E}(e^{Y_1}e^{Y_2})=\exp(\frac12 \mathbf{E}(Y_1Y_2))$. Therefore one can compute
\begin{equation}\label{convergence in L^2 second step}
\begin{split}
&\mathbf E \left(u^{\delta_1,\varepsilon_1}(t,x)u^{\delta_2,\varepsilon_2}(t,x)\right)\\
&=\mathbf{E}\left(u_0(B^{1}_t)u_0({B}^{2}_t)
\exp\left(\frac{\beta^{2}}{2}\mathbf E^W[Z_{\delta_1,\varepsilon_1}^{B^1}(t,x)Z_{\delta_2,\varepsilon_2}^{B^2}(t,x)]\right)1_{(t<T_1)}1_{(t<T_2)}\right) \, ,   
\end{split}
\end{equation}
where we notice that the diagonal terms involving $Z_{\delta,\varepsilon}^{B^{i}}$ cancel out with the martingale type term $\mathbf{E}^W[Z^B_{\delta,\varepsilon}(t,x)]^2$ in \eqref{eq-F-K smoothed}.

One can now take limits in \eqref{convergence in L^2 second step} as $\delta_1,\delta_2\to0$, like in \eqref{eq-Z-cov-limit}. Then similar limits in $\varepsilon_1,\varepsilon_2$ can be taken, where the arguments follow from \cite{Hu15}. We end up with
\begin{equation}\label{F-K after limit}
\begin{split}
&\lim_{\varepsilon_1,\delta_1,\varepsilon_2,\delta_2\to0}
\mathbf E  \left(u^{\delta_1,\varepsilon_1}(t,x)u^{\delta_2,\varepsilon_2}(t,x)\right) \\ &=\mathbf{E}\left(u_0(B^{1}_t)u_0({B}^{2}_t)
\exp\left(\frac{\beta^{2}}{2}{\int_0^t G_{2\alpha}^D(B_v^{1},{B}_v^{2})dv}\right)1_{(t<T_1)}1_{(t<T_2)}\right).
\end{split}
\end{equation}
In the above, we have used our assumption $\alpha>0$ so that  $G_{2\alpha}^D(B^1_v,B^2_v)$ is well-defined (even though it could blowup when $B^1_v=B^2_v$). Moreover, that the second line of \eqref{F-K after limit} remains finite is a consequence of Lemma \ref{lem:small t} (after Taylor-expanding the exponential).  
Therefore owing to Lemma \ref{lem:L2 convergence} the sequence $\{u^{\delta,\varepsilon}(t,x),\delta, \epsilon>0\}$ converges in $L^2$. Formula \eqref{eq-p-mom-fk} for $p=2$ is then a matter of standard considerations for which we refer again to~\cite{Hu15}. 
\end{proof}

\begin{remark}\label{Rm: positivity of sol}
    From the proof of Theorem \ref{th: Feynman-Kac moments}, we see that $u^{\delta,\varepsilon}$ converges to $u$ in $L^2$ as $\delta, \varepsilon\to 0$. On the other hand, we clearly have $u^{\delta,\varepsilon}\geq0$ from the Feynman-Kac representation~\eqref{eq-F-K smoothed} if one assumes $u_0\geq0.$ Therefore, we conclude that $u\geq0$ if $u_0\geq0$.
\end{remark}

\subsection{Lower bounds for the moments of the solution}\label{sec: lower bound for moments}

With the Feynman-Kac representation \eqref{eq-p-mom-fk} in hand, we are now ready to study the asymptotic behavior of $t\mapsto \mathbf{E}[u(t,x)^p]$ for $p\ge2$. More specifically, we are interested in lower bounds for the Lyapounov exponent $\lim_{t\to\infty}\frac{\ln \mathbf E(u(t,x)^p)}{t} $. Unlike the upper bound case, it will be important to separate the colored case $\alpha>0$ from the white noise case $\alpha=0$ due to differences in methodology.

\subsubsection{The case \texorpdfstring{$\alpha>0$}{alpha}}

We begin this section with a general estimate which gives the correct rate in terms of the temperature parameter $\beta$ in the case of a noise $W^D_\alpha$ which is more regular than white noise.

\begin{thm}\label{thm:p-moment-lower-Diri}
Let $\alpha, \beta$ be two strictly positive parameters. The noise $W^D_\alpha$ is given in Definition \ref{def-frac-Gaus-field}. Let $u$ be the unique solution to equation \eqref{SHEtheot}. In addition, we consider an open subset $A \subset U$ such that $\overline A \subset U$. The initial condition $u_0$ of \eqref{SHEtheot} is assumed to be in $L^\infty(U)$, nonnegative, and such that $u_0|_A \ge \varepsilon$ for a given $\varepsilon>0$. Then for $x\in A$ and  $t\ge0$ we have
\begin{align}\label{eq-up-lb}
    \mathbf E (u(t,x)^p) \ge \varepsilon^p \exp \left( C \beta^2 p(p-1)  t \right) (P^{D,A}_t 1)(x)^p, 
\end{align}
where $C$ is a positive  constant  that  depends on $A$ and $\alpha$ but not $\beta$ or $p$ and $P^{D,A}_t $ is the Dirichlet heat semigroup on $A$.
In particular, for $x$ in $A$
\begin{align}\label{eq-up-lb-asmp}
\liminf_{t \to +\infty} \frac{\ln \mathbf E(u(t,x)^p)}{t} \ge C  \beta^2  p(p-1)  -p\lambda_1 (A),
\end{align}
where $\lambda_1(A)>0$ is the first eigenvalue of $-\Delta_{D,A}$. 
\end{thm}

\begin{proof}

Let us start from the representation \eqref{eq-p-mom-fk} for $\mathbf{E}(u(t,x)^p)$. Let us call $T_A^i$ the hitting time of $A^c$ by the Brownian motion $B^i$. Using the fact that $u_0$ is positive and $1_{(t<T^j)}\ge 1_{(t<T_A^j)}$ we have
\begin{equation}\label{eq-up-1TA}
    \mathbf E (u(t,x)^p)\ge \mathbf{E}\left( \prod_{j=1}^p u_0(B_t^{x,j}) 1_{(t<T_A^j)} \exp\left\{\beta^2\sum_{1\le i,k \le p, i\not=k}\int_0^tG_{2\alpha}^D(B^{x,i}_s,B^{x,k}_s)ds\right\} \right).
   \end{equation}
Moreover since $A$ is a bounded set, the lower bounds in \eqref{two sided bounds G 1}, \eqref{two sided bounds G 2}, \eqref{two sided bounds G 3} imply
\begin{equation}\label{eq-exp-lb}
\exp \left(\beta^2\sum_{1\le i,k \le p,\, i\not=k}\int_0^tG_{2\alpha}^D(B^{x,i}_s,B^{x,k}_s)ds\right) 
    \ge
    \exp\left(\beta^2c \,\hat{m}^2 p(p-1)t\right),
\end{equation}
where $\phi_1$ is the first eigenvector for $\Delta_{D,U}$ and $\hat{m}$ is given by $\hat{m}=\inf\{\phi_1(y);y\in A\}$.
Since $\hat{m}=\hat{m}_A>0 $ according to Remark \ref{rmk: positivity of first ef}, we get that there exists $C=C_A$ such that
\begin{equation*}
    \exp \left(\beta^2 \sum_{1\le i,k \le p, \, i\not=k}\int_0^tG_{2\alpha}^D(B^{x,i}_s,B^{x,k}_s)ds\right)\prod_{j=1}^p 1_{(t<T_A^j)}
    \ge
    \exp\left(\beta^2 C p(p-1)t\right)\prod_{j=1}^p 1_{(t<T_A^j)}.
\end{equation*}
Plugging this inequality into \eqref{eq-up-1TA} we end up with the following lower bound, valid for $x\in A$ and $t$ sufficiently large,
\begin{align}
\label{eq-up-lb-d}
    \mathbf E (u(t,x)^p)
    &\ge
    \varepsilon^p \exp \left( C \beta^2 p(p-1)  t \right) \mathbf E \Big( \prod_{j=1}^p 1_{t<T_A^j}\Big) \notag\\
    &=
 \varepsilon^p \exp \left( C \beta^2 p(p-1)  t \right) \left[P^{D,A}_t 1 (x)\right]^p,
\end{align}
where we have invoked \eqref{def: Dirichlet semigroup} for the second identity. This proves our claim \eqref{eq-up-lb}.

In order to prove \eqref{eq-up-lb-asmp}, we argue that since the first eigenfunction of $\Delta_{D,A}$ is 
positive on $A$ (see Remark \ref{rmk: positivity of first ef} again),  it follows from the spectrum expansion that for $x\in A$ and $t$ sufficiently large 
\begin{align}\label{eq-Pt-DA-lb}
    (P_t^{D,A}1)(x) \ge Ce^{-\lambda_1(A)t}.
\end{align}
Therefore gathering \eqref{eq-up-lb-d} and \eqref{eq-Pt-DA-lb} we discover that
\[
 \mathbf E (u(t,x)^p) \ge C\varepsilon^p \exp \left( C \beta^2 p(p-1)  t-p\lambda_1(A)t \right),
\]
from which \eqref{eq-up-lb-asmp} is easily obtained. The proof is now complete. 
\end{proof}

The estimate \eqref{eq-up-lb-asmp} gives the correct order in $\beta$ for a noise $W_\alpha$ in the more regular regime $\alpha>d_h/2d_w$. We can get the following improvement in the range $0<\alpha \le \frac{d_h}{2d_w}$.

\begin{thm}\label{improved lower bound}
As in Theorem \ref{thm:p-moment-lower-Diri}, we consider the noise $W^D_\alpha$ in Definition \ref{def-frac-Gaus-field} with $\alpha>0$. Let $\beta>0$ and let also $u_0$ be a bounded and  nonnegative initial condition. Let $x\in U$, and assume that $u_0$ that is bounded from below by a positive constant on a neighborhood of $x$. The process $u$ is the unique solution to equation \eqref{SHEtheot}. Then one can improve the lower bounds on $\mathbf{E}(u(t,x)^p)$ from Theorem \ref{thm:p-moment-lower-Diri} in the following way:
\begin{itemize}
\item For $0<\alpha <\frac{d_h}{2d_w}$,
\begin{align}\label{eq-lap-lb-1}
    \liminf_{\beta \to +\infty}\liminf_{t \to +\infty}\frac{1}{\beta^{\frac{2d_w}{(1+2\alpha)d_w-d_h}} }\frac{\ln \mathbf E(u(t,x)^p)}{t} \ge C p(p-1).
\end{align}

\item For $\alpha =\frac{d_h}{2d_w}$,
\[
\liminf_{\beta \to +\infty}\liminf_{t \to +\infty}\frac{1}{\beta^{2} \ln \beta}\frac{\ln \mathbf E(u(t,x)^p)}{t} \ge C p(p-1).
\]
\end{itemize}
\end{thm}

\begin{proof}
We will only deal with the case $0<\alpha <\frac{d_h}{2d_w}$. The case $\alpha =\frac{d_h}{2d_w}$ is almost identical and left to the reader. As in \eqref{eq-up-1TA} we start by using the Feynman-Kac representation~\eqref{eq-p-mom-fk}. In this representation we lower bound the stopping time $T^j$ by another stopping time $T^j_\varepsilon$ defined in the following way: let $0<\varepsilon <\varepsilon_0<r_0$, where $r_0$ is such that $B(x,r_0) \subset U$ and $\varepsilon_0$ is chosen so that $u_0(y)\ge c_1^{1/p}>0$ on $B(x,\varepsilon_0)$. Then we set
\[
T_\varepsilon^j=\inf \{t \ge 0, B_t^j \notin B(x,\varepsilon) \}.
\]
Since $T^j>T^j_\varepsilon$ we deduce from \eqref{eq-p-mom-fk} that 
\begin{align*}
    \mathbf E (u(t,x)^p)
     &\ge c_1 \mathbf{E}\left( \prod_{j=1}^p  \exp\left\{\beta^2\sum_{1\le i, k\le p,i\not=k}\int_0^tG_{2\alpha}^D(B^{x,i}_s,B^{x,k}_s)ds\right\} 1_{(t<T_\varepsilon^j)}\right).
\end{align*}
Furthermore one can argue as in \eqref{eq-exp-lb} and use the lower bound in Proposition \ref{estimate G} (applied with $\alpha:=2\alpha$) in order to get the existence of a constant $c_2>0$ such that 
\begin{equation}\label{eq-lb-mid-j}
      \mathbf E (u(t,x)^p) \ge c_1 \mathbf{E}\left( \prod_{j=1}^p  \exp\left\{ c_2 \beta^2\sum_{1\le i,k\le p,i\not=k}\int_0^t\frac{ds}{d_U(B_s^{x,i},B_s^{x,k})^{d_h-2\alpha d_w}}\right\} 1_{(t<T_\varepsilon^j)}\right).
\end{equation}
Next we resort to some explicit lower bounds in \eqref{eq-lb-mid-j}. Namely for $s \le \min_j T^j_\varepsilon $ all the Brownian motions $B^j$ are confined in $B_U(x,\varepsilon)$. Hence we have $d_U(B_s^{x,i},B_s^{x,k}) \le 2 \varepsilon$ for all $i\not=k$. This yields
\begin{align}\label{eq-uplb-surv}
    \mathbf E (u(t,x)^p) \ge c_1 \exp\left( c_3p(p-1) \frac{\beta^2}{\varepsilon^{d_h-2\alpha d_w}}t  \right) \mathbf{P}\left(  T_\varepsilon^1 >t\right)^p.
\end{align}
In addition, using again classical spectral theory arguments we get for $t$ sufficiently large
\[
\mathbf{P}\left(  T_\varepsilon^1 >t\right) \ge c_4 e^{ - \lambda_1 (B(x,\varepsilon)) t}.
\]
Moreover, it is known that the sub-Gaussian bound on the heat kernel  \eqref{eq:subGauss-upper} implies that $\lambda_1 (B(x,\varepsilon))$ is bounded from above by $\frac{c}{ \varepsilon^{d_w}}$ (see \cite[Lemma 7.1]{mariano2022spectral}). We therefore obtain
\[
\mathbf{P}\left(  T_\varepsilon^1 >t\right)\ge c_4 e^{ -c_5  \frac{t}{\varepsilon^{d_w}}}.
\]
Plugging this inequality into \eqref{eq-uplb-surv} we end up with 
\begin{align}\label{eq-lb-est-ell}
    \mathbf E (u(t,x)^p) \ge c_6 \exp \left\{ \left( c_3p(p-1) \frac{\beta^2}{\varepsilon^{d_h-2\alpha d_w}} - p \frac{c_5}{\varepsilon^{d_w}} \right)t \right\}.
\end{align}
We now optimize \eqref{eq-lb-est-ell} over the values of $\varepsilon$. To this aim, recall from Assumption \ref{hyp:dirichlet-form} that $d_h-d_w<0$. Hence we also have $d_h-d_w<2\alpha d_w$ and thus $d_h-2\alpha d_w<d_w$. Therefore one can pick a constant $c_7>0$ large enough so that
\[
\frac{c_3 }{c_7^{d_h-2\alpha d_w}} -  \frac{c_5}{c_7^{d_w}} >0 \, .
\]
Next we choose $\varepsilon$ of the form $c_7\beta^\gamma$ with $\gamma>0$ so that the two terms in the exponent in~\eqref{eq-lb-est-ell} are of the same magnitude. This is achieved whenever
$
\frac{\beta^2}{\varepsilon^{d_h-2\alpha d_w}}\simeq\frac{1}{\varepsilon^{d_w}}$ which yields
$ \varepsilon\simeq
\beta^{-\frac{2}{(1+2\alpha)d_w-d_h}}.
$
Summarizing those considerations we will pick $\varepsilon$ of the form 
\begin{align}\label{choice of epsilon in Laplace}
\varepsilon= \frac{c_7}{\beta^{\frac{2}{(1+2\alpha)d_w-d_h}}},
\end{align}
and notice that for $\beta$ large enough we indeed have $\varepsilon<\varepsilon_0$. Plugging the value \eqref{choice of epsilon in Laplace} into~\eqref{eq-lb-est-ell}, this allows to conclude that for $\beta$ large enough
\[
\mathbf E (u(t,x)^p) \ge c_8 \exp \left\{ c_{9} p(p-1)\beta^{ \frac{2d_w}{(1+2\alpha)d_w-d_h}} t \right\},
\]
which easily implies our claim \eqref{eq-lap-lb-1}. 

Let us briefly give some hints about the case $\alpha =\frac{d_h}{2d_w}$. The only difference with \eqref{eq-lap-lb-1} is that the lower bound in Proposition \ref{estimate G} yields
\[
\mathbf E (u(t,x)^p) \ge c_6 \exp \left\{ \left( -c_3p(p-1) \beta^2  \ln \varepsilon - p \frac{c_5}{\varepsilon^{d_w}} \right)t \right\} \, .
\]
Then in the optimization step we now need to choose $\varepsilon =\frac{1}{\beta^{2/d_w}}$.
\end{proof}

\subsubsection{Lower bound for the moments in the case \texorpdfstring{$\alpha=0$}{alpha0}}

Our moment representation Theorem \ref{th: Feynman-Kac moments} was relying on the assumption $\alpha>0$. When $\alpha=0$, a similar representation would require a proper definition of intersection local times for independent Brownian motions on the uniform inner domain $U$. This is clearly out of the scope of the current article and is deferred to a subsequent publication. Our strategy to lower bound moments in the white noise case $\alpha=0$ is thus inspired by a direct analytic method spelled out in \cite{FoondunNualart}, with the caveat that it only works for $p=2$.

\begin{thm}\label{lower bound alpha zero} Let $\beta>0$ and consider an inner uniform domain $U$. Denote by $u$ the unique solution to \eqref{SHEtheot} driven by the white noise $W_0$ in Definition \ref{def:st white}. We assume the initial condition $u_0 $ to be bounded and nonnegative on $U$. Let $x \in U$ and assume that $u_0$ is bounded from below by a positive constant in a neighborhood of $x$. Recall that $\lambda_1=\lambda_1(U)$ designates the first eigenvalue of $-\Delta_{D,U}$. Then $u$ satisfies the following $L^2$-lower bound estimate:
\begin{align}\label{L^2 Lower bound white}
\liminf_{t \to +\infty} \frac{\ln \mathbf E(u(t,x)^2)}{t} \ge \Phi( \beta ^2)    -2\lambda_1,
\end{align}
where $\Phi$ is the inverse function of $y \to c_1 y e^{c_2y^{1/d_w}}$ and $c_1,c_2>0$ are positive constants that do not depend on $\beta$ but  depend on $x$.
\end{thm}

\begin{proof}
Let $x \in U$ and recall that we assume $u_0$ is bounded from below by a positive constant in a neighborhood of $x$. Since $U$ is specified in Definition~\ref{Def: inner uni domain} to be an open set, one can consider $r>0$ such that $B(x,2r) \subset U$ and such that $u_0(y) \ge c >0$ for all $y \in B(x,r)$. We now split the proof in a few steps. 

\noindent{\it Step 1: Deriving a renewal equation.}  Consider $y\in B(x, r)$ with $r$ as above. We compute the second moment of $u(t,y)$ from the mild formulation \eqref{eq:PAM-mild form-Diri} of the equation. Invoking the fact that the term $I(t,x)$ is centered and applying It\^{o}'s isometry with the inner product \eqref{eq-Gau-cov}, we get
    \begin{align}\label{thm 3.19 ms 0}
    \mathbf E\left(u(t,y)^2\right)  =\left(P_t^{D}u_0(y)\right)^2
    + \beta^2 \int_0^t \int_U p_{t-s}^{D}(y,z)^2 \, \mathbf E\left( u(s,z)^2\right)d\mu(z) ds.
    \end{align}
 Since $u_0(z)> c$ on $B(x,r)$, we thus get
 \[
     \mathbf E\left(u(t,y)^2\right) \ge c^2 \left(P_t^{D}1_{B(x,r)}(y)\right)^2
     + \beta^2 \int_0^t \int_{B(x,r/2)} p_{t-s}^{D}(y,z)^2 \, \mathbf E\left( u(s,z)^2\right)d\mu(z) ds.
 \]
 From Lemma \ref{lem:hk-lower-spectral}, this gives that for every $y \in B(x,r/2)$,
\[
\mathbf E\left(u(t,y)^2\right) \ge c^2 e^{-2\lambda_1 t}+ \beta^2 \int_0^t \int_{B(x,r/2)} p_{t-s}^{D}(y,z)^2 \mathbf E\left( u(s,z)^2\right)d\mu(z) ds.
\]
By denoting $\phi(t)= \inf_{y \in B(x,r/2)} \mathbf E\left(u(t,y)^2\right)$, we  get the following inequality
\begin{align}\label{thm 3.19 ms 1}
\phi(t) \ge c^2 e^{-2\lambda_1 t}+ \beta^2 \int_0^t \left( \inf_{y \in B(x,r/2)}\int_{B(x,r/2)} p_{t-s}^{D}(y,z)^2 d\mu(z) \right) \phi(s) ds.
\end{align}
We now invoke the lower bound \eqref{eq:subGauss-lowerD}, together with the fact that the eigenfunction $\Phi_1$ is strictly positive on $B(x,r/2)=B_U(x,r/2)$, in order to get 
\begin{align}\label{thm 3.19 ms 2}
\inf_{y \in B(x,r/2)}\int_{B(x,r/2)} p_{t-s}^{D}(y,z)^2 d\mu(z) & \ge c_1 \exp \left( -2\lambda_1 (t-s) -\frac{c_2}{(t-s)^{\frac{1}{d_w-1}}}\right).
\end{align}

Gathering \eqref{thm 3.19 ms 1} and \eqref{thm 3.19 ms 2}, we thus get the renewal identity
\begin{align}\label{thm 3.19 ms 3}
\psi(t) \ge c_1 +  c_2 \beta^2 \int_0^t \exp\left(-\frac{c_3}{(t-s)^{\frac{1}{d_w-1}}}\right) \psi(s) ds, \quad\mathrm{where}\ \ \psi(t)=e^{2\lambda_1 t} \phi(t).
\end{align}
It is then easily deduced from \eqref{thm 3.19 ms 3} that $\psi (t) \ge \zeta (t)$ where $\zeta (t)$ solves the renewal equation
\begin{align}\label{thm 3.19 ms 4}
\zeta(t) = c_1 + c_2 \beta^2\int_0^t \exp\left(-\frac{c_3}{(t-s)^{\frac{1}{d_w-1}}}\right)  \zeta(s) ds.
\end{align}

\noindent{\it Step 2: Asymptotic behavior of $\zeta$ in the renewal equation.}  Our lower bound \eqref{L^2 Lower bound white} is now reduced to a lower bound on the function $\zeta$ defined by~\eqref{thm 3.19 ms 4}. Furthermore, the estimate for $\zeta$ in \eqref{thm 3.19 ms 4} is based on Laplace transform techniques. That is, set $\hat{\zeta} (\rho)=\int_0^{+\infty} e^{-\rho t} \zeta (t) dt$ for all $\rho>0$. Applying Laplace transform on both sides of \eqref{thm 3.19 ms 4}  yields
\begin{align}\label{thm 3.19 ms 5}
\hat{\zeta}(\rho)=\frac{c_1}{\rho}+ c_2 \beta^2 \hat{\mu}(\rho) \hat{\zeta}(\rho),\quad \mathrm{where}\quad  
\hat{\mu}(\rho)=\int_0^{+\infty} \exp\left(-\rho t-\frac{c_3}{ t^{\frac{1}{d_w-1}}}\right) dt. 
\end{align}
Equation \eqref{thm 3.19 ms 5} is now easily solved algebraically and we end up with
\begin{align}\label{thm 3.19 ms 6}
\hat{\zeta}(\rho)=\frac{c_1}{\rho(1 -c_2\beta^2  \hat{\mu}(\rho)) }\, .
\end{align}
In order to ensure a proper definition for the right hand-side of \eqref{thm 3.19 ms 6}, we just want to make sure that $c_2\beta^2\hat{\mu}(\rho)<1$. This is achieved by monotonicity arguments whenever $\rho>\rho_c$, where $\rho_c$ is the unique positive number such that
\begin{align}\label{thm 3.19 ms 7}
\hat{\mu}(\rho_c)=\frac{1}{c_2\beta^2 }.
\end{align}
In addition, due to the fact that $\hat{\mu}$ is a differentiable function on $\{\rho>0\}$, there exists a constant $c>0$ such that
\begin{align}\label{thm 3.19 ms 7.5}
\lim_{\rho\downarrow \rho_c}(\rho-\rho_c)\hat{\zeta}(\rho)=c.
\end{align}
As explained e.g. in \cite{Korevaar}, the Wiener-Ikahara theorem for Laplace transforms with poles states that \eqref{thm 3.19 ms 7} implies
\begin{align}\label{thm 3.19 ms 8}
\liminf_{t \to +\infty} \frac{\ln \zeta (t)}{t}=\rho_c.
\end{align}

\noindent{\it Step 3: Estimating $\rho_c$.}  In order to estimate $\rho_c$ in \eqref{thm 3.19 ms 8} we go back to its definition \eqref{thm 3.19 ms 7}. Then we estimate $\hat{\mu}(\rho)$ as given in \eqref{thm 3.19 ms 5} in the following two regimes.

\noindent
\emph{(1) Case of a  large $\rho$.} In this situation we use Laplace's method. Namely denote $f_{\rho}(t):= \rho t+{c_3} t^{-\frac{1}{d_w-1}}$. Clearly it has unique global minimum at a point $t^{*}$ of the form $t^*=( \frac{c_3}{(d_w-1)\rho})^{(d_w-1)/d_w}$. Moreover we have
\begin{align}\label{eq-LM-ft*}
 &f_{\rho}(t^*)=C_1\rho^{\frac{1}{d_w}},\quad\text{and}\quad 
 f_{\rho}''(t^*)=C_2\rho^{2-\frac1{d_w}} \, , 
\end{align} 
where $C_1, C_2$ are positive constants that depend on $c_3$ and $d_w$. By Laplace's method we know that $\hat{\mu}(\rho)=\int_0^\infty e^{-f_{\rho}(t)}dt$ is comparable to $e^{-f_{\rho}(t^*)}\sqrt{\frac{2\pi}{|f_{\rho}''(t^*)|}}.$
Plugging  \eqref{eq-LM-ft*} into the above expression we find that for $\rho$ away from 0 we have some constants $c_4,c_5,c_6,c_7>0$ such that
\[
\frac{c_4}{\rho^{1-\frac{1}{2d_w}}} e^{-c_5 \rho^{1/d_w}} \le \hat{\mu}(\rho) \le \frac{c_6}{\rho^{1-\frac{1}{2d_w}}} e^{-c_7 \rho^{1/d_w}}.
\]

\noindent
\emph{(2) Case $\rho$ close to $0$.} For $\rho$ close to $0$, note that 
\begin{align}
  \hat{\mu}(\rho)\ge \int_1^\infty e^{-\rho t}e^{-c_3t^{-\frac{1}{d_w-1}}}dt \ge 
  e^{-c_3}\frac{1}{\rho}e^{-\rho} ,
 \quad\text{and}\quad 
 \hat{\mu}(\rho)\le \int_0^\infty e^{-\rho t}dt=\frac{1}{\rho} .
\end{align}
Therefore as $\rho\to0$, we end up with $\hat{\mu}(\rho) \simeq C\rho^{-1}$.

Summarizing the estimates for small and large $\rho$, we have thus obtained that there exist some positive constants $C_{1},C_{2},C_{3},C_{4}$ such that for all $\rho>0$ we have
\[
\frac{C_1}{\rho} e^{-C_2 \rho^{1/d_w}} \le \hat{\mu}(\rho) \le \frac{C_3}{\rho} e^{-C_4 \rho^{1/d_w}}.
\]
Reporting this bound in \eqref{thm 3.19 ms 7} yields the existence of a constant $C_{5}$ such that
\[
\beta^2 \le C_{5} \rho_c e^{C_2\rho_c^{{1}/{d_w}}} \, 
\quad\Longleftrightarrow\quad
\rho_{c} \ge \Phi(\beta^2) \, , 
\]
where  $\Phi$ denotes the inverse function of $x \to C_{5} x e^{C_2x^{1/d_w}}$. Inserting this inequality in~\eqref{thm 3.19 ms 8} we have
\begin{align}\label{eq-liminf-ge}
  \liminf_{t \to +\infty} \frac{\ln \zeta (t)}{t}=\rho_c\ge \Phi(\beta^2).  
\end{align}

\noindent{\it Step 4: Conclusion.}  We now conclude in the following way: recall that $\psi(t)\ge \zeta(t)$ and we have seen that $\zeta(t)$ verifies \eqref{eq-liminf-ge}. Therefore $\ln(\psi(t))/t$ verifies \eqref{eq-liminf-ge} as well. In addition, owing to~\eqref{thm 3.19 ms 3}, we have 
\[
\liminf_{t\to\infty} \frac{\phi(t)}{t}=\liminf_{t\to\infty}\frac{\psi(t)}{t}-2\lambda_1\ge \Phi(\beta^2)-2\lambda_1.
\]
Recalling that $\phi(t)=\inf_{y\in B(x,r/2)}\mathbf{E}(u(t,y)^2)$, this proves our claim \eqref{L^2 Lower bound white}.
\end{proof} 

When the boundary of a domain $A$ satisfies a cone-type condition, one can substantially improve the lower bound \eqref{L^2 Lower bound white}. This is summarized in the theorem below.

\begin{thm}\label{lower bound moment cusp}
Under the same conditions as for Theorem \ref{lower bound alpha zero} , let  $A$ be an open subset of $U$  such that $\overline A \subset U$ and such that $A$ satisfies the  hypothesis \eqref{no cusp}. We consider an initial condition $u_0 \in L^\infty(U)$ such that $u_0$ is non negative and $u_0 \ge \varepsilon >0$ on $A$. Then the solution $u$ to \eqref{SHEtheot} is such that for all $p\geq2$ and for all $x \in A$ we have
 \begin{align}\label{improved lower bound cusp}
\liminf_{t \to +\infty} \frac{\ln \mathbf E(u(t,x)^2)}{t} \ge c\beta^{\frac{2d_w}{d_w-d_h}}     -2\lambda_1,
\end{align}
where $c>0$ does not depend on $\beta$.
\end{thm}

\begin{proof}
The expression \eqref{thm 3.19 ms 0} for the second moment of $u(t,y)$ is still valid. Furthermore, we have $u_0(z)\geq\varepsilon 1_A(z)$. Hence we get
    \begin{align*}
    \mathbf E\left(u(t,y)^2\right) 
     \ge \varepsilon^2\left(P_t^{D}1_A(y)\right)^2+ \beta^2 \int_0^t \int_A p_{t-s}^{D}(y,z)^2 \mathbf E\left( u(s,z)^2\right)d\mu(z) ds.
    \end{align*}
    Similarly to what we did in the proof of Theorem \ref{lower bound alpha zero}, set 
    $$\phi(t)= e^{2\lambda_1 t} \inf_{x \in A} \mathbf E\left(u(t,x)^2\right).$$
    Along the same lines as for \eqref{thm 3.19 ms 1}-\eqref{thm 3.19 ms 2}-\eqref{thm 3.19 ms 3} and resorting to relation \eqref{jku}, which is valid whenever \eqref{no cusp} is fulfilled, we get that $\phi(t)\geq\zeta(t)$ where $\zeta$ solves the renewal equation 
\begin{align}\label{eq: zeta}
\zeta(t) = c_1 + c_2 \beta^2  \int_0^t \frac{1}{(t-s)^{d_h/d_w}} \zeta(s) ds.
\end{align}
Next we still follow closely the argument of Theorem \ref{lower bound alpha zero}, the main difference being that now we can compute explicitly the Laplace transform of $t\mapsto t^{-d_h/d_w}$. Namely applying Laplace transform on both sides of \eqref{eq: zeta} we get
\[
\hat{\zeta}(\rho)=\frac{c_1}{\rho}+c_2 \beta^2  \hat{\mu}(\rho) \hat{\zeta}(\rho),\quad\mathrm{where}\ \ \hat{\mu}(\rho)=\int_0^{+\infty} \frac{e^{-\rho t} }{t^{d_h/d_w}}dt =c_3 \rho^{-1+d_h/d_w}.
\]
This equation can be solved algebraically and we get
$$\hat{\zeta}(\rho)=\frac{c_1}{\rho^{d_h/d_w}\left(\rho^{1-d_h/d_w}-\beta^2c_2c_3\right)}.
$$
In particular, the positive pole for $\hat{\zeta}$ is of the form
$$\rho_c=c_4\beta^{\frac{2d_w}{d_w-d_h}}.$$
The same steps as for Theorem \ref{lower bound alpha zero} yield our claim \eqref{improved lower bound cusp}.
\end{proof}

\subsection{ Some properties of the solutions, Intermittency}\label{sec: property intermittency}
As mentioned in the introduction, when considering a Dirichlet type boundary condition in an inner uniform domain, the parabolic Anderson model exhibits an interesting phase transition for the behavior of moments. This is spelled out below. 

\begin{cor}\label{cor-interm}
Let $U$ be an inner uniform domain and $u_0  \in L^\infty(U)$ a non negative initial condition. We consider a noise $W_\alpha$ with $\alpha\ge0$ and the corresponding solution $u$  to the equation~\eqref{SHEtheot}. Let also $p\ge2$ and $x \in U$. Then there exists $\beta_{c,1}>0$ such that for all $\beta \le \beta_{c,1}$, we have 
\begin{align}\label{eq-p-mom-lim}
    \lim_{t \to +\infty} \mathbf{E} ( u(t,x)^p) =0.
\end{align}
If we assume moreover that $\alpha >0$, and that $u_0$ satisfies $u_0>\varepsilon >0$ in a neighborhood of $x$, then there exists $\beta_{c,2}>0$ such that for every $\beta \ge \beta_{c,2}$ \begin{align}\label{eq-p-mom-lim-l}
\lim_{t \to +\infty} \mathbf{E} ( u(t,x)^p) =+\infty.
\end{align}
\end{cor}
\begin{proof}
The first assertion \eqref{eq-p-mom-lim} is derived from the upper bound \eqref{eq-Lp-ub} in Theorem \ref{thm:p-moment-upper-Diri} for $p\ge2$ (see \eqref{eq-lim-u2-ub} for the particular case $p=2$). Specifically, identity \eqref{eq-p-mom-lim} is satisfied for a fixed $p$ as long as the right hand side
\[
h_{\alpha,\beta}:= \frac{p}{2}\Theta_\alpha(\beta \sqrt{p-1})-p\lambda_1
\]
verifies $h_{\alpha,\beta}<0$ for $\beta<\beta_{c,1}$. Now it is trivially checked from Theorem \ref{existence PAM diri} that for all $\alpha\ge0$
\[
\lim_{\beta\to0}\Theta_\alpha(\beta\sqrt{p-1})=0.
\]
We thus get our desired conclusion $h_{\alpha,\beta}<0$ for $\beta$ small enough. This ends the proof of~\eqref{eq-p-mom-lim}.

We now turn to the claim \eqref{eq-p-mom-lim-l}. Since we assume that $u_0>\varepsilon$ in a neighborhood of $x$ and $U$
 is an open set, one can work under the conditions of Theorem \ref{thm:p-moment-lower-Diri} with $A=B(x,r)$ for $r$ small enough. Also recall that we are considering $\alpha>0$. Hence the lower bound \eqref{eq-up-lb} holds true. Since the right hand side 
 \begin{align}\label{eq-G}
     G(\beta):= C\beta^2 p(p-1)-p\lambda_1(A)
 \end{align}
of \eqref{eq-up-lb} verifies $\lim_{\beta\to \infty}G(\beta)=\infty$, we clearly have $G(\beta)>0$ for $\beta>\beta_{c,2}$ and a large enough $\beta_{c,2}$. This yields relation \eqref{eq-p-mom-lim-l}.
\end{proof}

We now get closer to the intermittency phenomenon, by proving that large moments of $u(t,x)$ are divergent as $t\to\infty$.
\begin{cor}
Let the conditions of Corollary \ref{cor-interm} prevail. Consider $\alpha > 0$, $\beta > 0$ and $x\in U$ such that $u_0 >\varepsilon >0$ in a neighborhood of $x$. Then there exists $p_0$ such that  for all $p \ge p_0$ we have 
\[
\lim_{t \to +\infty} \mathbf{E} ( u(t,x)^p) =+\infty.
\]
\end{cor}
\begin{proof}
The proof is very similar to the proof of \eqref{eq-p-mom-lim-l} above. It is based on Theorem~\ref{thm:p-moment-lower-Diri} and the asymptotic behavior of the function $G$ in \eqref{eq-G}, now seen as a function of $p$. Details are left to the reader.
\end{proof}

One way to quantify intermittency for the parabolic Anderson model is to show that in large time the $L^p(\Omega)$-norm of $u(t,x)$ is much larger than the $L^q(\Omega)$-norm whenever $p>q$. We are now able to prove this for inner uniform domains.

\begin{cor}\label{cor-intem-1}
We still work under the conditions of Corollary \ref{cor-interm}, with $\alpha,\beta > 0$,  an inner uniform domain $U$,   and $x \in U$ such that $u_0 >\varepsilon$ in a neighborhood of $x$. Let $q \ge 2$. Then, there exists $p_0>q$ such that for $p \ge p_0$,
\begin{equation}\label{eq-intm-ratio}
    \lim_{t \to +\infty} \frac{\mathbf{E} ( u(t,x)^p)^{1/p}}{\mathbf{E} ( u(t,x)^q)^{1/q} } =+\infty.
\end{equation}
\end{cor}
\begin{proof}
Denote by $Q_{p,q}(t,x)$ the ratio $\|u(t,x)\|_{L^p(\Omega)}/\|u(t,x)\|_{L^q(\Omega)}$. Then
\begin{equation*}
    \frac{1}{t}\ln (Q_{p,q}(t,x))=\frac{1}{t}\left(\frac{1}{p}\ln(\mathbf{E} [u(t,x)^p])-\frac{1}{q}\ln(\mathbf{E} [u(t,x)^q]) \right).
\end{equation*}
    According to the lower bound in Theorem \ref{thm:p-moment-lower-Diri} (for a domain $A$ of the form $B(x,r)$ like in Corollary \ref{cor-intem-1}) and the upper bound in Theorem \ref{thm:p-moment-upper-Diri}, we obtain 
\begin{equation}\label{eq-asym-Q}
    \liminf_{t\to\infty}\frac{1}{t}\ln (Q_{p,q}(t,x))\ge
    C\beta^2(p-1)-\lambda_1(A)-\left(\frac{1}{2}\Theta_\alpha(\beta\sqrt{q-1})-\lambda_1\right).
\end{equation}
As a function of $p$, the dominant term in \eqref{eq-asym-Q} is $C\beta^2p$. Therefore if $p$ is large enough we get 
\[
\liminf_{t\to\infty}\frac{1}{t}\ln (Q_{p,q}(t,x))>0.
\]
This is enough to prove \eqref{eq-intm-ratio}.  
\end{proof}

\subsection{Scaling invariance in the Sierpi\'nski gasket  }\label{sec: scaling invariance} 

Up to now we have stated results and bounds which are valid for a wide range of underlying spaces $X$. When particularizing our study to sets with convenient invariances, the stochastic heat equation also exhibits nice scaling properties. In this section we examine the case of the Sierpi\'nski gasket with zero boundary condition at the vertices.

Let us thus go back to the setting of Section \ref{section gasket}, with the functions $\mathfrak f_i$ given by \eqref{eq-SK-f} and the Sierpi\'nski gasket $K$ defined by \eqref{eq-SK-K}. The inner uniform domain we consider here is $U=K\setminus V_0$. For a word $w=(i_1,\dots, i_n)$ in $\{1,2,3\}^n$ we define 
\begin{equation}\label{eq-U-w}
\mathfrak f_w=\mathfrak f_{i_1} \circ \cdots \circ \mathfrak f_{i_n},
\quad\text{and}\quad
    U_w:=\mathfrak f_w(K) \setminus \mathfrak f_w(V_0)=\mathfrak f_w(U).
\end{equation}
Since uniform curves in $U$ are transformed  by $\mathfrak f_w$ into uniform curves in $U_w$, it follows that every $U_w$ is a uniform and hence also a inner uniform  domain. With this additional notation in mind, we can now state our invariance result.

\begin{thm}\label{scaling Sierpinski}
    Let $\alpha, \beta$ be two positive parameters. On the inner uniform domain $U=K\setminus V_0$ we define the Dirichlet Laplace operator $\Delta_{D,U}$ as in Section \ref{sec-semi-g} and the noise $W^{D,U}_\alpha$ through Definition \ref{def-frac-Gaus-field}. We denote by $u$ the unique solution of 
\begin{equation}\label{eq-SDE-u-U}
\partial_{t} u(t,x) = \Delta_{D,U} u(t,x) + \beta u(t,x) \, \dot{W}_\alpha^{D,U}(t,x),\quad\text{for}\quad t\ge0, \,x\in U,
\end{equation}
with initial condition $u(0,\cdot)=1$. For a word $w=(i_1,\dots, i_n)\in \{1,2,3\}^n$ and $U_w$ defined by~\eqref{eq-U-w}, we similarly introduce $\Delta_{D, U_w}$ and $W^{D,U_w}_\alpha$. Furthermore we define $v$ as the unique solution to  
\begin{equation}\label{eq-SDE-v-U}
\partial_{t} v(t,x) = \Delta_{D,U_w} v(t,x) + \beta \left( \frac{5^{\alpha +1/2 }}{3} \right)^n v(t,x) \, \dot{W}_\alpha^{D,U_w}(t,x),\quad\text{for}\quad t\ge0, \,x\in U,
\end{equation}
with initial condition $v(0,\cdot)=1$. Then the following identity in the sense of finite dimensional distributions holds true:
\begin{align}\label{eq-v-u}
    (v(t,x))_{t \ge 0, x \in U_w }\stackrel{\mathrm{(fdd)}}{=} (u(5^n t,\mathfrak f^{-1}_w (x)))_{t \ge 0, x \in U_w}.
\end{align}
\end{thm}

\begin{proof}
We will prove the identity for one single value of $(t,x)$. The corresponding property for finite dimensional distributions is done in the same way and left to the reader. 
We have seen in Theorem \ref{Laplacian gasket} that $\Delta_{D,U}$ could be obtained as the extension of the Kigami Laplacian defined through the limiting procedure \eqref{discrete laplacian}-\eqref{eq-domain}. A similar construction can be performed for $\Delta_{D, U_w}$. Hence it is easily seen from \eqref{discrete laplacian} that for a generic Schwartz function $g$ (see Definition~\ref{def:a1} for the notion of Schwartz function) we have
\begin{align}\label{eq-Del-w-n}
    (\Delta_{D,U_w} g )\circ \mathfrak f_w= 5^n\Delta_{D,U} (g \circ \mathfrak f_w).
\end{align}
Moreover, our representation \eqref{eq:HKexpansion-diri} implies the corresponding relations for heat semigroups and heat kernels, valid for $t\ge0$ and $x,y\in U$:
\begin{align}\label{eq-Pt-U-w}
 (P_t^{D,U_w} g) \circ \mathfrak f_w=P_{5^nt}^{D,U}(g \circ \mathfrak f_w), 
 \quad\text{and}\quad
 p_t^{D,U_w}( \mathfrak f_w(x),\mathfrak f_w(y))=3^n p^{D,U}_{5^n t} (x,y).   
\end{align}
In the same way, one can invoke either \eqref{frac Laplacian spectrum} or \eqref{kernel for frac Laplacian} to see that 
\begin{align}\label{eq-Del-U-U-w}
    ((-\Delta_{D,U_w})^{-\alpha} g )\circ \mathfrak f_w= 5^{-n\alpha}(-\Delta_{D,U})^{-\alpha} (g \circ \mathfrak f_w).
\end{align}
Let us now specify the invariance for the noise $W_\alpha$. To begin with, from \eqref{eq-Del-U-U-w} one deduces that for $h\in L^2(U,\mu)$ and $\alpha\ge0$
\begin{align*}
\int_{U_w} \left[(-\Delta_{D,U_w})^{-\alpha} ( h \circ \mathfrak f^{-1}_w)\right]^2 d\mu 
& =5^{-2n\alpha} \int_{U_w}\left[((-\Delta_{D,U})^{-\alpha}h )\circ \mathfrak f^{-1}_w\right]^2 d\mu .
\end{align*}
Next we apply the scaling invariance of the Hausdorff measure on $K$, which stipulates that $\int_{U_w} f\circ f^{-1}_w d\mu=3^{-n}\int_U fd\mu$. We end up with
\begin{align}\label{eq-int-U-w}
\int_{U_w} \left[(-\Delta_{D,U_w})^{-\alpha} ( h \circ \mathfrak f^{-1}_w)\right]^2 d\mu 
=3^{-n}5^{-2n\alpha} \int_U ((-\Delta_{D,U})^{-\alpha}h )^2 d\mu,
\end{align}
and note that one can recast \eqref{eq-int-U-w} as 
\begin{align}\label{iso cell Sierpinski}
\| h \circ \mathfrak f^{-1}_w \|^2_{\mathcal{W}_D^{-\alpha}(U_w)}=3^{-n}5^{-2n \alpha} \|h  \|^2_{\mathcal{W}_D^{-\alpha}(U)}.
\end{align}
Taking into account the covariance function \eqref{kernel for frac Laplacian} of $W_\alpha^{D,U}$ and $W_\alpha^{D,U_w}$, we get the following equality in distribution for processes:
\begin{equation}\label{eq-cov-W-D-U-w}
    \lbrace W_\alpha^{D,U_w}(g \circ \mathfrak f^{-1}_w ); \, g\in W_D^{-\alpha}(U)\rbrace \stackrel{(d)}{=}
    \lbrace 3^{-n/2}5^{-n \alpha} W_\alpha^{D,U} (g);\,g \in \mathcal{W}_D^{-\alpha}(U)\rbrace.
\end{equation}
We are ready to translate the previous observations into invariances for the stochastic heat equation. That is start from the mild formulation of \eqref{eq-SDE-u-U}:
\[
u(t,x)=(P_t^{D,U}1)(x)+\beta\int_0^t \int_U p^{D,U}_{t-s}(x,y) u(s,y) W^{D,U}_\alpha (ds,dy).
\]
Composing the above relation with $\mathfrak f^{-1}_w$ and applying a Brownian scaling in time we get
\begin{align*}
u(5^n t,\mathfrak f_w^{-1} (x))&=(P_t^{D,U_w}1)(x)+\beta \int_0^{5^nt} \int_U p^{D,U}_{5^nt-s}(\mathfrak f_w^{-1}(x),y) u(s,y) W^{D,U}_\alpha (ds,dy) \\
 & =(P_t^{D,U_w}1)(x)+\beta 5^{n/2} \int_0^{t} \int_U p^{D,U}_{5^nt-5^n s}(\mathfrak f_w^{-1}(x),y) u(5^n s,y) \tilde{W}^{D,U}_\alpha (ds,dy), 
\end{align*}
where $\tilde{W}^{D,U}$ is a field such that $ \tilde{W}^{D,U}=W^{D,U}$ (specifically  $\tilde{W}^{D,U}(t,\cdot)$ is obtained as $5^{n/2}W^{D,U}(t/5^n,\cdot)$).  
We now resort to \eqref{eq-Pt-U-w} in order to write 
\begin{align*}
u(5^n t,\mathfrak f_w^{-1} (x))
 & =(P_t^{D,U_w}1)(x)+\beta 5^{n/2} 3^{-n} \int_0^{t} \int_U p^{D,U_w}_{t-s}( x,\mathfrak f_w(y)) u(5^ns,y) \tilde{W}^{D,U}_\alpha (ds,dy).
\end{align*}
Eventually, thanks to \eqref{eq-cov-W-D-U-w} we get the existence of a field $\tilde{W}^{D,U_w}$
 such that $\tilde{W}^{D,U_w} \stackrel{(d)}{=}{W}^{D,U_w}$ and 
\begin{align*}
u(5^n t,\mathfrak f_w^{-1} (x))
  = &\,(P_t^{D,U_w}1)(x)\\
  &+\beta \left( \frac{5^{\alpha +1/2 }}{3} \right)^n \int_0^{t} \int_{U_w} p^{D,U_w}_{t-s}(x,y) u(5^ns,\mathfrak f_w^{-1} (y)) \tilde{W}^{D,U_w}_\alpha (ds,dy).
\end{align*}
 Otherwise stated, $u(5^n t,\mathfrak f_w^{-1} (x))$ solves equation \eqref{eq-SDE-v-U}. This shows identity \eqref{eq-v-u} and finishes our proof.
\end{proof}

\section{Neumann parabolic Anderson model in uniform domains}\label{sec: Neumann PAM}

In this section we investigate the behavior of a parabolic Anderson model subject to Neumann boundary conditions.  Since many of the computations are similar to the Dirichlet case, we sill just sketch the main differences. To begin with, let us label the hypothesis on the domain $U$ considered below.

\begin{assump}
The domain $U$ is a uniform domain as specified in Definition \ref{def: uniform domain}. On this domain we consider the Neumann Laplace operator from Proposition \ref{prop:Sepctral decomp Neumann}.
\end{assump}

We now proceed to a description of the noise in Section \ref{sec:frac noise-Neumann}, and we state the main results about the stochastic heat equation in Sections \ref{sec: existence-uniqueness-Neumann} to \ref{sec: other properties-Neumann}.

\subsection{Fractional noise}\label{sec:frac noise-Neumann}

The family of noises we consider for the Neumann case are still described as white noises in time, with a spatial covariance based on the Neumann Laplacian.
\begin{defn}\label{def: noise-Neumann}
For $\alpha>0$ the noise $W^N_\alpha$ is introduced similarly to Definition \ref{def-frac-Gaus-field}, with a space $\mathcal{H}_\alpha$ given as
$\mathcal{H}_\alpha=L^2(\mathbb{R}_+, \mathcal{W}_N^{-\alpha}(U))$. Notice that the space $\mathcal{W}_N^{-\alpha}(U)$ is introduced in Definition  \ref{sobo_def_N}, and we suppose that the constant $C_U$ in \eqref{Sobo_norm_N} is such that  $G^{N,U}_{2\alpha}(x,y)+C_U$ is lower bounded by a positive constant (see Remark \ref{rmk regarding C_U}).  The Neumann fractional field $\{W_\alpha^N(\phi); \phi\in \mathcal{H}_{\alpha}\}$ is then the centered isonormal Gaussian family with covariance 
\begin{align}\label{inner product H neumann}
\mathbf E\left[ W_{\alpha}^N(\varphi) W_{\alpha}^N(\psi)\right]
=\int_{\R_+}\ \left\langle \varphi (t,\cdot) , \psi (t,\cdot)\right\rangle_{\mathcal{W}_N^{-\alpha}(U)}  dt,
\end{align}
It is assumed to be defined on a complete probability space $(\Omega, \mathcal{G}, \mathbf{P})$.
\end{defn} 

\begin{remark}For $\alpha=0$, the space-time white noise  $W^N_0$ is in fact defined exactly as in Definition~\ref{def:st white}. Otherwise stated, $W^N_0=W^D_0$. 

\end{remark}

\begin{remark}
As we did in Remark \ref{remk:regularity of the filed} for the Dirichlet case, let us mention that Neumann noises on fractals have been introduced in \cite{BaudoinLacaux}. 
For $\alpha \le {d_h}/{2d_w}$ it had been obtained that $W^N_\alpha$ is a space-time distribution,   while for $\alpha > {d_h}/{2d_w}$  a density property similar to \eqref{density field} holds true. 
\end{remark}

\subsection{Existence, uniqueness and \texorpdfstring{$L^p$}{Lp} upper bounds}\label{sec: existence-uniqueness-Neumann}

We are now in a position to state a result about the $L^p$ moment of the parabolic Anderson model with Neumann Laplacian.

\begin{thm}\label{existence unique neuma}
Let $\alpha, \beta \ge 0$ and let $W^N_\alpha$ be the noise from Definition \ref{def: noise-Neumann}.  We consider the equation 
\begin{equation}\label{SHEtheot-N}
\partial_{t} u(t,x) = \Delta_N u(t,x) + \beta u(t,x) \, \dot{W}_\alpha^N(t,x)
\end{equation}
interpreted in the mild sense as in \eqref{eq:PAM-mild form-Diri}. The initial condition $u_0$ lies in $L^\infty(U,\mu)$. Then the equation admits a unique solution $u$ according to Definition \ref{def of Ito sol}. Moreover, for any integer $p\geq2$, the solution $u$ satisfies the following $L^p$-upper bound estimates: 
\begin{align}\label{L^p upper Neumann}
\limsup_{t \to +\infty} \frac{\ln \mathbf E(u(t,x)^p)}{t} \le \frac{p}{2}   \widehat\Theta_\alpha \left(\beta \sqrt{p-1}\right),
\end{align}
where $\widehat\Theta_\alpha$ is a continuous increasing function whose asymptotic behavior is exactly the same as $\Theta_\alpha$ in Theorem \ref{existence PAM diri}.
\end{thm}

\begin{proof} We proceed using chaos expansions like in Theorem \ref{existence PAM diri}. We thus have to check a bound like \eqref{f3}, with $p_t^D$ replaced by $p_t^N$. In order to check \eqref{f3} we then adopt the same strategy as in Lemma \ref{lem:small t}. This leads to estimate a quantity like $M_k(t,x)$ in \eqref{lemma 3.8 ms 2}. Let us just highlight a small difference between the definition of the current $M_k(t,x)$ and the one we had in the Dirichlet case: relation \eqref{lemma 3.8 ms 2} involved the positive kernel $G_{2\alpha}^D$, which should be replaced by $G^N_{2\alpha}$. Now the problem is that $G_{2\alpha}^N$ is not positive anymore. Nevertheless as mentioned in Definition \ref{def: noise-Neumann}, Remark \ref{rmk regarding C_U} entails the existence of a constant $C_U$ such that $G_{2\alpha}+C_U\geq0$. With the definition~\eqref{inner product H neumann} of inner product in $\mathcal{H}$ and similarly to \eqref{lemma 3.8 ms 2} one can express $M_k(t,x)$ as
\begin{equation}\label{M_k Neumann}
\begin{split}
M_k(t, x)= &
\int_{[0,t]^k_<}\int_{U^{2k}}g^N(s,y,t,x) \\
&\times \prod_{i=1}^k (G_{2\alpha}^N(y_i,y_i')+C_U)g^N(s,y',t,x)d\mu(y)d\mu(y')e^{-2\lambda_1 s_1}ds,
\end{split}
\end{equation}
where the function $g^N$ is defined by 
\[
g^N(s,y,t,x)=p_{t-s_k}^N(x,y_k)\cdots p_{s_2-s_1}^N(y_2, y_1).
\]
Starting from \eqref{M_k Neumann}, the proof is completely similar to the proof of Lemma \ref{lem:small t}. We just mention that the bound for $\|(-\Delta_D)^{-\alpha}p_t^D\|_{L^2(U,\mu)}$ now has to be changed into a bound for $\left\|\left[(-\Delta_{N,U})^{-\alpha}+C_U \right]p_t^N\right\|_{L^2(U,\mu)}$. However, it is easy to see from the proof of Lemma~\ref{estimate frac pt neumann} that   
  \[
 \sup_{x \in U} \left\|\left[(-\Delta_{N,U})^{-\alpha}+C_U\right]p_t^{N,U}(x,\cdot)\right\|_{L^2(U,\mu)} \le
 \begin{cases}
C ( t^{\alpha-\frac{d_h}{2d_w}}+1), \quad 0<\alpha<\frac{d_h}{2d_w}
\\
C  |\ln (\max (t,1/2))|, \quad \alpha=\frac{d_h}{2d_w} \\
C ,\quad \alpha >\frac{d_h}{2d_w},
 \end{cases}
  \]
The bound \eqref{L^p upper Neumann} is thus obtained with the same estimates as in Lemma \ref{lem:small t} and Theorem~\ref{existence PAM diri}. Details are left to the patient reader. 
\end{proof}

\subsection{\texorpdfstring{$L^p$}{Lp} lower bounds}
In this section we turn to an analog of the lower bounds in Section \ref{sec: lower bound for moments}. We start with a general lower bound in the case $\alpha>0$.

\begin{thm}\label{th: Neumann-lower bound}
Let  $\beta \ge 0$ and $\alpha > 0$,  and consider the noise $W^N_\alpha$ given in Definition~\ref{def: noise-Neumann}. Let $u$ be the solution to equation \eqref{SHEtheot-N}, where the initial condition $u_0$ is an element of $L^\infty(U,\mu)$  such that $u_0 \ge \varepsilon>0$ on $U$.  Let $p$ be an integer with $p\geq2$. Then  for $x \in U$ and  $t \ge 0$ we have
\begin{align}\label{Neumann-lower bound}
 \mathbf E (u(t,x)^p) \ge \varepsilon^p \exp \left( C \beta^2 p(p-1)  t \right),
\end{align}
where $C$ is a positive  constant. In particular, for $x \in U$ 
\begin{align}\label{Neumann-lower bound asymp}
\liminf_{t \to +\infty} \frac{\ln \mathbf E(u(t,x)^p)}{t} \ge C  \beta^2  p(p-1) .
\end{align}
\end{thm}

\begin{proof}
From the Feynman-Kac representation of the  moments, which is established along the same lines as in the Dirichlet boundary case, one has
\begin{align*}
    \mathbf E (u(t,x)^p)=\mathbf{E}\left( \prod_{j=1}^p u_0(\hat B_t^{x,j}) \exp\left\{\beta^2\sum_{i\not=k}^p\int_0^t\left(G_{2\alpha}^N(\hat B^{x,i}_s,\hat B^{x,k}_s)+C_U\right)ds\right\} \right),
\end{align*}
where $\hat B$ is the reflected Brownian motion on $U$. Otherwise stated, $\hat{B}$ is the diffusion with generator $\Delta_N$.
Replacing $G_{2\alpha}^D$ by $G_{2\alpha}^N$ in the computations going from \eqref{eq-exp-lb} to \eqref{eq-up-lb-d}, we easily get that
\[
\mathbf E \left(u(t,x)^p\right)\geq \varepsilon^p \exp \left( C\beta^2p(p-1)t\right)\left(P_t^{N,U} 1(x)\right)^p.
\]
In addition, for Neumann boundary conditions we have $P^{N,U}_t1(x)=1$, from which \eqref{Neumann-lower bound} is easily deduced. Relation  \eqref{Neumann-lower bound asymp} is a trivial consequence of \eqref{Neumann-lower bound}.
\end{proof}

\begin{remark}
In the Dirichlet case, we had observed a phase transition for moments in Corollary~\ref{cor-interm}. This phase transition is absent from the Neumann boundary picture. Indeed relation~\eqref{Neumann-lower bound} implies that
$$\lim_{t\to\infty}\mathbf E \left(u(t,x)^p\right)=\infty,\quad\mathrm{for\ all}\ \beta>0.$$
This difference in the asymptotic behavior stems from the fact that $\Delta_{D,U}$ has a strictly negative principal eigenvalue $\lambda_1$.
\end{remark}

As in the Dirichlet case, one can improve the bounds of Theorem \ref{th: Neumann-lower bound} when the noise is rough enough. This is summarized in the theorem below.

\begin{thm}\label{th:Neumann-lower bound improved}
Under the same conditions as in Theorem \ref{th: Neumann-lower bound}, let $u$ be the solution to equation~\eqref{SHEtheot-N}. The initial condition $u_0$ is bounded, non-negative  on $U$ and bounded from below by a positive constant on the neighborhood of a given $x\in U$. Then, for an integer $p \ge 2$,  we have the following lower bounds. 
\begin{itemize}
\item For $0<\alpha <\frac{d_h}{2d_w}$,
\[
\liminf_{\beta \to +\infty}\liminf_{t \to +\infty}\frac{1}{\beta^{\frac{2d_w}{(1+2\alpha)d_w-d_h}} }\frac{\ln \mathbf E(u(t,x)^p)}{t} \ge C p(p-1),
\]
\item For $\alpha =\frac{d_h}{2d_w}$,
\[
\liminf_{\beta \to +\infty}\liminf_{t \to +\infty}\frac{1}{\beta^{2} \ln \beta}\frac{\ln \mathbf E(u(t,x)^p)}{t} \ge C p(p-1).
\]
\end{itemize}
\end{thm}
\begin{proof}
The proof is very similar to the proof of Theorem \ref{improved lower bound}. In particular, notice that the equivalent of \eqref{eq-uplb-surv} still involves the exit time $T^1_\varepsilon$. We thus end up with \eqref{eq-lb-est-ell} as before. From there, the computations proceed as in Theorem \ref{improved lower bound}.
\end{proof}

We now label the result for the space-time white noise case. Its proof goes along the same lines as for Theorem \ref{lower bound moment cusp}.

\begin{thm}\label{lower bound alpha zero-N}
Let the conditions of Theorem \ref{th: Neumann-lower bound} and \ref{th:Neumann-lower bound improved} prevail. We assume that we are in the white noise case $\alpha=0$. Moreover, the initial condition $u_0$ is assumed to be in $L^\infty(U,\mu)$ and lower bounded by $\varepsilon>0$. Then the solution $u$ to \eqref{SHEtheot-N} satisfies the following $L^2$-lower bound estimate:
\[
\liminf_{t \to +\infty} \frac{\ln \mathbf E(u(t,x)^2)}{t} \ge c \beta^{\frac{2d_w}{d_w-d_h}},\quad \mathrm{for\ all}\ x\in U,
\] 
where $c>0$ is a positive constant that does not depend on $\beta$.
\end{thm}

\subsection{Scaling invariance in the Sierpi\'nski gasket}\label{sec: other properties-Neumann}

In this section we go back to the setting of Section \ref{sec: scaling invariance}  and we establish scaling factors in the Sierpinski gasket $K$. Our findings are summarized below.


\begin{thm}\label{scaling Sierpinski neumann}
Let the notation and assumptions of Theorem \ref{scaling Sierpinski} prevail. On the uniform domain $U=K\setminus V_0$, we consider the unique solution $u$ to 
\begin{equation*}
\partial_{t} u(t,x) = \Delta_{N,U} u(t,x) + \beta u(t,x) \, \dot{W}_\alpha^{N,U}(t,x),\quad \mathrm{for}\ t\geq0, x\in U,
\end{equation*}
with initial condition $u(0,\cdot)=1$. For the set $U_w$ defined by \eqref{eq-U-w}, we also consider the solution $v$ to
\begin{equation}\label{eq: v}
\partial_{t} v(t,x) = \Delta_{N,U_w} v(t,x) + \beta \left( \frac{5^{\alpha +1/2 }}{3} \right)^n v(t,x) \, \dot{W}_\alpha^{N,U_w}(t,x),\quad \mathrm{for}\ t\geq0, x\in U_w,
\end{equation}
with initial condition $v(0,\cdot)=1$. 
Then the following identity in the sense of finite dimensional distribution  holds true:
\begin{align*}
    (v(t,x))_{t \ge 0, x \in U_w }\stackrel{(f\!dd)}{=} (u(5^n t,\mathfrak f^{-1}_w (x)))_{t \ge 0, x \in U_w}.
\end{align*}
\end{thm}
\begin{proof}
The proof goes along the same lines as for Theorem \ref{scaling Sierpinski}. Details are spared for sake of conciseness. Let us just highlight a subtle change in the definition of the kernel $G_{2\alpha}^{N,U_w}$ related to \eqref{eq: v}. Indeed, in the current situation we have
\begin{align*}
\int_{U_w} (-\Delta_{N,U_w})^{-\alpha} ( g \circ \mathfrak f^{-1}_w)^2 d\mu =3^{-n}5^{-2n\alpha} \int_U ((-\Delta_{N,U})^{-\alpha}g )^2 d\mu.
\end{align*}
Hence if one wishes \eqref{eq-int-U-w} to hold in our context with the operator $\left(-\Delta_{N,U_w}\right)^{-\alpha}$ replaced by 
$\left(-\Delta_{N,U_w}\right)^{-\alpha}+C_{U_w}$, we need to choose
$
C_{U_w}=3^n5^{-2n\alpha}C_U.
$  
\end{proof}


\bibliographystyle{abbrv}
\bibliography{bibfile}

\begin{thebibliography}{10}

\bibitem{balsam2023density}
H.~{Balsam}, K.~{Kaleta}, M.~{Olszewski}, and K.~{Pietruska-Pa{\l}uba}.
\newblock {Density of states for the Anderson model on nested fractals}.
\newblock {\em arXiv:2303.05980}, 2023.

\bibitem{Barlow}
M.~T. Barlow.
\newblock Diffusions on fractals.
\newblock In {\em Lectures on probability theory and statistics
  ({S}aint-{F}lour, 1995)}, volume 1690 of {\em Lecture Notes in Math.}, pages
  1--121. Springer, Berlin, 1998.

\bibitem{BaudoinChen}
F.~Baudoin and L.~Chen.
\newblock Dirichlet fractional {G}aussian fields on the {S}ierpinski gasket and
  their discrete graph approximations.
\newblock {\em Stochastic Process. Appl.}, 162:593--616, 2023.

\bibitem{BCHOTW}
F.~Baudoin, L.~Chen, C.-H. Huang, C.~Ouyang, S.~Tindel, and J.~Wang.
\newblock Parabolic {A}nderson model on {H}eisenberg groups: weighted {B}esov
  spaces and {S}tratonovich setting.
\newblock {\em preprint}, 2023.

\bibitem{BK}
F.~Baudoin and D.~J. Kelleher.
\newblock Differential one-forms on {D}irichlet spaces and {B}akry-\'{E}mery
  estimates on metric graphs.
\newblock {\em Trans. Amer. Math. Soc.}, 371(5):3145--3178, 2019.

\bibitem{BaudoinLacaux}
F.~Baudoin and C.~Lacaux.
\newblock Fractional {G}aussian fields on the {S}ierpi\'{n}ski gasket and
  related fractals.
\newblock {\em J. Anal. Math.}, 146(2):719--739, 2022.

\bibitem{BOTW}
F.~Baudoin, C.~Ouyang, S.~Tindel, and J.~Wang.
\newblock Parabolic {A}nderson model on {H}eisenberg groups: the {I}t\^{o}
  setting.
\newblock {\em J. Funct. Anal.}, 285(1):Paper No. 109920, 44, 2023.

\bibitem{Chen-Dirichlet}
Z.~Q. Chen.
\newblock On reflected {D}irichlet spaces.
\newblock {\em Probab. Theory Related Fields}, 94(2):135--162, 1992.

\bibitem{CY}
F.~Comets and N.~Yoshida.
\newblock Directed polymers in random environment are diffusive at weak
  disorder.
\newblock {\em The Annals of Probability}, 34(5):1746--1770, 2006.

\bibitem{CSZ}
C.~Cosco, I.~Seroussi, and O.~Zeitouni.
\newblock Directed polymers on infinite graphs.
\newblock {\em Communications in mathematical physics}, 386(1):395--432, 2021.

\bibitem{Da}
R.~Dalang.
\newblock Extending the martingale measure stochastic integral with
  applications to spatially homogeneous spde's.
\newblock {\em Electronic Journal of Probability}, 4:1--29, 1999.

\bibitem{Dalang-Quer}
R.~C. Dalang and L.~Quer-Sardanyons.
\newblock Stochastic integrals for spde's: a comparison.
\newblock {\em Expositiones Mathematicae}, 29(1):67--109, 2011.

\bibitem{MR0793651}
L.~de~Haan and U.~Stadtm\"{u}ller.
\newblock Dominated variation and related concepts and {T}auberian theorems for
  {L}aplace transforms.
\newblock {\em J. Math. Anal. Appl.}, 108(2):344--365, 1985.

\bibitem{MR2480553}
M.~Foondun and D.~Khoshnevisan.
\newblock Intermittence and nonlinear parabolic stochastic partial differential
  equations.
\newblock {\em Electron. J. Probab.}, 14:no. 21, 548--568, 2009.

\bibitem{FoondunNualart}
M.~Foondun and E.~Nualart.
\newblock On the behaviour of stochastic heat equations on bounded domains.
\newblock {\em ALEA Lat. Am. J. Probab. Math. Stat.}, 12(2):551--571, 2015.

\bibitem{FOT}
M.~Fukushima, Y.~\={O}shima, and M.~Takeda.
\newblock {\em Dirichlet forms and symmetric {M}arkov processes}, volume~19 of
  {\em De Gruyter Studies in Mathematics}.
\newblock Walter de Gruyter \& Co., Berlin, 1994.

\bibitem{MR2807275}
P.~Gyrya and L.~Saloff-Coste.
\newblock Neumann and {D}irichlet heat kernels in inner uniform domains.
\newblock {\em Ast\'{e}risque}, (336):viii+144, 2011.

\bibitem{haeseler2011heat}
S.~Haeseler.
\newblock Heat kernel estimates and related inequalities on metric graphs.
\newblock {\em arXiv:1101.3010}, 2011.

\bibitem{hairer2023regularity}
M.~Hairer and H.~Singh.
\newblock Regularity structures on manifolds and vector bundles, 2023.

\bibitem{Hu15}
Y.~Hu, J.~Huang, D.~Nualart, and S.~Tindel.
\newblock Stochastic heat equations with general multiplicative {G}aussian
  noises: {H}\"{o}lder continuity and intermittency.
\newblock {\em Electron. J. Probab.}, 20:no. 55, 50, 2015.

\bibitem{HuNualart}
Y.~Hu and D.~Nualart.
\newblock Stochastic heat equation driven by fractional noise and local time.
\newblock {\em Probab. Theory Related Fields}, 143(1-2):285--328, 2009.

\bibitem{MR3359595}
D.~Khoshnevisan and K.~Kim.
\newblock Non-linear noise excitation and intermittency under high disorder.
\newblock {\em Proc. Amer. Math. Soc.}, 143(9):4073--4083, 2015.

\bibitem{KK2015}
D.~Khoshnevisan and K.~Kim.
\newblock Non-linear noise excitation of intermittent stochastic pdes and the
  topology of lca groups.
\newblock {\em Annals of Probability}, 43(4):1944--1991, 2015.

\bibitem{Kigami1993}
J.~Kigami.
\newblock Harmonic calculus on p.c.f. self-similar sets.
\newblock {\em Trans. Amer. Math. Soc.}, 335(2):721--755, 1993.

\bibitem{Kigami}
J.~Kigami.
\newblock {\em Analysis on fractals}, volume 143 of {\em Cambridge Tracts in
  Mathematics}.
\newblock Cambridge University Press, Cambridge, 2001.

\bibitem{Korevaar}
J.~Korevaar.
\newblock A century of complex {T}auberian theory.
\newblock {\em Bull. Amer. Math. Soc. (N.S.)}, 39(4):475--531, 2002.

\bibitem{La10}
H.~Lacoin.
\newblock New bounds for the free energy of directed polymers in dimension 1+ 1
  and 1+ 2.
\newblock {\em Communications in Mathematical Physics}, 294(2):471--503, 2010.

\bibitem{Lierl}
J.~Lierl.
\newblock The {D}irichlet heat kernel in inner uniform domains in fractal-type
  spaces.
\newblock {\em Potential Anal.}, 57(4):521--543, 2022.

\bibitem{MR3170207}
J.~Lierl and L.~Saloff-Coste.
\newblock The {D}irichlet heat kernel in inner uniform domains: local results,
  compact domains and non-symmetric forms.
\newblock {\em J. Funct. Anal.}, 266(7):4189--4235, 2014.

\bibitem{Lindstrom}
T.~Lindstr{\o}m.
\newblock Brownian motion on nested fractals.
\newblock {\em Mem. Amer. Math. Soc.}, 83(420):iv+128, 1990.

\bibitem{lodhia2016fractional}
A.~Lodhia, S.~Sheffield, X.~Sun, and S.~S. Watson.
\newblock Fractional {G}aussian fields: a survey.
\newblock {\em Probab. Surv.}, 13:1--56, 2016.

\bibitem{mariano2022spectral}
P.~Mariano and J.~Wang.
\newblock Spectral bounds for exit times on metric measure dirichlet spaces and
  applications.
\newblock {\em arXiv preprint arXiv:2211.05894}, 2022.

\bibitem{MRT}
D.~M\'{a}rquez-Carreras, C.~Rovira, and S.~Tindel.
\newblock A model of continuous time polymer on the lattice.
\newblock {\em Commun. Stoch. Anal.}, 5(1):103--120, 2011.

\bibitem{murugan2023heat}
M.~Murugan.
\newblock Heat kernel for reflected diffusion and extension property on uniform
  domains, 2023.

\bibitem{Nualart}
D.~Nualart.
\newblock {\em The {M}alliavin calculus and related topics}.
\newblock Probability and its Applications (New York). Springer-Verlag, Berlin,
  second edition, 2006.

\bibitem{Pos12}
O.~Post.
\newblock {\em Spectral analysis on graph-like spaces}, volume 2039 of {\em
  Lecture Notes in Mathematics}.
\newblock Springer, Heidelberg, 2012.

\bibitem{RT}
C.~Rovira and S.~Tindel.
\newblock On the {B}rownian-directed polymer in a {G}aussian random
  environment.
\newblock {\em Journal of Functional Analysis}, 222(1):178--201, 2005.

\bibitem{Walsh}
J.~B. Walsh.
\newblock An introduction to stochastic partial differential equations.
\newblock In {\em \'{E}cole d'\'{e}t\'{e} de probabilit\'{e}s de
  {S}aint-{F}lour, {XIV}---1984}, volume 1180 of {\em Lecture Notes in Math.},
  pages 265--439. Springer, Berlin, 1986.

\end{thebibliography}

\small{
\noindent
\textbf{Fabrice Baudoin}: \url{fbaudoin@math.au.dk}\\
Department of Mathematics,
Aarhus University,
Denmark

\noindent \textbf{Li Chen}: \url{lichen@lsu.edu, lchen@math.au.dk}\\
Department of Mathematics, Louisiana State University, USA \&
Department of Mathematics,
Aarhus University, Denmark

\noindent \textbf{Che-Hung Huang}: \url{huan1160@purdue.edu} \\
Department of Mathematics, Purdue University, USA

\noindent \textbf{Cheng Ouyang}: \url{couyang@uic.edu}\\
Department of Mathematics, Statistics, and Computer Science, University of Illinois at Chicago, USA

\noindent
\textbf{Samy Tindel}:\url{stindel@purdue.edu} \\
Department of Mathematics, Purdue University, USA

\noindent
\textbf{Jing Wang}:\url{jingwang@purdue.edu} \\
Department of Mathematics, Purdue University, USA

}

\end{document}